\documentclass[a4paper,12pt,leqno]{article}
\usepackage{amsmath}
\usepackage{amsthm}
\usepackage{amssymb}
\usepackage{amscd}
\usepackage{graphicx, color}
\usepackage[all]{xy}
\newtheorem{thm}{Theorem}[section]
\newtheorem{dfn}[thm]{Definition}
\newtheorem{prp}[thm]{Proposition}%[section]
\newtheorem{lmm}[thm]{Lemma}%[section]
\newtheorem{crl}[thm]{Corollary}%[section]
\newtheorem{example}[thm]{Example}%[section]
\newtheorem{rmk}[thm]{Remark}%[section]
\numberwithin{equation}{section}

\usepackage[all]{xy}

\newcommand{\Hol}{\mbox{{\rm Hol}}}
\newcommand{\Alg}{\mbox{{\rm Alg}}}

\newcommand{\Pol}{\mbox{{\rm Pol}}}

\newcommand{\Z}{\Bbb Z}
\newcommand{\N}{\Bbb N}
\newcommand{\C}{\Bbb C}
\newcommand{\R}{\Bbb R}

\renewcommand{\P}{{\rm P}}
\newcommand{\SP}{\mbox{{\rm SP}}}

\newcommand{\RP}{\Bbb R\mbox{{\rm P}}}

\newcommand{\Map}{\mbox{{\rm Map}}}

\newcommand{\CP}{\Bbb C {\rm P}}

\newcommand{\dis}{\displaystyle}
\newcommand{\p}{\prime}

\newcommand{\I}{\mbox{{\rm (i)}}}
\newcommand{\II}{\mbox{{\rm (ii)}}}
\newcommand{\III}{\mbox{{\rm (iii)}}}
\newcommand{\IV}{\mbox{{\rm (iv)}}}
\newcommand{\V}{\mbox{{\rm (v)}}}

\newcommand{\XS}{X_{\Sigma}}

\newcommand{\GS}{G_{\Sigma}}

\newcommand{\dmin}{d_{\rm min}}
\newcommand{\rmin}{r_{\rm min}}
\newcommand{\KS}{\mathcal{K}_{\Sigma}}
\newcommand{\T}{\Bbb T}

\newcommand{\SZ}{{\mathcal{X}}^{D}}

\newcommand{\e}{\textbf{\textit{e}}}

\newcommand{\ES}{E^{\Sigma, \R}}
\title{\bf
The homotopy type of the
space of algebraic loops on a toric variety}
\author{Andrzej Kozlowski\footnote{%
Institute of Applied Mathematics and Mechanics,
University of Warsaw, Banacha 2, 02-097 Warsaw, Poland
(E-mail: akoz@mimuw.edu.pl)
}
\  and \ 
Kohhei Yamaguchi\footnote{%
Department of Mathematics,
University of Electro-Communications,  Chofu, Tokyo 182-8585, Japan
(E-mail: kohhe@im.uec.ac.jp);
The second author is supported by 
JSPS KAKENHI Grant Number 18K03295.
\newline
\quad 2010 {\it Mathematics Subject Clasification.} Primary 55P10; Secondly 55R80, 55P35, 14M25.}}

\date{}
\begin{document}
\maketitle
%%

%%%%(Abstract)%%%%%.
%\begin{abstract}
%%%%%%%%%%%%
%
%%%(Abstract)%%%%%%%%
\begin{abstract}
%Spaces of holomorphic maps 
%%
We  investigate the homotopy type of the space of tuples of polynomials inducing
%$\Alg^* (S^1,\XS)$ 
 base-point preserving algebraic
maps  from the circle $S^1$ to a toric variety $\XS$.
%represented by tuples of complex coefficients polynomials.
In particular,  we  prove 
 a homotopy stability result for this space
by combining the Vassiliev spectral sequence \cite{Va}
and the scanning map \cite{Se}.
\end{abstract}
%%%%%

%%%%(SECTION 1: Introduction)%%%%
%%%%%
\section{Introduction}\label{section 1}
%%%%%%%%%
%%%%
For two topological spaces $X$ and $Y$
with base-points, let $\Map^*(X,Y)$ 
denote the space of all continuous based maps
(i.e. base-point preserving maps)
$f:X\to Y$  with compact-open topology.
When these two spaces
have some additional structure,
e.g.  that of a complex or symplectic manifold or an algebraic variety,
it is natural to consider the subspace
$\mathcal{S}(X,Y)\subset \Map^*(X,Y)$  of all based maps $f$
which preserve this structure and to
ask whether the inclusion map
$i:\mathcal{S}(X,Y)\to \Map^*(X,Y)$ 
is a homotopy or homology equivalence up to some dimension.
%is in some sense
%(e.g. homotopy or homology) equivalent.
Early examples of this type of phenomenon 
can be found in \cite{Grom}.
In many cases of interest the infinite dimensional space  $\mathcal{S} (X,Y)$  
has a filtration by finite dimensional subspaces, 
given by some kind of  \lq\lq  degree of maps\rq\rq, 
and  the topology of these finite dimensional spaces 
approximates the topology of the entire space of continuous maps;
the approximation becoming more accurate as the degree increases. 
A great deal of attention has been 
devoted to  problems of the above kind in the case 
where the  \lq\lq additional structure\rq\rq\  
in question is the structure of a complex manifold 
(so that the structure preserving maps are holomorphic maps).
The first explicit result of this kind seems to have been the following theorem of Segal:

%%%(Theorem 1.1: Segal's Theorem)%%%
\begin{thm}
[G. Segal, \cite{Se}]\label{thm: S}
%%%%%%%%%%%
If $d\in \N$,
the inclusion map
$$
i_d:\Hol_d^*(S^2,\CP^{n})\to \Map_d^*(S^2,\CP^{n})=\Omega^2_d\CP^{n}
\simeq \Omega^2S^{2n+1}
$$
is a homotopy equivalence through dimension $(2n-1)d-1$.
%\par
Here,
we denote by $\N$  the space of all positive integers and
the space $\Hol^*_d(S^2,\CP^n)$ $($resp. $\Omega^2_d\CP^n)$
denotes the space consisting of all
based  holomorphic $($resp. based continuous$)$
maps $f:S^2\to \CP^n$ of degree $d$.\footnote{%
%%%(Footnote 1)%%
The stablization dimension $(2n-1)d-1$ obtained by Segal is not optimal for $n\ge 2$. The optimal dimension is $(2n-1)(d+1)-1$ (see \cite{KY6}).
}
%%(End of Footnote 1)%%%
\qed
%%%%%
\end{thm}
%%%%%%%%
%%(End of Theorem 1.1)%%
%%
%%
%%
%%(Remark 1.2)%%
\begin{rmk}
%%%%%%%%%%%
{\rm
Recall that a map $g:V\to W$ is called {\it a homology $($resp. homotopy$)$
equivalence through dimension }$N$ if the induced homomorphism
$
g_*:H_k(V;\Z)\to H_k(W;\Z)$
(resp.
$g_*:\pi_k(V)\to\pi_k(W))$
is an isomorphism for all $k\leq N$.
\qed
}
%%%%%%%%
\end{rmk}
%%(End of Remark 1.2)%%%
%%%%%%
%%%%%%%%%%%%
Segal conjectured that this theorem should have analogues for holomorphic maps from 
the Riemann sphere $S^2$ (or even any compact Riemann surface) to various complex manifolds $Y.$ He also suggested that some  analogues may exists for higher dimensional complex manifolds in place of $S^2$. All of Segal's conjectures have been shown to be true, although no fully satisfactory general explanation of these phenomena has been found. 
\par
Guest \cite{Gu2} was the first one to consider Segal's problem for the space of holomorphic maps from $S^2$ to a compact toric variety. His result was very general but the stability dimension he obtained was rather low. A much better stability bound was obtained by Mostovoy and 
Munguia-Villanueva \cite{MV}. Using methods developed by Vassiliev and Mostovoy \cite{Mo3}, they were able to obtain a Segal type theorem for the case of spaces of holomorphic maps from $\CP^n$ for $n\ge 1$ to complex compact toric varieties. Although their method works with arbitrary $\CP^n$ as source, for technical reasons the target space has to be compact. That leaves out some interesting toric varieties for which Segal's type theorems still hold (see for example \cite{KY6}). 
By replacing Mostovoy's approach (which relies on the Stone-Weierstrass Theorem) we were able to prove the Mostovoy and Munguia-Villanueva result for a larger class of target varieties but with $\CP^1=S^2$ as the source space and this is stated as follows.
%%%%%%
%%%%%%
%%(Theorem 1.3: The main Theorem of KY9)%%%%%
\begin{thm}[\cite{KY9}]\label{thm: KY9}
%%% 
Let $\XS$ be a simply connected non-singular toric variety associated to the fan $\Sigma$ such that 
the condition $($\ref{equ: homogenous}.1$)$ (see below)
is satisfied.\footnote{%
%%%(Footnote 2)%%
If $\XS$ is compact, the condition (\ref{equ: homogenous}.1)
is always satisfied (see Remark \ref{rmk: assumption}).
}
%%%(End of Footnote 2}
Then
if $D=(d_1,\cdots ,d_r)\in \N^r$ and $\sum_{k=1}^rd_k\textbf{\textit{n}}_k={\bf 0}_r$,
the inclusion map
$$
i_{D,hol}:\Hol_D^*(S^2,\XS) \stackrel{\subset}{\longrightarrow}
 \Omega^2_D\XS
$$ 
is a homotopy equivalence
through dimension $d_*(D,\Sigma)$ if $\rmin (\Sigma)\geq 3$
and a homology equivalence through dimension 
$d_*(D,\Sigma )=d_{min}-2$ if $\rmin (\Sigma)=2$.%\footnote{%
%%%%%(Footnote 3)%%%%%%
%See  (\ref{eq: D-hol}) in (iv) of Definition \ref{dfn: holomorphic}
%in detail.
%}
%%(End of Footnote 3)%%%%%%
\par
Here, 
$\Omega^2_D\XS$ (resp. $\Hol_D^*(S^2,\XS))$
denotes the space of based continuous (resp. based holomorphic) maps from 
$S^2$ to $\XS$ of
degree $D$,
the number $\rmin (\Sigma)$ is defined in (\ref{eq: rmin1}) and $d_*(D,\Sigma)$
is the number given by
%%%(1.1)%%
\begin{equation}
%%%%%%%%
d_*(D,\Sigma) =(2\rmin (\Sigma)-3)d_{min}-2,
\ \ 
\mbox{where }
d_{min}=\min \{d_1,\cdots ,d_r\}.
\qed
%%%%%
\end{equation}
%%%%
%where
%$d_{min}=\min \{d_1,\cdots ,d_r\}.$
%\qed
\end{thm}
%%%%%%%%(End of Theorem 1.3)%%

In his seminal work Segal observed that a real analogue of his theorem also holds.  
In fact, it was this real case that inspired the interest in the entire area.  
He considered the space of real rational maps, that is pairs of monic polynomials of the same degree with real coefficients and without common (complex) roots (see \cite[Proposition 1.4]{Se}). 
\par
Later in \cite{GKY2}, \cite{Mo1} and \cite{AKY1},  a different \lq real version\rq\ of Seqal's theory was introduced; it involves looking at tuples of real polynomials of the same degree  without common real roots.  One such case was considered in \cite{KOY1}, where the analogue of the Mostovoy and Munguia-Villanueva theorem was proved for real rational maps from $\RP^n$ to a  compact complex toric variety. The method used in it was the same as in \cite{MV}
and it  was limited only to case of compact toric varieties. 
By imposing a condition involving common {\it real} roots rather than complex ones, one can also study spaces of rational maps from $S^1$ to complex toric varieties. 
\par
An important difference between them
is that tuples of polynomials without common real roots no longer correspond to algebraic maps. For example, a pair $(p(z),q(z))$ of monic polynomials  of the same degree, without a common real root defines a unique rational (algebraic) map $S^1 \to \CP^1$ but by multiplying both polynomials by the same monic polynomial without real roots, we obtain the same rational map. Thus it is necessary to distinguish between spaces of tuples of polynomials and that of rational maps - their image in the space of continuous maps. 
\par
Now we can state the main result of this paper 
(precise definitions of all the terms used will be given in \S \ref{section: introduction2}).
%%%
%%%
%%%
%%(The main Theorem)%%%%%
%%%
%%(Theorem 1.4)%%%%%%%%%%%
\begin{thm}
[Theorem \ref{thm: I}, Theorem \ref{thm: II}]\label{thm: main result}
%%%%%%%%%%%%%%%%%%%%%%%%%%%%%%%%%%%%
Let $D=(d_1,\cdots ,d_r)\in \N^r$ be an $r$-tuple of positive integers 
and let $\XS$ be a simply connected non-singular toric variety associated to the fan $\Sigma$ 
such that 
the condition (\ref{equ: homogenous}.1) is satisfied.
%%%
Then the natural map 
$$
j_D:\Pol_D^*(S^1,\XS) \to \Omega U(\KS)\simeq
\Omega \mathcal{Z}_{\KS}
$$
is a homotopy equivalence through dimension
$d(D,\Sigma),$
where $d_{min}$ denotes the positive integer
$d_{min}=\min \{d_1,\cdots ,d_r\}$ and
the number $d(D,\Sigma)$ is given by
%%(1.2)%%%%%%%%%
\begin{equation}
%%%%%%%%
d(D,\Sigma)=2(\rmin  (\Sigma) -1)d_{\rm min}-2.
%%%%%%%%%%%%%
\end{equation}
%%%%%%%%%%%%
Here the space $U(\KS)$ (resp. $\mathcal{Z}_{\KS}$) is 
the complement $U(\KS)$ of the coordinate subspaces of type $\KS$
(resp. the moment-angle complex of the underlying
simplicial complex $\KS$) and
$\Pol_D^*(S^1,\XS)$ denotes the space of $r$-tuples $(p_1(z), \dots,p_r(z))\in \C [z]^r$ of monic polynomials of  degrees $d_1,\dots,d_r$
satisfying the condition (\ref{eq: Pd}.1).  The image of this space  under $j_D$ is 
contained in the space $\Alg^*(S^1,\XS)$ of 
based algebraic loops  on  $ \XS$.
%%%%%%%%
\end{thm}
%%%%%(End of Theorem 1.4)%%
%%%
%%%
%%%
%%%
%%%%%(Corollary 1.5)%%%%%%%%%
\begin{crl}[Corollary \ref{crl: I}, Corollary \ref{crl:  II}]
%%%%%%%%%%%%%%%%%%%%%%%%%%%%%
Under the same assumptions as Theorem \ref{thm: main result}, there is a natural map
$$
i_D:\Pol^*_D(S^1,\XS)\to \Omega \XS
$$
which induces an isomorphism
on homotopy groups
$$
(i_D)_*:\pi_k(\Pol^*_D(S^1,\XS))\stackrel{\cong}{\longrightarrow} 
\pi_k(\Omega \XS)\cong \pi_{k+1}(\XS)
$$
for any $2\leq k\leq d(D,\Sigma)$. 
\end{crl}
%%%%%%
%%%(End of Corollary 1.5)%%%%%
%%%%
%%%%
%%%%(Organization of the paper)%%%%%%
This paper is organized as follows.
In \S \ref{section: introduction2} we recall some basic facts about toric varieties and the definitions needed to give precise statements of our results. We also establish the notation that will be followed in the rest of the paper and then give full statements of all the main results
(Theorem \ref{thm: I},  Corollary \ref{crl: I}, 
 Theorem \ref{thm: II}, Corollary \ref{crl: II}, 
 Corollary \ref{crl: I-2}, Corollary \ref{crl: II-2}),
%%%
which will be proved in \S \ref{section: proofs}.
We also consider some example (Example \ref{example: H(k)}).
In \S \ref{section: toric variety}
we recall several definitions and known results from toric topology.
% give the proof of Lemma \ref{lmm: Map D}, which is stated in \S \ref{section: introduction2}
% without proof.
In \S \ref{section: simplicial resolution},
we recall the definitions of the non-degenerate simplicial resolution
and the associated truncated simplicial resolutions.
In \S \ref{section: spectral sequence}, we construct the Vassiliev spectral sequence and
compute its $E_1$-terms.
In  \S \ref{section: stabilization map}
we define stabilization maps and prove the homotopy stability 
(Theorem \ref{thm: III}, Theorem \ref{thm: IV}).
Moreover, we give the definition of the map
$j_D$ when $\sum_{k=1}^rd_k\textbf{\textit{n}}_k\not= {\bf 0}_n$.
In \S \ref{section: scanning map} we recall the stable horizontal scanning map and prove that
it is a homotopy equivalence (Theorem \ref{thm: scanning map}). In \S \ref{section: stability}
we give the proof of stability result (Theorem \ref{thm: V}) by using the stable scanning maps.
Finally in \S \ref{section: proofs} we give the proofs of the main results of this paper.
%%%
%%%(End of SECTION 1)%%%
%%
%%
%%(SECTION 2)%%
\section{Basic definitions and the main results}\label{section: introduction2}
%%%%%%%%%%%%%

In this section we shall recall several basic definitions and facts 
for describing the main results. 
%%%
%%%%%%%%%%%%%%%%%%%%%%%%%%
\paragraph{Fans and toric varieties}
%%%%%%%%%%%%%%%%%%%%%%%%%%
{\it A convex rational polyhedral cone} $\sigma$ 
in $\R^n$
is a subset of $\R^n$ of the form
%%%(2.1)%%
\begin{align}
%%%%%%
\sigma &=\mbox{Cone}(S)
=
\mbox{Cone}(\textit{\textbf{m}}_1,\cdots,\textit{\textbf{m}}_s)
%\\
%\nonumber
%&=
=
\Big\{\sum_{k=1}^s\lambda_k\textit{\textbf{m}}_k:\lambda_k\geq 0
\mbox{ for any }k\Big\}
\end{align}
%%%
for a finite set $S=\{\textit{\textbf{m}}_k\}_{k=1}^s\subset\Z^n.$\footnote{%
%%%%%%%%%%%%%%%%%%%%%%%%
%%%(Footnote 4)%%%%%%%%%
%\footnote{%
When $S$ is the emptyset $\emptyset$,
we set $\mbox{Cone}(\emptyset)=\{{\bf 0}_n\}$ and we may also regard it
as one of strongly convex rational polyhedral cones in
$\R^n$, where we denote by ${\bf 0}_n$ the zero vector in $\R^n$ defined by
${\bf 0}_n=(0,\cdots ,0)\in \R^n$.
}
%%%(End of Footnote 4)%%%%%
%\par
A convex rational polyhedral cone $\sigma$
is called
{\it  strongly convex} if
$\sigma \cap (-\sigma)=\{{\bf 0}_n\}$, and
its dimension $\dim \sigma$ is the dimension of
the smallest subspace in $\R^n$ which contains $\sigma$.
%%%
{\it A face} $\tau$ of 
$\sigma$ is a subset $\tau\subset \sigma$ of the form
%%(2.2)%%%
\begin{equation}
\tau =\sigma\cap \{\textit{\textbf{x}}\in\R^n:L(\textit{\textbf{x}})=0\}
\end{equation}
%%%%%%
for some linear form $L$ on $\R^n$, such that
$L(\textit{\textbf{x}})\geq 0$ for any $\textit{\textbf{x}}\in \sigma$.
%%%%
If
$\{k:L(\textit{\textbf{m}}_k)=0,1\leq k\leq s\}=\{i_1,\cdots ,i_t\},$
we easily see that $\tau =\mbox{Cone}(\textit{\textbf{m}}_{i_1},
\cdots , \textit{\textbf{m}}_{i_t})$.
Thus, a face $\tau$ is also a 
strongly convex rational polyhedral cone if $\sigma$ is so.
\par
A finite collection $\Sigma$ of strongly convex rational polyhedral cones
in $\R^n$ 
is called {\it a fan} in $\R^n$ if 
every face $\tau$ of $\sigma\in\Sigma$ belongs to $\Sigma$ and
the intersection of any two elements of $\Sigma$ is a face of each.
%%%%%
\par
An $n$ dimensional irreducible normal  variety
$X$ (over $\C$) is called {\it a toric variety}
if it has a Zariski open subset
 $\T^n_{\C}=(\C^*)^n$ and the action of $\T^n_{\C}$ on itself
extends to an action of $\T^n_{\C}$ on $X$.
The most significant property of a toric variety
is the fact that it is characterized up to isomorphism entirely by its 
associated fan 
$\Sigma$. 
We denote by $\XS$ the toric variety associated to a fan $\Sigma$
(see \cite{CLS} for the details).
\par
Since the fan of $\T^n_{\C}$ is $\{{\bf 0}_n\}$ and
the case $\XS = \T^n_{\C}$ is trivial, 
we always assume that 
any fan $\Sigma$ in $\R^n$
satisfies the condition
$\{{\bf 0}_n\}\subsetneqq \Sigma .$
%%%%
%%%%%
%%%
%%(Definition 2.1)%%
\begin{dfn}\label{dfn: fan}
%%%%%%
{\rm
Let $\Sigma$ be a fan in $\R^n$
such that $\{{\bf 0}_n\}\subsetneqq \Sigma$ and let
%%%%%%%%
%%%(2.3)%%
\begin{equation}\label{eq: one dim cone}
%%%%%%%
\Sigma (1)=\{\rho_1,\cdots ,\rho_r\}
%%%%%%%%
\end{equation}
%%%%% 
denote the set of all
one dimensional cones in $\Sigma$. 
%for some positive integer $r$.
%Note that $r\geq 1$ is a positive integer under the condition (\ref{eq: assumption0}).
%%
For each integer $1\leq k\leq r$,
we denote by $\textbf{\textit{n}}_k\in\Z^n$ 
\textit{the primitive generator} of
$\rho_k$, such that 
%%(2.4)%%
\begin{equation}\label{eq: primitive}
%%%
\rho_k \cap \Z^n=\Z_{\geq 0}\cdot \textbf{\textit{n}}_k.
\end{equation}
%%%%%%%%
%
Note that
$\rho_k =\mbox{Cone}(\textit{\textbf{n}}_k)=\R_{\geq 0}\cdot \textit{\textbf{n}}_k$
for each $1\leq k\leq r$.
\qed
}
\end{dfn}
%%%(End of Definition 2.1)%%%%

%%%%%%%%%%%
\paragraph{Polyhedral products}
%%%%%%%%%%%%%%%%%%%
Next, recall the definition of polyhedral products.
%%%(Definition 2.2)%%%
\begin{dfn}
%%%%%%%%%%%%%%%%%%%%%%
{\rm
Let $K$ be a simplicial complex on the vertex set $[r]=\{1,2,\cdots ,r\}$,%
%%%%%%%%%%%%%%%%%%
%%%(Footnote 4)%%%%%%
\footnote{%
Let $K$ be some set of subsets of $[r]$.
Then the set $K$ is called {\it an abstract simplicial complex} on the vertex set $[r]$ if
the following condition holds:
 if $\tau \subset \sigma$ and $\sigma\in K$, then $\tau\in K$.
In this paper by a simplicial complex $K$ we always mean an  \textit{an abstract simplicial complex}, 
and we always assume that a simplicial complex $K$  contains the empty set 
$\emptyset$.
}
%%%%%%%%%%%
%%%(End of FootNote 4)%%%%%%%%%
and let $(\underline{X},\underline{A})=\{(X_1,A_1),\cdots ,(X_r,A_r)\}$ be a set of pairs of based spaces
such that $A_i\subset X_i$ for each $1\leq i\leq r$.
%%%
\par
%%%%
(i)
 Let $\mathcal{Z}_K
(\underline{X},\underline{A})$ 
denote {\it the polyhedral product} of $(\underline{X},\underline{A})$
with respect to $K$
given by the union
%%%%%%%%%%%
%%(2.5)%%%
\begin{align}
%%%%%
\mathcal{Z}_K(\underline{X},\underline{A})
&=
\bigcup_{\sigma\in K}(\underline{X},\underline{A})^{\sigma},
\quad \mbox{where we set}
%%%
%\\
%\nonumber
%\mbox{where we set}&
\\
%%%()%%%%
\nonumber
(\underline{X},\underline{A})^{\sigma}
&=
\{(x_1,\cdots ,x_r)\in X_1\times \cdots \times X_r:
x_k\in A_k\mbox{ if }k\notin \sigma\}.
\end{align}
When $(X_i,A_i)=(X,A)$ for each $1\leq i\leq r$, we write
%%%%%
$\mathcal{Z}_K(X,A)=\mathcal{Z}_K(\underline{X},\underline{A}).$
%%%%%%%
%%%%
\par
(ii)
For a subset $\sigma %=\{i_1,\cdots ,i_s\}
\subset [r]$, let
$L_{\sigma}$ denote {\it the coordinate subspace} of $\C^r$ defined by
%%(2.6)%%
\begin{equation}
%%%%%%%
L_{\sigma}=\{(x_1,\cdots ,x_r)\in\C^r:
x_{i}=0\mbox{ if }i\in\sigma\}.
\end{equation}
%%%%%%
Define {\it the complement $U(K)$ of the coordinate subspaces of type}  $K$
by
%%%(2.7)%%%
\begin{align}\label{IK}
%%%
%\nonumber
U(K)&=
\C^r\setminus \bigcup_{\sigma\in I(K)}L_{\sigma}
%=\C^r\setminus \tilde{L}(K),
\quad
\mbox{where we set }
%\\
%%%%()%%%
I(K)=\{\sigma\subset [r]:\sigma\notin K\}.
%%%%%
\end{align}
%%%%%%
%%%%%
\par
%%%(iii)%%%%
(iii)
%%%%%%%%%%%
For a fan $\Sigma$ in $\R^n$,
let $\mathcal{K}_{\Sigma}$ denote {\it the underlying simplicial complex of}  $\Sigma$
defined by
%%%%%
%%%(2.8)%%%%%%
\begin{equation}
%%%%%%%%%%%%%
\KS =
\Big\{\{i_1,\cdots ,i_s\}\subset [r]:
\mbox{Cone}(\textbf{\textit{n}}_{i_1},\textbf{\textit{n}}_{i_2},\cdots
,\textbf{\textit{n}}_{i_s})\in \Sigma
%\mbox{ span a cone in }\Sigma
\Big\}.%\cup \{\emptyset\}.
\end{equation}
%%%%%%
Note that $\KS$ is a simplicial complex on the vertex set $[r]$.
\qed
}
\end{dfn}
%%%(end of Definition 2.2)%%%%%
%%%%%
%%%(Remark 2.3)%%%
\begin{rmk}\label{rmk: fan}
%%%%%%%
%%%%%%%%%%%%%%%%%
{\rm
(i)
It is easy to see that the following equality holds:
%%(2.9)%%
\begin{equation}\label{equ: UK}
%%%%%%%%%%
U(K)=\mathcal{Z}_K(\C,\C^*).
\end{equation}
%%%%%%
\par
(ii)
The fan $\Sigma$ is completely determined by
the pair
$(\mathcal{K}_{\Sigma},\{\textit{\textbf{n}}_k\}_{k=1}^r).$
\par
Indeed, if we set $C(\sigma)=\mbox{Cone}
(\textit{\textbf{n}}_{i_1},\cdots ,\textit{\textbf{n}}_{i_s})$
for $\sigma =\{i_1,\cdots ,i_s\}\subset [r]$
and $C(\emptyset )=\{{\bf 0}_n\}$,
%%%%%%%%%
then it is easy to see that
%\begin{equation}
$\Sigma =
\{\mbox{C}(\sigma):\sigma \in \mathcal{K}_{\Sigma}\}.$
%%%
\qed
%%%%
}
%%%%%%%
\end{rmk}
%%%%%%%(End of Remark 2.3)%%%%%%%%

%%%(Homogenous coordinates)%%%%%
\paragraph{Homogenous coordinates}
%%%%%%%%%%
Recall %the basic facts concerning to 
the homogenous coordinates on
toric varieties.
Let $\Sigma$ be a fan in $\R^n$ as in Definition \ref{dfn: fan}.

%%%(Definition 2.4)%%
\begin{dfn}
%%%%%%%
{\rm
%%%
\par
(i)
Let $\GS\subset \T^r_{\C}=(\C^*)^r$ denote the multiplicative subgroup of
$\T^r_{\C}$
defined by
%%%%%%%%
%%(2.10)%%
\begin{equation}
%%%%%%%%
\GS =\{
(\mu_1,\cdots ,\mu_r)\in \T^r_{\C}:
\prod_{k=1}^r(\mu_k)^{\langle \textbf{\textit{n}}_k,\textbf{\textit{m}}
\rangle}=1
\mbox{ for all }\textbf{\textit{m}}\in\Z^n\},
%%%%%%%%%%%%
\end{equation}
%%%
where
$\langle \ ,\ \rangle$ denotes the standard inner product on $\R^n$ given by
$\langle \textbf{\textit{u}}, \textbf{\textit{v}}\rangle=
\sum_{k=1}^nu_kv_k$ for
$\textbf{\textit{u}}=(u_1,\cdots ,u_n)$
and  $\textbf{\textit{v}}=(v_1,\cdots ,v_n)\in\R^n$.
%%%%
\par
(ii)
Consider the natural $\GS$-action on
$\mathcal{Z}_{\mathcal{K}_{\Sigma}}(\C,\C^*)$ given by 
coordinate-wise multiplication, i.e.
%%(2.11)%%
\begin{equation}\label{eq: multiplication}
%%%%%%%%%%%%%%
{\bf \mu}\cdot \textit{\textbf{x}}=
(\mu_1x_1,\cdots ,\mu_rx_r)
%%%%%%%
\end{equation}
%%%%
for
$({\bf \mu}, \textit{\textbf{x}})=
((\mu_1,\cdots ,\mu_r),(x_1,\cdots ,x_r))\in \GS \times 
\mathcal{Z}_{\KS}(\C,\C^*)$.
We denote by
%%(2.12)%%
\begin{equation}
\mathcal{Z}_{\mathcal{K}_{\Sigma}}(\C,\C^*)/\GS =U(\KS)/\GS
\end{equation}
%%%%%
the corresponding orbit space and let
%%(2.13)%%
\begin{equation}
%%%
q_{\Sigma}:
\mathcal{Z}_{\mathcal{K}_{\Sigma}}(\C,\C^*)
\to
\mathcal{Z}_{\mathcal{K}_{\Sigma}}(\C,\C^*)/\GS =U(\KS)/\GS
%%%%%%
\end{equation}
%%%
denote the canonical projection.
\qed
%%%%
%%%%%
%%
}
\end{dfn}
%%(End of Definition 2.4)%%%%
%%%%%%%%%%%%%%%%%
 The following theorem, which plays a crucial role in the proof of our main result, states 
 that in a toric variety, under certain mild conditions, one can construct \lq\lq homogeneous coordinates\rq\rq\  similar to those in a projective space. 
%%
%%
%%
%%%(Theorem 2.5: Theorem of Cox)%%
\begin{thm}
[\cite{Cox1}, Theorem 2.1] %\cite{Cox2}, Theorem 3.1
\label{lmm: Cox}
%%%%%%%%%
If the set $\{\textit{\textbf{n}}_k\}_{k=1}^r$ 
of all primitive generators  spans $\R^n$
$($i.e. $\sum_{k=1}^r\R\cdot \textit{\textbf{n}}_k =\R^n),$
there is a natural isomorphism
%%(2.14)%%
\begin{equation}\label{equ: homogenous}
\XS\cong
\mathcal{Z}_{\mathcal{K}_{\Sigma}}(\C,\C^*)/\GS =U(\KS)/\GS . 
\qquad \qed
\end{equation}
%%%%
%%%
\end{thm}
%%%%%%(End of THeorem 2.5)%%%%%
By using the above result we can obtain the following result
whose proof is postponed in the last part of this section.

%%(Lemma 2.6)%%
\begin{lmm}[cf. \cite{Cox2}, Theorem 3.1]\label{lmm: Map D}
%%%
Suppose that
the set $\{\textit{\textbf{n}}_k\}_{k=1}^r$ 
of all primitive generators  spans $\R^n$, and 
let $f_k\in \C [z_0,\cdots ,z_m]$ be a homogenous polynomial of
the degree $d_k^*$ for each $1\leq k\leq r$ such that
the polynomials $\{f_k\}_{k\in \sigma}$ have no common real root
except ${\bf 0}_{m+1}\in \R^{m+1}$ for each $\sigma\in I(\KS)$.
Then 
there is a unique map
$f:\RP^m\to \XS$ such that the following diagram
%%()%%
\begin{equation*}\label{eq: homogenous-CD}
%%%%
\begin{CD}
\R^{m+1}\setminus \{{\bf 0}\} @>(f_1,\cdots ,f_r)>> U(\mathcal{K}_{\Sigma})
=\mathcal{Z}_{\KS}(\C,\C^*)
\\
@V{\gamma_m}VV @V{q_{\Sigma}}VV
\\
\RP^m @>f>> U(\mathcal{K}_{\Sigma})/\GS =\XS
\end{CD}
\end{equation*}
%%%%%%%%
is commutative
if and only if $\sum_{k=1}^rd_k^*\textit{\textbf{n}}_k={\bf 0}_n$,
where $\gamma_m:\R^{m+1}\setminus \{{\bf 0}\}\to \RP^m$ denotes
the canonical double covering and the map $q_{\Sigma}$ is the canonical projection
induced from the identification $($\ref{equ: homogenous}$)$.
\end{lmm}
%%%(End of Lemma 2.6)%%
%%
%%
%%
%%(Remark 2.7)%%%
\begin{rmk}
%%%%%%
{\rm
We call the map $f$ determined by an $r$-tuple $(f_1,\cdots ,f_r)$
of homogenous polynomials as \textit{an algebraic map}  
and we write $f=[f_1,\cdots ,f_r]$.
Note that two different such $r$-tuples of polynomials can determine the same maps. 
%In fact, if we multiply all polynomials in such an $r$-tuple by the same polynomial which does not have any real roots except ${\bf 0}_m$, we obtain the same algebraic map.
For example, suppose that  $(f_1,\cdots ,f_r)$
is the $r$-tuple of homogenous polynomials in $\C[z_0,\cdots ,z_m]$
of degree $d_1^*,\cdots ,d_r^*$ 
satisfying the condition
given in Lemma \ref{lmm: Map D}. 
Then if $(a_1,\cdots ,a_r)\in \N^r$ is the $r$-tuple
of positive integers and it satisfies the condition
$\sum_{k=1}^rd_k^*\textbf{\textit{n}}_k=
\sum_{k=1}^ra_k\textbf{\textit{n}}_k={\bf 0}_n,$
we can easily see that
$
f=[f_1,\cdots ,f_r]=[g^{a_1}f_1,\cdots ,g^{a_r}f_r]
=[h^{a_1}f_1,\cdots ,h^{a_r}f_r]
$
for
$g=\sum_{k=0}^mz_k^2$
and $h=(z_0+z_1)^2+\sum_{k=2}^mz_k^2$.
\qed
%%%%
}
%%%%%%%%
\end{rmk}
%%(End of Remark 2.7)%%%%
%%
\par\vspace{1mm}\par
%%
%%%%(Assumptions)%%%%%%%
\paragraph{Assumptions}
%%%%%%
Let $\Sigma$ be a fan in $\R^n$
satisfying the condition (\ref{eq: one dim cone})
as in Definition \ref{dfn: fan}.
From now on,
we assume  that the following two conditions hold.
%%%%(Assumptions)%%%%
\begin{enumerate}
%%(2.15.1)%%%%)%%
\item[(\ref{equ: homogenous}.1)]
%%%%
There is an $r$-tuple $D_*=(d_1^*,\cdots ,d_r^*)\in \N^r$ 
of positive integers 
such that
$\sum_{k=1}^rd_k^*\textit{\textbf{n}}_k={\bf 0}_n.$
%%%(2.15.2)%%%
\item[(\ref{equ: homogenous}.2)]
%%%%%%%%%%%%%%%
The set
$\{\textbf{\textit{n}}_k\}_{k=1}^r$ of primitive generators spans $\Z^n$
over $\Z$.
%\newline
%$\sum_{k=1}^r\Z\cdot \textit{\textbf{n}}_k=\Z^n$.
\end{enumerate}
%%%%%

%%%(Remark 2.8)%%%
\begin{rmk}\label{rmk: assumption}
%%%%%%%%%%%%%%%%%%
{\rm
(i)
Note that $\XS$ is a compact iff $\bigcup_{\sigma\in \Sigma}\sigma =\R^n$ 
\cite[Theorem 3.4.1]{CLS}.
Note also that
$\XS$ is simply connected if and only if $\sum_{k=1}^r\Z\cdot \textit{\textbf{n}}_k=\Z^n$
(see Lemma \ref{lmm: XS} below).
Hence, the condition (\ref{equ: homogenous}.2) always holds if
$\XS$ is compact or simply connected.
On the other hand, if
the condition (\ref{equ: homogenous}.2) holds,
one can easily see that  the set $\{\textit{\textbf{n}}_k\}_{k=1}^r$ spans
$\R^n$ over $\R$, and
there is an  isomorphism (\ref{equ: homogenous}) for the space $\XS$.
\par
(ii)
We know that the condition (\ref{equ: homogenous}.1) holds
if $\XS$ is compact and non-singular \cite[Theorem 3.1]{Cox2}.
%%%%%
\par
(iii)
Let $\Sigma$ denote the fan in $\R^2$ given by
$\Sigma =\{\{{\bf 0}_2\},\mbox{Cone}(\textit{\textbf{n}}_1),
\mbox{Cone}(\textit{\textbf{n}}_2)\}$
for the standard basis $\textit{\textbf{n}}_1=\textit{\textbf{e}}_1=(1,0),$
$\textit{\textbf{n}}_2=\textit{\textbf{e}}_2=(0,1)$.
Then the toric variety $\XS$ of  $\Sigma$ is
$\C^2$ which has trivial homogenous coordinates.
It is clearly a (simply connected) smooth toric variety, and
the condition (\ref{equ: homogenous}.1) also  holds.
%\par
However, in this case,
$\sum_{k=1}^2d_k^*\textit{\textbf{n}}_k={\bf 0}_2$ iff
$(d_1^*,d_2^*)=(0,0)$. 
Hence, it follows from Lemma \ref{lmm: Map D} that there are no algebraic maps
$\RP^m\to \XS=\C^2$ other than the constant maps.
Assuming the condition
(\ref{equ: homogenous}.1) guarantees the existence of non-trivial
algebraic maps $\RP^m\to \XS$.
Of course, it would be sufficient to assume that 
$D=(d_1,\dots ,d_r)\not= (0,\dots 0)$ but if $d_i=0$ for some $i$, then the number $d(D,\Sigma)$ (defined in (\ref{eq: rmin})) is not a positive integer
and our assertion (Theorem \ref{thm: I} below) is vacuous.
For this reason, we will assume the condition $d_k^*\geq 1$ for each $1\leq k\leq r$ in
(\ref{equ: homogenous}.1).
\qed
}
\end{rmk}
%%(End of Remark 2.8)%%
%%%%%
\paragraph{Spaces of tuples of polynomials which define algebraic maps}
%%%%%%%%
Let $\XS$ be a non-singular toric variety and
 make the identification $\XS =U(\KS)/\GS$.
 Let
 $z_0,\cdots ,z_m$ be variables.
\par
Now we
consider the space of all tuples of polynomials which define  
based algebraic maps.
%%%
%%(Definition 2.9)%%
\begin{dfn}
%%%%%%%
{\rm
(i)
For each $d,m\in \N$,  let $\mathcal{H}_m^d(\C)$ denote the space of all
homogenous polynomials $f(z_0,\cdots ,z_m)\in \C[z_0,\cdots ,z_m]$
of degree $d$.
\par
(ii)
For each $r$-tuple $D=(d_1,\cdots ,d_r)\in\N^r$, let
$\Pol_D^*(\RP^m,\XS)$ denote the space of
$r$-tuples
$$
f=(f_1(z_0,\cdots ,z_m),\cdots ,f_r(z_0,\cdots ,z_m))\in
\mathcal{H}^{d_1}_m(\C)\times \cdots \times \mathcal{H}^{d_r}_m(\C)
$$
of homogenous polynomials satisfying the following two conditions:
%%%
\begin{enumerate}
%%%(2.16.1)%%%
\item[(\ref{eq: alg}.1)]
$f(\textbf{\textit{x}})=(f_1(\textbf{\textit{x}}),\cdots ,f_r(\textbf{\textit{x}}))
\in U(\KS)$
for any point $\textbf{\textit{x}}=(x_0,\cdots ,x_m)\in \R^{m+1}\setminus
\{{\bf 0}_{m+1}\}$.
%%%(2.16.2)%%%
\item[(\ref{eq: alg}.2)]
$f(\textbf{\textit{e}}_1)=(f_1(\textbf{\textit{e}}_1),\cdots ,f_r(\textbf{\textit{e}}_1))
=(1,1,\cdots ,1)$, i.e.
the coefficient of $(z_0)^{d_k}$ in $f_k(z_0,\cdots ,z_m)$ is $1$ for each $1\leq k\leq r$,
where we write $\textbf{\textit{e}}_1=(1,0,\cdots ,0)\in \R^{m+1}$.
\qed
\end{enumerate}
%%%%%%%%
}
\end{dfn}
%%(End of Definition 2.9)%%
%%
%%
%%(Definition 2.10)%%
\begin{dfn}
%%%%%
{\rm
We always assume the 
identification
$\XS=U(\KS)/\GS$,  denote by $[y_1,\cdots ,y_r]$ the point in $\XS$ represented by
$(y_1,\cdots ,y_r)\in U(\KS),$
and  choose
the two points $[1:0:\cdots :0]\in \RP^m$ and $*=[1,\cdots ,1]\in \XS$
as the base-points of $\RP^m$ and $\XS$ respectively.
%%%%
\par\vspace{1mm}\par
Let  $D=(d_1,\cdots ,d_r)\in \N^r$ be an $r$-tuple of positive integers such that
$\sum_{k=1}^rd_k\textbf{\textit{n}}_k={\bf 0}_n$. 
Then
by using Lemma \ref{lmm: Map D}, for each $r$-tuple 
$$
f=(f_1(z_0,\cdots ,z_m),\cdots ,f_r(z_0,\cdots ,z_m))\in
\Pol_D^*(\RP^m,\XS)
$$ one can define 
based algebraic map
%%(2.15)%%
\begin{eqnarray}\label{eq: alg}
%%%%%%
[f]&=&[f_1,\cdots ,f_r]:(\RP^m,[\textbf{\textit{e}}_1])\to (\XS, *)
\quad
\mbox{ by }
\end{eqnarray}
%%(2.16)%%
%%%
\begin{eqnarray}
[f]([\textbf{\textit{x}}])
&=&
[f_1(\textbf{\textit{x}}),\cdots ,f_r(\textbf{\textit{x}})]
\end{eqnarray}
%%%%%
for $[\textbf{\textit{x}}]=[x_0:\cdots :x_m]\in \RP^m$,
where
$\textbf{\textit{x}}=(x_0,\cdots ,x_m)
\in \R^{m+1}\setminus \{{\bf 0}_{m+1}\}$.
Hence, we obtain the natural map
%%(2.17)%%
\begin{equation}
%%%%%%
i_{D,m}:\Pol_D^*(\RP^m,\XS) \to \Map^*_D(\RP^m,\XS)
\quad
\mbox{ given by}
%%%
\end{equation}
%%%% 
%%%(2.18)%%
\begin{equation}
i_{D,m}(f)=[f]=[f_1,\cdots ,f_r]
\end{equation}
%%%%
 for
$f=(f_1(z_0,\cdots ,z_m),\cdots ,f_r(z_0,\cdots ,z_m))\in
\Pol_D^*(\RP^m,\XS)$, 
where we denote by
$\Map_D^*(\RP^m,\XS)$ the
path-component of $\Map^*(\RP^m,\XS)$ 
which contains all algebraic maps of degree $D$.
\qed
}
\end{dfn}
%%%(End of Definition 2.10)%%%%%%

%%%(Remark 2.11)%%%%
\begin{rmk}
%%%%%
{\rm
When $m=1$,  we  make the identification $\RP^1=S^1=\R \cup \infty$ and choose the points $\infty$  as the base-point of
$\RP^1$.
Then, by setting $z=\frac{z_0}{z_1}$, 
we can view a homogenous polynomial $f(z_0,z_1)\in \C [z_0,z_1]$ of degree $d$
as a monic polynomial  $f_k(z)\in\C [z]$
of degree $d$.
}
\end{rmk}
%%(End of Remark 2.10)%%%%

Thus, when $m=1$, one can redefine the space 
$\Pol_D^*(S^1,\XS)$ as follows.
%%%%(Definition 2.12)%%%%%%
\begin{dfn}\label{dfn: holomorphic}
%%%%%%%%%%%%%%%%%%%%%%%%%%%
{\rm
(i)
Let $\P^d(\C)$ denote the space of all monic polynomials 
$f(z)=z^d+a_1z^{d-1}+\cdots +a_{d-1}z+a_d\in \C [z]$ of degree
$d$, and let
%%(2.19)%%%
\begin{equation}\label{eq: Pd}
%%%%%%%%
\P^D=\P^{d_1}(\C)\times \P^{d_2}(\C)\times\cdots \times \P^{d_r}(\C).
\end{equation}
%%%
Note that there is a homeomorphism
$\phi:\P^d(\C)\cong \C^d$ given by
$\phi (z^d+\sum_{k=1}^d a_kz^{d-k})
=(a_1,\cdots ,a_d)\in\C^d.$
\par
%%%(ii)%%%
(ii)
For any $r$-tuple $D=(d_1,\cdots ,d_r)\in \N^r$, 
%satisfying the condition (\ref{equ: homogenous}.1),
%%% 
let $\Pol_D^*(S^1,\XS)$ denote the
space of all $r$-tuples
$(f_1(z),\cdots ,f_r(z))\in \P^D$ of monic polynomials
satisfying the following condition $(\dagger)$:
%%%%%
\begin{enumerate}
%%()%%
\item[%(\ref{eq: Pd}.1)
$(\dagger)$]
%%%%%%%
The polynomials
$f_{i_1}(z),\cdots ,f_{i_s}(z)$ have no common {\it real} root
for any $\sigma =\{i_1,\cdots ,i_s\}\in I(\KS)$,
i.e.
$(f_{i_1}(\alpha),\cdots ,f_{i_s}(\alpha))\not= {\bf 0}_s$
for any $\alpha \in  \R$.
\end{enumerate}
%%%%%
When the condition $\sum_{k=1}^rd_k\textit{\textbf{n}}_k={\bf 0}_n$ holds,
by identifying 
$\XS =U(\KS)/\GS$ and $\RP^1=S^1=\R\cup\infty$,
one can define a natural map
%%%%(2.20)%%%
\begin{equation}\label{eq: iD-original}
%%%%%%%%%%%%%%%%
i_D=i_{D,1}:\Pol_D^*(S^1,\XS)\to \Map^*(S^1,\XS)=\Omega \XS
\quad
\mbox{ by}
\end{equation}
%%%
%%(2.21)%%
\begin{equation}
%%%%%%%%%
i_D(f_1(z),\cdots ,f_r(z))(\alpha )=
\begin{cases}
[f_1(\alpha),\cdots ,f_r(\alpha)] & \mbox{ if }\alpha \in\R
\\
[1,1,\cdots ,1] & \mbox{ if }\alpha =\infty
\end{cases}
\end{equation}
%%%%%%%%%%
for $(f_1(z),\cdots ,f_r(z))\in \Pol_D^*(S^1,\XS)$ and $\alpha \in S^1=\R\cup \infty$,
where
we choose the points $\infty$ and $[1,1,\cdots ,1]$
as the base-points of $S^1$ and $\XS$.
\par
%%(iii)%%
\par\vspace{1mm}\par
%%%
Note that $\Pol_D^*(S^1,\XS)$ is simply connected 
(which will be proved in Proposition \ref{prp: simply connected})
and
that the map $\Omega q_{\Sigma}:
\Omega U(\KS)
\to
\Omega \XS$ is a universal covering
(which will be shown in Lemma \ref{lmm: XS}).
%and Proposition \ref{prp: 1-coonected}
%(cf. Definition \ref{def: moment-angle} and Lemma \ref{Lemma: BP})).
%\par
Thus, when $\sum_{k=1}^rd_k\textit{\textbf{n}}_k={\bf 0}_n$,
the map $i_D$ lifts to the space $\Omega \mathcal{Z}_{\KS}$
and there is a map
%%(2.22)%%
\begin{equation}\label{eq:jD-original}
%%%%%%
j_D:\Pol_D^*(S^1,\XS)\to 
\Omega U(\KS)\simeq
\Omega \mathcal{Z}_{\KS}
%%%
\end{equation}
%%%%%
such that
%%(2.23)%%
\begin{equation}\label{eq: lift}
%%%
\Omega q_{\Sigma}\circ j_D=i_D.
%\qquad
%\mbox{(up to homotopy).}
%%%
\end{equation}
%%%%%%%%
%\par
%(iii) For an $r$-tuple
%$D=(d_1,\cdots ,d_r)\in \N^r$ of positive integers satisfying 
%the condition $\sum_{k=1}^rd_k\textit{\textbf{n}}_k={\bf 0}_n$, 
%let $\Hol_D^*(S^2,\XS)$ denote the space
%of $r$-tuples $(f_1(z),\cdots ,f_r(z))\in \P^D$
%of monic polynomials
%satisfying the condition
%%%%%%
%\begin{enumerate}
%%%(2.24.1)%%
%\item[(\ref{eq: lift}.1)]
%%%%%%%%
%the polynomials
%$f_{i_1}(z),\cdots ,f_{i_s}(z)$ have no common  root
%for any $\sigma =\{i_1,\cdots ,i_s\}\in I(\KS)$,
%i.e.
%$(f_{i_1}(\alpha),\cdots ,f_{i_s}(\alpha))\not= {\bf 0}_s$
%for any $\alpha \in  \C$.
%\qed
%\end{enumerate}
%%%%%%
%\par
%(iv)
%If the condition
%$\sum_{k=1}^rd_k\textit{\textbf{n}}_k={\bf 0}_n$
%is satisfied, then
%$\Hol_D^*(S^2,\XS)\subset \Pol^*_D(S^1,\XS)$ and
%one can define the map
%%%(2.25)%%
%\begin{equation}\label{eq: D-hol}
%%%%%%%
%i_{D,hol}:\Hol_D^*(S^2,\XS)\to 
%\Omega^2_D \XS
%\qquad
%\mbox{by}
%%%%
%\end{equation}
%%%%%%
%%%(2.22)%%
%\begin{equation}
%%%%%%%%%%
%i_{D,hol}(f_1(z),\cdots ,f_r(z))(\alpha )=
%\begin{cases}
%[f_1(\alpha),\cdots ,f_r(\alpha)] & \mbox{ if }\alpha \in\C
%\\
%[1,1,\cdots ,1] & \mbox{ if }\alpha =\infty
%\end{cases}
%\end{equation}
%%%%%%%%%%%
%for $(f_1(z),\cdots ,f_r(z))\in \Hol_D^*(S^2,\XS)$ and 
%$\alpha \in S^2=\R^2\cup \infty$,
%where
%$\Omega^2_D\XS$ denotes the path-component of $\Omega^2\XS$
%which contains the image of the map $i_{D,hol}$.
%Note that this map $i_{D,hol}$ is the same map explained in Theorem \ref{thm: KY9}
%and this is needed in \S \ref{section: stability}.
}
%%%%%%%
\end{dfn}
%%%%%(End of Definition 2.12)%%

%%(Remark 2.13)%%%
%%%
\begin{rmk}
%%%%
{\rm
(i)
Note that $\Pol_D^*(S^1,\XS)$ is path-connected
(see (ii) of Remark \ref{rmk: ESigma}), and that
$\XS$ is simply connected if
the condition (\ref{equ: homogenous}.2) is
satisfied
(see (i) of Lemma \ref{lmm: XS}).
\par
(ii)
Even if $\sum_{k=1}^rd_k\textbf{\textit{n}}_k\not= {\bf 0}_n$
we can define the two maps 
$$
i_D:\Pol^*_D(S^1,\XS) \to \Omega \XS , \ \ 
j_D:\Pol^*_D(S^1,\XS) \to \Omega U(\KS)
$$
and this will be done in \S \ref{section: stabilization map}
(see (\ref{eq:jDDD}) %and (\ref{eq: iD corollary}) 
in detail).
\qed
}
\end{rmk}
%%%%%%(End of Remark 2.13)%%%%

%%%%%%
Now we need to define the numbers
$\rmin (\Sigma)$ and
$d(D,\Sigma)$.

%%(Definition 2.14)%%%%%%%
\begin{dfn}\label{dfn: rnumber}
%%%%%%%%%%%%%%%%%%%%%%%%%
{\rm
Let $\Sigma$ be a fan in $\R^n$ as in Definition \ref{dfn: fan}.
\par
(i)
We say that a set $S=\{\textbf{\textit{n}}_{i_1},\cdots ,\textbf{\textit{n}}_{i_s}\}$
%of some primitive generators
is  {\it primitive in }$\Sigma$
if $\mbox{Cone}(S)\notin \Sigma$ but $\mbox{Cone}(T)\in \Sigma$ for any proper
subset $T\subsetneqq S$.
% if it does not span a cone in $\Sigma$ but any proper subset of it 
%does.
\par
(ii)
For $D=(d_1,\cdots ,d_r)\in \N^r$
define integers $\rmin (\Sigma)$ and $d(D,\Sigma ;m)$
by
%%(2.24)%%%%%%%
\begin{equation}\label{eq: rmin1}
%%%%%%%%%%%%%%%
\begin{cases}
%%%%%%%%
\quad
\rmin (\Sigma)
&=\min\{s\in\N:
\{\textbf{\textit{n}}_{i_1},\cdots ,\textbf{\textit{n}}_{i_s}\}
\mbox{ is primitive in }\Sigma\},
\\
%%%%%%%
d(D,\Sigma ;m)&=(2\rmin  (\Sigma) -m-1)d_{\rm min}-2,
\ \ \mbox{where}
\\
\quad
d_{min}&=\min \{d_1,\cdots ,d_r\}.
\end{cases}
\end{equation}
%%%%
In particular, when $m=1$ we also define the positive integer 
$d(D,\Sigma)$ by
%%(2.25)%%
\begin{equation}\label{eq: rmin}
d(D,\Sigma)=d(D,\Sigma;1)=(2\rmin  (\Sigma) -2)d_{\rm min}-2.
\qquad
\qed
\end{equation}
%%%%
%%%%%%%
}
%%%%%%
\end{dfn}
%%%%%%%%(End of Definition 2.14)%%%

Now recall the following result.
%%
%%(Theorem 2.15)%%
\begin{thm}[\cite{KOY1}]\label{thm: KOY1}
%%%%%%
Let $m\geq 2$ be a positive integer,  $\XS$ be a compact smooth toric variety
and $D=(d_1,\cdots ,d_r)\in \N^r$ be an $r$-tuple of positive integers such that
$\sum_{k=1}^rd_k\textbf{\textit{n}}_k={\bf 0}_n$.
Then the natural map
$$
i_{D,m}:\Pol_D^*(\RP^m,\XS)\to \Map^*_D(\RP^m,\XS)
$$
is a homology equivalence through dimension $d(D,\Sigma;m)$.
\qed
\end{thm}
%%(End of Theorem 2.15)%%%
%%
Note that the above result does not hold for the case $m=1$.
For example, this can be seen in \cite{GKY2} for the case $\XS =\CP^n$. 
%(the complex projective space).
%\par
In fact, the main purpose of this paper is to investigate the  result corresponding to this theorem for the case $m=1$.

%%%%%%(The main results)%%
\paragraph{The main results}
%%%%%%%%%%%%%%%%%%
More precisely  the main results of this paper are as follows.
%%%%
%%%%
%%%%
%%%%
%%%(The main Theorem I)%%%%%
%%%
%%%(Theorem 2.16: Theorem I)%%%%%
\begin{thm}\label{thm: I}
%%%%%%%%%%%%%%%%%%%%%%
Let $D=(d_1,\cdots ,d_r)\in \N^r$ be an $r$-tuple of positive integers 
satisfying the condition $\sum_{k=1}^rd_k\textit{\textbf{n}}_k={\bf 0}_n$,
and let $\XS$ be a simply connected non-singular toric variety such that 
the condition (\ref{equ: homogenous}.1) holds.
%%%
\par
Then the map
$$
j_D:\Pol_D^*(S^1,\XS) \to 
\Omega U(\KS)\simeq
\Omega \mathcal{Z}_{\KS}
$$
is a homotopy equivalence through dimension
$d(D,\Sigma).$
%%%%
\end{thm}
%%%%%%%%(End of Theorem 2.16)%%
%%
%%
%%(Corollary 2.17: Corollary I)%%
\begin{crl}\label{crl: I}
%%%%
Under the same assumption as in Theorem \ref{thm: I},
the map
$i_D:\Pol_D^*(S^1,\XS)\to \Omega \XS$ induces an isomorphism
$$
(i_D)_*:\pi_k(\Pol_D^*(S^1,\XS))
\stackrel{\cong}{\longrightarrow}
\pi_k(\Omega \XS)\cong \pi_{k+1}(\XS)
$$
for any $2\leq k\leq d(D,\Sigma)$.
%%%
\end{crl}
%%(End of Corollary 2.17)%%%%

%%
%%(Corollary 2.18: Corollary I-2)%%
\begin{crl}\label{crl: I-2}
%%%%%%%%%%%%%%%%%
Let $D=(d_1,\cdots ,d_r)\in \N^r$ be an $r$-tuple of positive integers 
satisfying the condition $\sum_{k=1}^rd_k\textit{\textbf{n}}_k={\bf 0}_n$,
and let $\XS$ be a simply connected compact non-singular toric variety. 
%such that 
%the condition (\ref{equ: homogenous}.1) holds.
Let $\Sigma (1)$ denote the set of all one dimensional cones in $\Sigma$, and $\Sigma_1$ any fan in $\R^n$ such that
%%%%%
$\Sigma (1)\subset \Sigma_1\subsetneqq \Sigma.$
%%%%%%%
\par
$\I$
Then $X_{\Sigma_1}$ is a non-singular open toric subvariety of $\XS$ and
the map
$$
j_D:\Pol_D^*(S^1,X_{\Sigma_1})\to 
\Omega U(\KS)\simeq
\Omega \mathcal{Z}_{\Sigma_1}
$$
is a homotopy equivalence through dimension
$d(D,\Sigma_1).$
\par
$\II$
Moreover,
the map
$i_D:\Pol_D^*(S^1,X_{\Sigma_1})\to \Omega X_{\Sigma_1}$ induces the isomorphism
$$
(i_D)_*:\pi_k(\Pol_D^*(S^1,X_{\Sigma_1}))
\stackrel{\cong}{\longrightarrow}
\pi_k(\Omega X_{\Sigma_1})\cong \pi_{k+1}(X_{\Sigma_1})
$$
for any $2\leq k\leq d(D,\Sigma_1)$.
%%%%%%%%%
\end{crl}
%%%%(End of Corollary 2.18)%%%%
%%
%%(Remark 2.19)%%
\begin{rmk}
%%%%%%%%%%%%%%%%
{\rm
If $\XS$ is compact,  the condition (\ref{equ: homogenous}.1) is satisfied 
(see (ii) of Remark \ref{rmk: assumption}) and
one can prove the above two results (Theorem \ref{thm: I}, Corollary \ref{crl: I})
by using \cite[Theorem 6.2]{KOY1}.
However, the proof given in \cite{KOY1} uses the Stone-Weierstrass theorem and
it cannot be applied  when $\XS$ is not compact.
\qed
%%%
}
%%%%%
\end{rmk}
%%%%%%%%%%%%%

Finally consider the $r$-tuple
$D=(d_1,\cdots ,d_r)\in \N^r$
such that $\sum_{k=1}^rd_k\textit{\textbf{n}}_k\not= {\bf 0}_n$.
Then the map $i_D$ is not well-defined
(see Lemma \ref{lmm: Map D}).
However, even in this situation we have a map
$j_D:\Pol_D^*(S^1,\XS)\to \Omega \mathcal{Z}_{\KS}$ and
the following result holds.

%%%(Theorem 2.20: The main Theorem II)%%%%%
\begin{thm}\label{thm: II}
%%%%%%%%%%%%%%%%%%%%%%
Let $D=(d_1,\cdots ,d_r)\in \N^r$ be an $r$-tuple of positive integers 
such that $\sum_{k=1}^rd_k\textit{\textbf{n}}_k\not= {\bf 0}_n$,
and let $\XS$ be a simply connected non-singular toric variety such that 
the condition (\ref{equ: homogenous}.1) holds. 
%%%
\par
Then the map defined by (\ref{eq:jDD})
$$
j_D:\Pol_D^*(S^1,\XS)\to 
\Omega U(\KS)
\simeq
\Omega \mathcal{Z}_{\KS}
$$ 
 is  a homotopy equivalence thorough dimension
$d(D,\Sigma).$
%%%%
\end{thm}
%%%%%%%%(End of Theorem 2.20)%%
%%
%%
%%(Corollary 2.21: Corollary II)%%
\begin{crl}\label{crl: II}
%%%%%%%%%%%%%%%%%%%
Under the same assumption as Theorem \ref{thm: II}, the map
$$
i_D=(\Omega q_{\Sigma})\circ j_D:\Pol^*_D(S^1,\XS) \to \Omega \XS
$$
induces an isomorphism
$$
(i_D)_*:\pi_k(\Pol_D^*(S^1,\XS))
\stackrel{\cong}{\longrightarrow}
\pi_k(\Omega \XS)\cong \pi_{k+1}(\XS)
$$
for any $2\leq k\leq d(D,\Sigma)$.
\end{crl}
%%%%(End of Collorary 2.21)%%%
%%
%%
Since $\XS$ is compact and $\Sigma (1)\subset \Sigma_1\subsetneqq \Sigma$,
the condition (\ref{equ: homogenous}.1) holds for the fan $\Sigma_1$ and
we obtain the following result by using Theorem \ref{thm: II} and
Corollary \ref{crl: II}.
%%
%%(Corollary 2.22: Corollary II-2)%%
\begin{crl}\label{crl: II-2}
%%%%%%%%%%%%%%%%%
Let $D=(d_1,\cdots ,d_r)\in \N^r$ be an $r$-tuple of positive integers,
and let $\XS$ be a simply connected compact non-singular toric variety. 
%such that 
%the condition (\ref{equ: homogenous}.1) holds.
Let $\Sigma (1)$ denote the set of all one dimensional cones in $\Sigma$, and $\Sigma_1$ any fan in $\R^n$ such that
%%%%%
$\Sigma (1)\subset \Sigma_1\subsetneqq \Sigma.$
%%%%%%%
\par
$\I$
Then $X_{\Sigma_1}$ is a non-singular open toric subvariety of $\XS$ and
the map
$$
j_D:\Pol_D^*(S^1,X_{\Sigma_1})\to \Omega U(\mathcal{K}_{\Sigma_1})
\simeq
\Omega \mathcal{Z}_{\Sigma_1}
$$
is a homotopy equivalence through dimesnion
$d(D,\Sigma_1).$
\par
$\II$
Furthermore, the map
$i_D:\Pol_D^*(S^1,X_{\Sigma_1})\to \Omega X_{\Sigma_1}$ induces the isomorphism
$$
(i_D)_*:\pi_k(\Pol_D^*(S^1,X_{\Sigma_1}))
\stackrel{\cong}{\longrightarrow}
\pi_k(\Omega X_{\Sigma_1})\cong \pi_{k+1}(X_{\Sigma_1})
$$
for any $2\leq k\leq d(D,\Sigma_1)$.
%%%%%%%%%
\end{crl}
%%%%(End of Corollary 2.22)%%%%

%%
%%
%%
%%%%%%%%%(Remark 2.23)%%
\begin{rmk}
%%%%%%%
{\rm
The map $i_D$ is  well-defined only when
$\sum_{k=1}^rd_k\textit{\textbf{n}}_k={\bf 0}_n$.
Hence, when $\sum_{k=1}^rd_k\textit{\textbf{n}}_k\not= {\bf 0}_n$,
there is no map $j_D:\Pol_D^*(S^1,\XS)\to \Omega \mathcal{Z}_{\KS}$ satisfying
the condition (\ref{eq: lift}).
However, even if $\sum_{k=1}^rd_k\textit{\textbf{n}}_k\not={\bf 0}_n$
we can define the map
$j_D:\Pol_D^*(S^1,\XS)\to \Omega\mathcal{Z}_{\KS}$ 
by using the map $j_{D_0}$ (for some 
$D_0=(m_0d_{1}^*,\cdots ,m_0d_{r}^*)\in \N^r$)
and the stabilization map $s_{D,D_0}$.
In fact, In this case the maps $j_D$ and $i_D$ are defined by
$j_D=j_{D_0}\circ s_{D,D_0}$ and
$i_D=\Omega q_{\Sigma}\circ j_D$ 
(see  (\ref{eq:jDD}) and (\ref{eq: iD corollary}) in detail).
So the condition (\ref{eq: lift}) still holds.
\qed
}
%%%%%%
\end{rmk}
%%%%%(End of Remark 2.23)%%%%
%%
%%
%%
%%
%%%(Example of main results)%%
\paragraph{Examples}
%%%%%%%%%%%%%%%%%%%%%
Since the case $\XS=\CP^n$ of Corollary \ref{crl: I-2}
was treated in \cite{KY1},  consider
the case that $\XS$ is the Hirzerbruch surface $H(k)$. 
%%%
%%%
%%%%(Definition 2.24)%%
\begin{dfn}\label{dfn: Hirzerbruch}
%%%%%%%%%%%%%%%%%%
{\rm
For an integer $k\in \Z$, let
$H(k)$ be {\it the Hirzerbruch surface} defined by
%()%%
\begin{equation*}
H(k)=\big\{([x_0:x_1:x_2],[y_1:y_2])\in\CP^2\times \CP^1:x_1y_1^k=x_2y_2^k
\big\}
\subset \CP^2\times \CP^1.
\end{equation*}
%%%
Since there are isomorphisms
$H(-k)\cong H(k)$ for $k \not=0$ and $H(0)\cong\CP^1\times \CP^1$,
without loss of generality we can assume that $k\geq 1$.
Let $\Sigma_k$ denote the fan
 in $\R^2$ given by
%%%%%%%
$$
\Sigma_k=
\big\{\mbox{Cone}({\bf n}_i,{\bf n}_{i+1})\ (1\leq i\leq 3),
\mbox{Cone}({\bf n}_4,{\bf n}_1),
\mbox{Cone}(\textit{\textbf{n}}_j)\ (1\leq j\leq 4),\ \{{\bf 0}\}\big\},
$$
where we set
$
\textit{\textbf{n}}_1%=\textit{\textbf{e}}_1
=(1,0),\  
\textit{\textbf{n}}_2%=\textit{\textbf{e}}_2
=(0,1),\  
\textit{\textbf{n}}_3%=-\textit{\textbf{e}}_1+k\textit{\textbf{e}}_2
=(-1,k), 
\ 
\textit{\textbf{n}}_4%=-\textit{\textbf{e}}_2
=(0,-1).
$
\par
It is easy to see that $\Sigma_k$ is the fan of $H(k)$
and that
 $H(k)$ is a compact non-singular toric variety.
Note that 
 $\Sigma_k(1)=\{\mbox{Cone}(\textit{\textbf{n}}_i):1\leq i\leq 4\}$.
Since $\{\textit{\textbf{n}}_1,\textit{\textbf{n}}_3\}$ and
$\{\textit{\textbf{n}}_2,\textit{\textbf{n}}_4\}$ are 
only primitive in $\Sigma_k$, $\rmin (\Sigma_k)=2$.
%%%
%%
\par
Moreover, for $D=(d_1,d_2,d_3,d_4)\in \N^4$
%of positive integers,
the equality
$\sum_{k=1}^4d_k\textit{\textbf{n}}_k=\textbf{0}_2$ holds 
iff
$(d_3,d_4)=(d_1,kd_1+d_2)$.
Thus, if $\sum_{k=1}^4d_k\textit{\textbf{n}}_k=\textbf{0}_2$,
we have
$
\dmin =\min \{d_1,d_2,d_3,d_4\}
=\min\{d_1,d_2\}.
$
\qed
%%
%%%%%%%%%%%%%%%%%%%%%%%%
}
%%%%%%%%%
\end{dfn}
%%(End of Definition 2.24)%%
%%
%%
%%

By Corollary \ref{crl: I-2} and Corollary \ref{crl: II-2} we have:
%%
%%
%%%%(Example 2.25)%%%
\begin{example}\label{example: H(k)}
%%%%%%%%%%%%%%%%%%%%
Let $D=(d_1,d_2,d_3,d_4)\in \N^4$,
$k\in \N$, and 
$\Sigma$ be a fan in $\R^2$ such that
$
\Sigma_k(1)=\{\mbox{\rm Cone}(\textit{\textbf{n}}_i):1\leq i\leq 4\}
\subset \Sigma \subset \Sigma_k$
as in Definition \ref{dfn: Hirzerbruch}.
\par
%%(i)%%%
$\I$ 
$\XS$ is a non-singular open toric subvariety of $H(k)$
if $\Sigma \subsetneqq \Sigma_k$.
\par
%%(ii)%%%
$\II$
If $\sum_{k=1}^4d_k\textbf{\textit{n}}_k={\bf 0}_2$,
the equality
$(d_3,d_4)=(d_1,kd_1+d_2)$ holds and
the map
$
j_D:\Pol_D^*(S^1,\XS)\to \Omega \mathcal{Z}_{\KS}
$
 is a homotopy equivalence
through dimension $2\min \{d_1,d_2\}-2$.
Moreover,
the map
$
i_D:\Pol_D^*(S^1,\XS)\to \Omega \XS
$
induces an isomorphism
$$
(i_D)_*:\pi_k(\Pol_D^*(S^1,\XS))
\stackrel{\cong}{\longrightarrow} \pi_k(\Omega \XS)\cong\pi_{k+1}(\XS)
$$
for any $2\leq k\leq 2 \min \{d_1,d_2\}-2$.
\par
$\III$
If $\sum_{k=1}^4d_k\textbf{\textit{n}}_k\not= {\bf 0}_2$,
the map
$
j_D:\Pol_D^*(S^1,\XS)\to \Omega \mathcal{Z}_{\KS}
$
 is a homotopy equivalence
through dimension $2\min \{d_1,d_2,d_3,d_4\}-2$, and
the map
$
i_D:\Pol_D^*(S^1,\XS)\to \Omega \XS
$
induces an isomorphism
$$
(i_D)_*:\pi_k(\Pol_D^*(S^1,\XS))
\stackrel{\cong}{\longrightarrow} \pi_k(\Omega \XS)\cong\pi_{k+1}(\XS)
$$
for any $2\leq k\leq 2 \min \{d_1,d_2,d_3,d_4\}-2$.
\qed
%%%%
\end{example}
%%(End of Example 2.25)%%

%%
%%(Remark 2.26)%%
\begin{rmk}
%%%%
{\rm
(i) There are $15$ non isomorphic (as varieties) non-compact subvarieties 
$\XS$ of $H(k)$
which satisfy the assumption of Corollary \ref{example: H(k)}.
\par
(ii)
Note that there is an isomorphism
%%(2.26)%%
\begin{equation}\label{eq: pi2XS}
%%%%%%%%
\pi_2(\XS)\cong \Z^{r-n}
\qquad
(\mbox{see Lemma \ref{lmm: XS} below)}, 
\end{equation}
%%%% 
in general.
%(see Lemma \ref{lmm: XS} below),
So $(r-n)$ of the $r$ positive integers $\{d_k\}_{k=1}^r$ can be 
chosen freely.
For example, in Example \ref{dfn: Hirzerbruch},
$(r,n)=(4,2)$ and
$r-n=4-2=2$. In this case, only two positive integers $d_1$, $d_2$ can be chosen freely
and the other integers $d_3$ and $d_4$ are determined by
uniquely as
$(d_3,d_4)=(d_1,kd_1+d_2).$
\qed
%%%%%%%%
}
\end{rmk}
%%(End of Remark 2.26)%%%
%%
Finally in this section we give the proof of Lemma \ref{lmm: Map D}.
%%
%%
%%
%%
%%
%%(Proof of Lemma 2.6)%%
\begin{proof}[Proof of Lemma \ref{lmm: Map D}]
%%%%%%%%%%
It follows from the assumptions that we can identify
$\XS= U(\KS)/\GS$ and that the map
$F=(f_1,\cdots ,f_r):\R^{m+1}\setminus \{{\bf 0}_{m+1}\}\to U(\KS)$ is well-defined.
Thus,
it suffices to show that
$F(\lambda \textit{\textbf{x}})=F(\textit{\textbf{x}})$ up to $\GS$-action
for any $(\lambda ,\textit{\textbf{x}})\in \R^*\times (\R^{m+1}\setminus \{{\bf 0}_{m+1}\})$
iff $\sum_{k=1}^rd_k\textit{\textbf{n}}_k={\bf 0}_n$.
Since $f_k$ is a homogenous polynomial of degree $d_k$, by using
(\ref{eq: multiplication})
we have
%%%
\begin{align*}
%%%
F(\lambda \textit{\textbf{x}})
&=
(f_1(\lambda \textit{\textbf{x}}),\cdots ,f_r(\lambda \textit{\textbf{x}}))
=(\lambda^{d_1}f_1(\textit{\textbf{x}}),\cdots ,\lambda^{d_r}f_r(\textit{\textbf{x}}))
\\
&=
(\lambda^{d_1},\cdots ,\lambda^{d_r})\cdot
(f_1(\textit{\textbf{x}}),\cdots ,f_r(\textit{\textbf{x}})).
%%%
\end{align*}
%%%
Hence, it remains to show that
$(\lambda^{d_1},\cdots ,\lambda^{d_r})\in \GS$
for any $\lambda \in \R^*$
iff $\sum_{k=1}^rd_k\textit{\textbf{n}}_k={\bf 0}_n$.
However,
$(\lambda^{d_1},\cdots ,\lambda^{d_r})\in \GS$ 
for any $\lambda \in \R^*$ iff
%%%
%\begin{align*}
%%%
%&(\lambda^{d_1},\cdots ,\lambda^{d_r})\in \GS
%\Leftrightarrow
$$
\prod_{k=1}^r(\lambda^{d_k})^{\langle \textit{\textbf{n}}_k,\textit{\textbf{m}}\rangle}
=\lambda^{\langle \sum_{k=1}^rd_k\textit{\textbf{n}}_k,\textit{\textbf{m}}\rangle }
=1
\ \mbox{ for any }\textit{\textbf{m}}\in \Z^n
\ 
\Leftrightarrow
\ 
\sum_{k=1}^rd_k\textit{\textbf{n}}_k={\bf 0}_n
$$
%\\
%&
%\Leftrightarrow
%\sum_{k=1}^rd_k\textit{\textbf{n}}_k={\bf 0}_n
%%%%
%\end{align*}
%%%
and this completes the proof.
\end{proof}
%%%(End of proof of Lemma 2.6)%%%
%%
%%
%%
%%
%%
%%
%%
%%
%%
%%
%%
%%
%%
%%
%%
%%%%%%%%%%%%%%
%%%(SECTION 3)%%%%
\section{Polyhedral products and toric varieties}\label{section: toric variety}
%%%%%%%%%%%%%%%

In this section, 
we recall several  known results from toric topology.
%%%
%%%

%%%(Definition 3.1)%%%
\begin{dfn}[\cite{BP}; Example 6.39]
\label{def: moment-angle}
%%%
{\rm
Let $K$ be a simplicial complex on the vertex set $[r]$.
\par
(i)
Then we denote by
$\mathcal{Z}_K$
%$\R\mathcal{Z}_K$ 
and $DJ(K)$ 
{\it the moment-angle complex} of $K$
 %and
% {\it the real moment-angle complex} of $K$ 
 and
{\it the Davis-Januszkiewicz space} of $K$
defined by
%%(3.1)%%%
\begin{equation}\label{DJ}
%%%%%%%%
\mathcal{Z}_K = 
\mathcal{Z}_K(D^2,S^1),
%\quad
%\R\mathcal{Z}_K = 
%\mathcal{Z}_K(I,\partial I),
\quad
DJ(K)\ = \ \mathcal{Z}_K(\CP^{\infty},*).
%%%
\end{equation}
%%%%%
\par
(ii)
Let $\iota_K:\mathcal{Z}_K=\mathcal{Z}_K(D^2,S^1)\to
DJ(K)=\mathcal{Z}_K(\CP^{\infty},*)$
denote the natural map induced from the following composite of maps
%%(3.2)%%
\begin{equation}\label{eq: pinch}
(D^2,S^1)\stackrel{pin}{\longrightarrow}(S^2,*)=(\CP^1,*)
\stackrel{\subset}{\longrightarrow}
(\CP^{\infty},*),
\end{equation}
where $pin:(D^2,S^1)\to (S^2,*)$ denotes the natural pinching map.
\qed
}
%where 
%$I=[-1,1]=D^1$ and $\partial I=\{-1,1\}=S^0$.
%\qed
%%%%%%%%%%%%%%%%%%%%
\end{dfn}
%%(End of Definition 3.1)%%%

%%%
%%%%%(Lemma 3.2)%%%
\begin{lmm}[\cite{BP}, \cite{KY9}]\label{Lemma: BP}
%%%%%%%%%%%%%%%%%%%
Let $K$ be a simplicial complex on the vertex set $[r]$.
\par
$\I$
The space $\mathcal{Z}_K$ is $2$-connected and
there is an $\T^r$-equivariant deformation retraction
%%(3.3)%%
\begin{equation}\label{eq: retr}
r_{\C}:U(K)=\mathcal{Z}_K(\C,\C^*)\stackrel{\simeq}{\longrightarrow}
\mathcal{Z}_K,
\quad\mbox{where }\ \T^r =(S^1)^r.
\end{equation}
%%%
\par
$\II$
There is a fibration sequence (up to homotopy)
%%(3.4)%%
\begin{equation}\label{eq: BP-DJ}
\mathcal{Z}_K \stackrel{\iota_K}{\longrightarrow}
DJ(K)
\stackrel{\subset}{\longrightarrow}
(\CP^{\infty})^r
%%%%%.
\end{equation}
%%%%
\par
%%%%
$\III$
If $\XS$ is a simply connected non-singular toric variety satisfying the condition
(\ref{equ: homogenous}.1),
there is a fibration sequence (up to homotopy)
%%(3.5)%%
\begin{equation}\label{eq: DJ}
\T^n_{\C}=(\C^*)^n \stackrel{}{\longrightarrow}
\XS
\stackrel{p_{\Sigma}}{\longrightarrow}DJ(\KS).
\end{equation}
%%%%
\par

\end{lmm}
%%(End of Lemma 3.2)%%
%%%
\begin{proof}
The assertions  follow from \cite[Theorem 6.33, Theorem 8.9]{BP},
\cite[Theorem 6.29, Corollary 6.30]{BP} and
\cite[Proposition 4.4]{KY9}.
\end{proof}
%%%
%%(Lemma 3.3)%%%%%%
\begin{lmm}[\cite{Pa1}; (6.2) and Proposition 6.7]\label{lmm: principal}
%%%%%%%%%%%%%%%
Let $\XS$ be a non-singular toric variety such that the condition 
(\ref{equ: homogenous}.2) holds.
Then
there is an isomorphism 
%%(3.6)%%
\begin{equation}\label{eq: GS-eq}
\GS\cong \T^{r-n}_{\C}=(\C^*)^{r-n},
\end{equation}
%%%%%
and the group $\GS$ acts on the space 
$\mathcal{Z}_{\mathcal{K}_{\Sigma}}(\C,\C^*)$
freely. 
So there is a principal
$\GS$-bundle
%%%
%%%(3.7)%%
\begin{equation}\label{eq: principal}
%%%%%%%%%
q_{\Sigma}:U(\KS)=
\mathcal{Z}_{\mathcal{K}_{\Sigma}}(\C,\C^*)\to\XS.
\qquad
\qed
%%%%%%%%%%%%%%
\end{equation}
%%%%%%%%%%%%%%}
%%%%%
\end{lmm}
%%%(End of Lemma 3.3)%%
%%
%%
%%
%%
%%%(Lemma 3.4)%%%
\begin{lmm}\label{lmm: XS}
%%%%%%%%%%%%%%%%%
$\I$
The space $\XS$ is simply connected if and only if
the condition (\ref{equ: homogenous}.2) is
satisfied.
\par
$\II$
If $\XS$ is a simply connected  non-singular toric variety, $\pi_2(\XS)=\Z^{r-n}$ and 
the map
%%(3.8)%%
\begin{equation}\label{eq: universal}
%%%%
\Omega q_{\Sigma}:\Omega \mathcal{Z}_{\KS}\to \Omega \XS
%%%
\end{equation}
is a universal covering projection with fiber $\Z^{r-n}$.
%%%
\end{lmm}
%%%%%%%%%%%%%%%%%
\begin{proof}
%%%
The assertion (i) easily follows from
\cite[Theorem 12.1.10]{CLS}, and
it suffices to show (ii).
%%%
By Lemma \ref{Lemma: BP} and (\ref{eq: retr}),  $\mathcal{Z}_{\KS}(\C,\C^*)
\simeq \mathcal{Z}_{\KS}$ is $2$-connected.
Then by using the homotopy exact sequence of the 
principal $\GS$-bundle (\ref{eq: principal}) and the isomorphism 
(\ref{eq: GS-eq}), we easily 
see that  $\pi_2(\XS)= \Z^{r-n}$
and that $\Omega q_{\Sigma}$ is a universal covering (up to homotopy).
%%%%%%%
\end{proof}
%%%%%%%%%(End of Proof of Lemma 3.4)%%
%%
%%(Remark 3.5)%%
\begin{rmk}\label{rmk: Omega XS}
%%%%%%
{\rm
If $\XS$ is a simply connected non-singular toric variety satisfying the condition
(\ref{equ: homogenous}.1), one can show that there is a homotopy equivalence
%%(3.9)%%
\begin{equation}\label{eq: omegaX}
\Omega \XS \simeq \Omega \mathcal{Z}_{\KS}\times \T^{r-n}.
\end{equation}
%%%
%%%%%
Although this can be proved by using 
(\ref{eq: retr}), (\ref{eq: GS-eq}) and (\ref{eq: principal}), we do not need this fact
and omit the detail.
\qed
}
%%%%%
\end{rmk}
%%%%%
%%(End of Remark 3.5)%%
%%
%%
%%
%%
%%%
%%%%%%%%%%%%%%%%%%%%%%%%%%%%%%%%%%%%%%%%%%%%%%%%%%%%%%%
%%(SECTION 4)%%%
\section{Simplicial resolutions}\label{section: simplicial resolution}
%%%%%%%%%%%%%%%%%%%%%%%%%%%%%%%%%%%%%%%%%%%%%%%%%%%%%%%
%%
%%
%%%%%%%%%%%
In this section, we summarize  the definitions of the non-degenerate simplicial resolution
and the associated truncated simplicial resolutions (\cite{Mo2},  \cite{Va}).
%%%%%
%%(Definition 4.1)%%
\begin{dfn}\label{def: def}
%%%%%%
{\rm
(i) For a finite set $\textbf{\textit{v}} =\{v_1,\cdots ,v_l\}\subset \R^N$,
let $\sigma (\textbf{\textit{v}})$ denote the convex hull spanned by 
$\textbf{\textit{v}}.$
%%%
%%%
Let $h:X\to Y$ be a surjective map such that
$h^{-1}(y)$ is a finite set for any $y\in Y$, and let
$i:X\to \R^N$ be an embedding.
Let  $\mathcal{X}^{\Delta}$  and $h^{\Delta}:{\mathcal{X}}^{\Delta}\to Y$ 
denote the space and the map
defined by
%%%
%%(4.1)%%%%%%%%
\begin{equation}
%%%%%%%%%%%%%%%
\mathcal{X}^{\Delta}=
\big\{(y,u)\in Y\times \R^N:
u\in \sigma (i(h^{-1}(y)))
\big\}\subset Y\times \R^N,
\ h^{\Delta}(y,u)=y.
\end{equation}
%%%%%%%%%%%%%%%%
The pair $(\mathcal{X}^{\Delta},h^{\Delta})$ is called
{\it the simplicial resolution of }$(h,i)$.
In particular, it %$(\mathcal{X}^{\Delta},h^{\Delta})$
is called {\it a non-degenerate simplicial resolution} if for each $y\in Y$
any $k$ points of $i(h^{-1}(y))$ span $(k-1)$-dimensional simplex of $\R^N$.
%%%%
\par
(ii)
For each $k\geq 0$, let $\mathcal{X}^{\Delta}_k\subset \mathcal{X}^{\Delta}$ be the subspace
of
the union of the $(k-1)$-skeletons of the simplices over all the points $y$ in $Y$
given by
%%(4.2)%%
\begin{equation}
%%%%%%%%
\mathcal{X}_k^{\Delta}=\big\{(y,u)\in \mathcal{X}^{\Delta}:
u \in\sigma (\textbf{\textit{v}}),
\textbf{\textit{v}}=\{v_1,\cdots ,v_l\}\subset i(h^{-1}(y)),l\leq k\big\}.
\end{equation}
%%%%
%%%%
We make the identification $X=\mathcal{X}^{\Delta}_1$ by identifying 
 $x\in X$ with the pair
$(h(x),i(x))\in \mathcal{X}^{\Delta}_1$,
and we note that  there is an increasing filtration
%%%(4.3)%%
\begin{equation}\label{equ: filtration}
%%%
\emptyset =
\mathcal{X}^{\Delta}_0\subset X=\mathcal{X}^{\Delta}_1\subset \mathcal{X}^{\Delta}_2\subset
\cdots \subset \mathcal{X}^{\Delta}_k\subset 
%\mathcal{X}^{\Delta}_{k+1}\subset
\cdots \subset \bigcup_{k= 0}^{\infty}\mathcal{X}^{\Delta}_k=\mathcal{X}^{\Delta}.
\end{equation}
%%%%%
Since the map $h^{\Delta}:\mathcal{X}^{\Delta}\stackrel{}{\rightarrow}Y$
 is a proper map, it extends to the map
 ${h}_+^{\Delta}:\mathcal{X}^{\Delta}_+\stackrel{}{\rightarrow}Y_+$
 between the one-point compactifications,
 where $X_+$ denotes the one-point compactification of a locally compact space $X$.
%%%
\par
(iii)
A space $X\subset \R^n$ is called  {\it semi-algebraic} if it is a subspace
of the form
$X=\bigcup_{i=1}^s\bigcap_{j=1}^{r_i}\{(x_1,\cdots ,x_n)\in\R^n:f_{ij}\ *_{ij}\  0\}$,
where
$f_{ij}\in \R [X_1,\cdots ,x_n]$ and
$*_{ij}$ is either $<$ or $=$, for $i=1,\cdots s$ and
$j=1,\cdots ,r_i$.
Similarly, when $X\subset \R^n$ and $Y\subset \R^m$ are semi-algebraic spaces,
a map $f:X\to Y$ is {\it a semi-algebraic map} if
the graph of $f$ is semi-algebraic.
}
\end{dfn}
%%%%%(End of Definition 4.1)%%%%%%%%

%%%(Lemma 4.2)%%%%%%%%%%%%%%%%%%
\begin{lmm}[\cite{Va}, \cite{Va2} (cf. 
Lemma 3.3 in \cite{KY7})]\label{lemma: simp}
%%%%%%%%%%%%%%%%%%%%%%%%%%%
Let $h:X\to Y$ be a surjective map such that
$h^{-1}(y)$ is a finite set for any $y\in Y,$ and let
$i:X\to \R^N$ be an embedding.
\par
%%(i)%%
$\I$
If $X$ and $Y$ are semi-algebraic spaces and the
two maps $h$, $i$ are semi-algebraic maps, then the map
${h}^{\Delta}_+:\mathcal{X}^{\Delta}_+\stackrel{\simeq}{\rightarrow}Y_+$
is a homotopy equivalence.
%%%
\par
$\II$
There is an embedding $j:X\to \R^M$ such that
the associated simplicial resolution
$(\tilde{\mathcal{X}}^{\Delta},\tilde{h}^{\Delta})$
of $(h,j)$ is non-degenerate.
%%%
%%%%
\par
$\III$
If there is an embedding $j:X\to \R^M$ such that the associated simplicial resolution
$(\tilde{\mathcal{X}}^{\Delta},\tilde{h}^{\Delta})$
of $(h,j)$ is non-degenerate,
the space $\tilde{\mathcal{X}}^{\Delta}$
is uniquely determined up to homeomorphism.
Moreover,
there is a filtration preserving homotopy equivalence
$q^{\Delta}:\tilde{\mathcal{X}}^{\Delta}\stackrel{\simeq}{\rightarrow}{\mathcal{X}}^{\Delta}$ such that $q^{\Delta}\vert X=\mbox{id}_X$.
\qed
%%%
\end{lmm}
%%%(End of Lemma 4.2)%%

%%%%%(Remark 4.3)%%%
\begin{rmk}\label{Remark: homotopy equ}
%%%%
{\rm
In this paper we only need the weaker assertion that the map ${h}_+^{\Delta}$ is a homology equivalence. 
One can easily prove this result
%that the map   $\overline{h}^{\Delta}$ is a homology equivalence 
by the same argument as used 
%in the proof of  Lemma 1 (page 90) 
in the second  revised edition of Vassiliev's book \cite[Proof of Lemma 1 (page 90)]{Va}.
\qed}
%%% 
\end{rmk}
%%%(End of Remark 4.3)%%%%

%%%%(Remark 4.4)%%%%%%%%%%%
\begin{rmk}[\cite{Va}, \cite{Va2}]\label{Remark: non-degenerate}
%%%%%%%%%%%%%%%%%%%%%%%%%%%
{\rm
Even for a  surjective map $h:X\to Y$ which is not finite to one,  
it is still possible to construct an associated non-degenerate simplicial resolution.
Recall that it is known that there exists a sequence of embeddings
$\{\tilde{i}_k:X\to \R^{N_k}\}_{k\geq 1}$ satisfying the following two conditions
for each $k\geq 1$ (\cite{Va}, \cite{Va2}).
%%%(Condition of non-degenerate simp. resolution)%%
\begin{enumerate}
%\item[(\ref{Remark: non-degenerate}$)_k$]
%%%%%%%%%%%%%%%%%%%
%\begin{enumerate}
%%(i)%%%
\item[(i)]
For any $y\in Y$,
any $t$ points of the set $\tilde{i}_k(h^{-1}(y))$ span $(t-1)$-dimensional affine subspace
of $\R^{N_k}$ if $t\leq 2k$.
%%(ii)%%%
\item[(ii)]
$N_k\leq N_{k+1}$ and if we identify $\R^{N_k}$ with a subspace of
$\R^{N_{k+1}}$, 
then $\tilde{i}_{k+1}=\hat{i}\circ \tilde{i}_k$,
where
$\hat{i}:\R^{N_k}\stackrel{\subset}{\rightarrow} \R^{N_{k+1}}$
denotes the inclusion.
%\end{enumerate}
\end{enumerate}
%%%%%%%%%%%%
In this situation, in fact,
a non-degenerate simplicial resolution may be
constructed by choosing a sequence of embeddings
$\{\tilde{i}_k:X\to \R^{N_k}\}_{k\geq 1}$ satisfying the above two conditions
for each $k\geq 1$.

%%%
Let
%%%()%%%%%
%begin{equation*}\label{2.1}
$\dis\mathcal{X}^{\Delta}_k=\big\{(y,u)\in Y\times \R^{N_k}:
u\in\sigma (\textbf{\textit{v}}),
\textbf{\textit{v}}
=\{v_1,\cdots ,v_l\}\subset \tilde{i}_k(h^{-1}(y)),l\leq k\big\}.$
%\end{equation*}
%%%%
%%
Then
by identifying naturally  ${\cal X}^{\Delta}_k$ with a subspace
of ${\cal X}_{k+1}^{\Delta}$,  define the non-degenerate simplicial
resolution ${\cal X}^{\Delta}$ of  $h$ as %the union  
$\dis {\cal X}^{\Delta}=\bigcup_{k\geq 1} {\cal X}^{\Delta}_k$.
%Non-degenerate simplicial resolutions have a long been used in algebraic geometry, and play  the central role in the work of Vassiliev  \cite{Va}.
}
\qed
\end{rmk}
%%%%%(End of Remark 4.4)%%%%%%%%
%%%

%%%(Definition 4.5)%%%
\begin{dfn}\label{def: 2.3}
%%%%%%%%%%%%%%%%%%%%%%
{\rm
Let $h:X\to Y$ be a surjective semi-algebraic map between semi-algebraic spaces, 
$j:X\to \R^N$ be a semi-algebraic embedding, and let
$(\mathcal{X}^{\Delta},h^{\Delta}:\mathcal{X}^{\Delta}\to Y)$
denote the associated non-degenerate  simplicial resolution of $(h,j)$. 
%Under these conditions the map $h^{\Delta}$ is a homotopy equivalence as in Lemma \ref{lemma: simp}.
%%%
\par
%%%
Let $k$ be a fixed positive integer and let
$h_k:\mathcal{X}^{\Delta}_k\to Y$ be the map
defined by the restriction
$h_k:=h^{\Delta}\vert \mathcal{X}^{\Delta}_k$.
%\par
The fibers of the map $h_k$ are $(k-1)$-skeleton of the fibers of $h^{\Delta}$ and, in general,  always
fail to be simplices over the subspace
$Y_k=\{y\in Y:\mbox{card}(h^{-1}(y))>k\}.$
%$$
%Y_k=\{y\in Y:h^{-1}(y)\mbox{ consists of more than $k$ points}\}.
%$$
Let $Y(k)$ denote the closure of the subspace $Y_k$.
We modify the subspace $\mathcal{X}^{\Delta}_k$ so as to make  all
the fibers of $h_k$ contractible by adding to each fibre of $Y(k)$ a cone whose base
is this fibre.
We denote by $X^{\Delta}(k)$ this resulting space and by
$h^{\Delta}_k:X^{\Delta}(k)\to Y$ the natural extension of $h_k$.
\par
%%%%%%%%%%%%
Following  \cite{Mo3}, we call the map $h^{\Delta}_k:X^{\Delta}(k)\to Y$
{\it the truncated $($after the $k$-th term$)$  simplicial resolution} of $Y$.
Note that 
that there is a natural filtration
$$
 X^{\Delta}_0\subset
 X^{\Delta}_1\subset
%X^{\Delta}_2\subset
\cdots 
\subset X^{\Delta}_l\subset X^{\Delta}_{l+1}\subset \cdots
\subset  X^{\Delta}_k\subset X^{\Delta}_{k+1}
=X^{\Delta}_{k+2}
%=X^{\Delta}_{k+3}
=\cdots =X^{\Delta}(k),
$$
where $X^{\Delta}_0=\emptyset$,
$X^{\Delta}_l=\mathcal{X}^{\Delta}_l$ if $l\leq k$ and
$X^{\Delta}_l=X^{\Delta}(k)$ if $l>k$.
}
\end{dfn}
%%%%(End of Definition 4.5)%%%
%%
%%
%%
%%
%%
%%(Lemma 4.6)%%%%%%%%%
\begin{lmm}[\cite{Mo3}, cf. Remark 2.4 and Lemma 2.5 in \cite{KY4}]
\label{Lemma: truncated}
%%%%%%%%%%%%%%%%%%%%%%%
%Let $h:X\to Y$ be a surjective semi-algebraic map between semi-algebraic spaces,
%and 
%let $j:X\to \R^N$ be a semi-algebraic embedding with the associated simplicial resolution
%$({\cal X}^{\Delta},h^{\Delta}:{\cal X}^{\Delta}\to Y)$.
%Then
Under the same assumptions and with the same notation as in Definition \ref{def: 2.3}, the map
$h^{\Delta}_k:X^{\Delta}(k)\stackrel{\simeq}{\longrightarrow} Y$ 
is a homotopy equivalence.
\qed
\end{lmm}
%%%%%%%%%%%
%%(End of Lemma 4.6)%%%
%%%
%%(End of SECTION 4: Simplicial resolutions)%%%%%%
%%
%%%%%%%%%%%%%%%
%%%
%%%
%%%
%%%
%%%
%%%
%%%
%%%
%%%(SECTION 5: The Vassiliev spectral sequence)%%%
\section{The Vassiliev spectral sequence}
\label{section: spectral sequence}
%%%%%%%%%%%%%%%%%

In this section, we always assume that
$D=(d_1,\cdots ,d_r)\in \N^r$ is an $r$-tuple of positive integers
and let us write $d_{min}=\min \{d_1,\cdots ,d_r\}$.
%%
%From now on, we identify $\Pol_D^*(S^1,\XS)$ with the space %consisting 
%of all
%$r$-tuples
%$(f_1(z),\cdots ,f_{r}(z))\in \P^D$
%of monic polynomials
%such that
%$f_{i_1}(z),\cdots ,f_{i_s}(z)$
%have no common {\it real} root for any
%$\sigma =\{i_1,\cdots ,i_s\}\in I(\mathcal{K}_{\Sigma})$
%as in Definition \ref{dfn: holomorphic}.
%%
%%%%%%%%%%%%%%%%%%%
First, we construct the Vassiliev spectral sequence.

%%(Definition 5.1)%%%%%%%
\begin{dfn}\label{Def: 4.1}
{\rm
%%(i)%%%%
(i)
For a space $X$ and a positive integer $k\geq 1$, let
$F(X,k)$ denote {\it the ordered configuration space} of $k$-distinct points in $X$
given by
%%(5.1)%%%
\begin{equation}
F(X,k)=\{(x_1,\cdots ,x_k)\in X^k:x_i\not= x_j
\mbox{ if }i\not= j\}.
\end{equation}
%%%%%%%%
The symmetric group $S_k$ of $k$ letters acts on $F(X,k)$ freely by the usual coordinate
permutation and we denote by $C_k(X)$
{\it the unordered configuration space}  of $k$-distinct points  in $X$ given by the orbit space
%%(5.2)%%
\begin{equation}
C_k(X)=F(X,k)/S_k.
\end{equation}
%%%%%%
\par
%%%%%(ii)%%%
(ii)
Let $\Sigma_D$ denote {\it the discriminant} of $\Pol_D^*(S^1,\XS)$ in $\P^D$ 
given by the complement
%%%%
\begin{align*}
\Sigma_D&=
\P^D \setminus \Pol_D^*(S^1,\XS)
\\
&=
\{(f_1(z),\cdots ,f_{r}(z))\in \P^D :
(f_1(x),\cdots ,f_{r}(x))\in L(\Sigma)
\mbox{ for some }x\in \R\},
\end{align*}
where 
%%(5.3)%%
\begin{equation}
L(\Sigma)
%\tilde{L}(\KS)=
=\bigcup_{\sigma\in I(\mathcal{K}_{\Sigma})}L_{\sigma}
=
~\bigcup_{\sigma\subset[r],\sigma\notin K_{\Sigma}}
L_{\sigma}.
%\quad
%\mbox{(as in (\ref{IK}))}
\end{equation}
%%%%%%
\par
%%(iii%%
(iii)
Let  $Z_D\subset \Sigma_D\times \R$
denote %the subspace
{\it the tautological normalization} of 
 $\Sigma_D$ consisting of 
 all pairs 
$((f_1(z),\ldots ,f_{r}(z)),
x)\in \Sigma_D\times\R$
satisfying the condition
$(f_1(x),\cdots ,f_{r}(x))\in L(\Sigma)$.
%%%
%\par
%%%%
Projection on the first factor  gives a surjective map
$\pi_D :Z_D\to \Sigma_D$.
%%%
%%%
}
\end{dfn}
%%%%%%
%%(End of Definition 5.1)%%%%
%%%
%%%
%%%
%%%
%%(Remark 5.2)%%
\begin{rmk}
%%%%%%%%%%%%%%%
{\rm
Let $\sigma_k\in [r]$ for $k=1,2$.
It is easy to see that
$L_{\sigma_1}\subset L_{\sigma_2}$ if
$\sigma_1\supset \sigma_2$.
Letting
%%()%%
\begin{equation*}
%%% 
Pr(\Sigma)=\{\sigma =\{i_1,\cdots ,i_s\} \subset [r]:
\{\textit{\textbf{n}}_{i_1},\cdots ,\textit{\textbf{n}}_{i_s}\}
\mbox{ is a primitive in $\Sigma$}\},
\end{equation*}
%%%
we see that
%%%
%%(5.4)%%
\begin{equation}
L(\Sigma)=\bigcup_{\sigma\in Pr(\Sigma)}L_{\sigma}
\end{equation}
%%%  
and by using (\ref{eq: rmin}) we obtain the equality
%%% 
%%%(5.5)%%
\begin{equation}\label{eq: dim rmin}
%%%
\dim L(\Sigma)=2(r-\rmin (\Sigma)).
\end{equation}
}
%%%
\end{rmk}
%%%(End of Remark 5.2)%%%%%
%%%
%%%
%%%
%%%
Our goal in this section is to construct, by means of the
{\it non-degenerate} simplicial resolution  of the discriminant, a spectral sequence converging to the homology of
$\Pol_D^*(S^1,\XS)$.

%%%
%%%
%%%
%%%
%%(Definition 5.3)%%%%
\begin{dfn}\label{non-degenerate simp.}
%%%
{\rm
Let 
$(\mathcal{X}^D,{\pi}^{\Delta}_D:\mathcal{X}^D\to\Sigma_D)$ 
%and
%$(\tilde{\SZ}(d),\ ^{\p}\tilde{\pi}_d^{\Delta}:\SZ(d)\to\Sigma_d^*)$ 
be the non-degenerate simplicial resolution associated to the surjective map
$\pi_D:Z_D\to \Sigma_D$ 
with the natural increasing filtration as in Definition \ref{def: def},
$$
\emptyset =
\SZ_0
\subset \SZ_1\subset 
\SZ_2\subset \cdots
\subset 
\SZ=\bigcup_{k= 0}^{\infty}\SZ_k.
\qquad \qed
$$
}
\end{dfn}
%%%(End of Definition 5.3)%%%%%%
%%%
%%%
%%%
%%%
%%
%%%%%% 
By Lemma \ref{lemma: simp},
the map
$\pi_D^{\Delta}:
\SZ\stackrel{\simeq}{\rightarrow}\Sigma_D$
is a homotopy equivalence which
extends to  a homotopy equivalence
%%%%%%%%%
$\pi_{D+}^{\Delta}:\SZ_+\stackrel{\simeq}{\rightarrow}{\Sigma_{D+}},$
%%%%%
where $X_+$ denotes the one-point compactification of a
locally compact space $X$.
%%
%\par
%%%
Since
${\mathcal{X}_k^{D}}_+/{\SZ_{k-1}}_+
\cong (\SZ_k\setminus \SZ_{k-1})_+$,
we have a spectral sequence 
%%%%%%%%%%%
%$$
$$
\big\{E_{t;D}^{k,s},
d_t:E_{t;D}^{k,s}\to E_{t;D}^{k+t,s+1-t}
\big\}
\Rightarrow
H^{k+s}_c(\Sigma_D;\Z),
$$
%$$
where
$E_{1;D}^{k,s}=H^{k+s}_c(\SZ_k\setminus\SZ_{k-1};\Z)$ and
$H_c^k(X;\Z)$ denotes the cohomology group with compact supports given by 
$
H_c^k(X,\Z)= \tilde{H}^k(X_+;\Z).
$
%%%%%%%%%%%%%
\par\vspace{2mm}\par
%%%%%
Let $N(D)$ denote the positive integer given by
%%(5.6)%%
\begin{equation}\label{eq: ND}
%%%%%%%%
N(D)=\sum_{k=1}^r d_k.
\end{equation}
%%%%%%%%%
%%%
%%%
Since there is a homeomorphism
$\P^D\cong \C^{N(D)}$,
by Alexander duality  there is a natural
isomorphism

%%%(5.7)%%%
\begin{equation}\label{Al}
%%%%%%%%%
\tilde{H}_k(\Pol_D^*(S^1,\XS);\Z)\cong
H_c^{2N(D)-k-1}(\Sigma_D;\Z)
\quad
\mbox{for any }k.
\end{equation}
%%%
By
reindexing we obtain a
spectral sequence
%%

%%%(5.8)%%%
\begin{eqnarray}\label{SS}
%%%%%%%%%%%%%%%%%%%
&&\big\{E^{t;D}_{k,s}, \tilde{d}^{t}:E^{t;D}_{k,s}\to E^{t;D}_{k+t,s+t-1}
\big\}
\Rightarrow H_{s-k}(\Pol_D^*(S^1,\XS);\Z),
\end{eqnarray}
%%%%%%%
%if $s-k\leq 2nd-2$,
where
$E^{1;D}_{k,s}=
H^{2N(D)+k-s-1}_c(\SZ_k\setminus\SZ_{k-1};\Z).$
%and $\tilde{E}^t_{r,s}(d)=E_t^{r,N_d^*-1-s}(d).$
%%%%%%
\par\vspace{2mm}\par
%%%%
%Let $C_k(X)$ be the configuration space of unordered 
%$k$-distinct  points in $X$ given by the orbit space
%$C_k(X)=F(X,k)/S_k$.
%%%
Let
 $L_{k;\Sigma}\subset (\R\times L(\Sigma))^k$ denote the subspace
defined by
%%(5.9)%%%%
\begin{equation}\label{l}
%%%%
L_{k;\Sigma}=\{((x_1,s_1),\cdots ,(x_k,s_k)): 
x_j\in \R,s_j\in L(\Sigma),
x_i\not= x_j\mbox{ if }i\not= j\}.
%%%%
\end{equation}
%%%%%%%%%
The symmetric group $S_k$ on $k$ letters  acts on $L_{k;\Sigma}$ by permuting coordinates. Let
$C_{k;\Sigma}$ denote the orbit space
%%%%%
%%%%(5.10)%%%%
\begin{equation}\label{Ck}
C_{k;\Sigma}=L_{k;\Sigma}/S_k.
\end{equation}
%%%%%%%%%%%
%%%%
By (\ref{eq: dim rmin}) we know that
$C_{k;\Sigma}$ 
is a cell-complex of  dimension
%%(5.11)%%
\begin{equation}\label{eq: dimension C}
%%%%%
\dim C_{k;\Sigma}=(1+2r-2r_{\rm min}(\Sigma))k.
\end{equation}
%%%%%%%
%%
%%%
%%%
%%%
%%%
%%%%(Lemma 5.4)%%%%
\begin{lmm}\label{lemma: vector bundle*}
%%%%%%%%%%%%%%%%%%
%%%
If  
$1\leq k\leq d_{\rm min}=\min\{d_1,\cdots ,d_r\}$,
$\SZ_k\setminus\SZ_{k-1}$
is homeomorphic to the total space of a real affine
bundle $\xi_{D,k}$ over $C_{k;\Sigma}$ with rank 
$l_{D,k}=2N(D)-2rk+k-1$.
%%%%%%%%%%%%%%%%%%
\end{lmm}
%%%%%%%%(Proof of Lemma 5.4)%%%
\begin{proof}
%%%%%%%%%%%
The argument is exactly analogous to the one in the proof of  
\cite[Lemma 4.4]{AKY1}. 
%%%
%%%
Namely, an element of $\SZ_k\setminus\SZ_{k-1}$ is represented by 
$(F,u)=((f_1,\cdots ,f_{r}),u)$, where 
$F=(f_1,\cdots ,f_{r})$ is an 
$r$-tuple of monic polynomials in $\Sigma_D$ and $u$ is an element of the interior of
the span of the images of $k$ distinct points 
$\{x_1,\cdots, x_k\}\in C_k(\R)$ 
such that
%%%%%
$F(x_j)=(f_1(x_j),\cdots ,f_{r}(x_j))\in L(\Sigma)$ for each $1\leq j\leq k$, 
%%%%
under a suitable embedding.
%of $\C$ into Euclidean space.% satisfying the condition ({\ref{equ: filtration}}$)_k$.
%%%%
\ 
Since the $k$ distinct points $\{x_j\}_{j=1}^k$ 
are uniquely determined by $u$, by the definition of the non-degenerate simplicial resolution, %(cf.  ({\ref{equ: filtration}}$)_k$),
 there are projection maps
%%%%%%%%%%%% 
$\pi_{k,D} :\mathcal{X}^{D}_k\setminus
\mathcal{X}^{D}_{k-1}\to C_{k;\Sigma}$
%%%%%%%%%%%%
given by
$((f_1,\cdots ,f_{r}),u) \mapsto 
\{(x_1,F(x_1)),\dots, (x_k,F(x_k))\}$. 

\par
%%%
%%%%%(Fiber of pi_k)%%%%%%
Now suppose that $1\leq k\leq d_{\rm min}$ and
let $c=\{(x_j,s_j)\}_{j=1}^k\in C_{k;\Sigma}$
$(x_j\in \R$, $s_j\in L(\Sigma))$ be any fixed element.
Consider the fibre  $\pi_{k,D}^{-1}(c)$.
%%%
For each $1\leq j\leq k$,
we set $s_j=(s_{1,j},\cdots ,s_{r,j})$ and
consider the condition  
%%%(5.12)%%%
\begin{equation}\label{equ: pik}
%%%%%%%%%%%
F(x_j)=(f_1(x_j),\cdots ,f_{r}(x_j))=s_j
\Leftrightarrow
f_t(x_j)=s_{t,j}
\quad
\mbox{for }1\leq t\leq r.
\end{equation}
%%%%%%%%%%%
%%
In general, %for each $1\leq t\leq r$,
the condition $f_t(x_j)=s_{t,j}$ gives
one  linear condition on the coefficients of $f_t$,
and determines an affine hyperplane in $\P^{d_t}(\C)$. 
%%%
For example, if we set $f_t(z)=z^{d_t}+\sum_{i=0}^{d_t-1}a_{i,t}z^{i}$,
then
$f_t(x_j)=s_{t,j}$ for any $1\leq j\leq k$
if and only if
%%%%%%
%%(5.13)(matrix equation)%%
\begin{equation}\label{equ: matrix equation}
%%%%%%%%%
\begin{bmatrix}
1 & x_1 & x_1^2 & \cdots & x_1^{d_t-1}
\\
1 & x_2 & x_2^2 & \cdots & x_2^{d_t-1}
\\
\vdots & \ddots & \ddots & \ddots & \vdots
%\\
%1 & x_{k-1} & x_{k-1}^2 & \cdots & x_{k-1}^{d-1}
\\
1 & x_k & x_k^2 & \cdots & x_k^{d_t-1}
\end{bmatrix}
%%%%
\cdot
\begin{bmatrix}
a_{0,t}\\ a_{1,t} \\ \vdots %\\ a_{2,t} 
\\ a_{d_t-1,t}
\end{bmatrix}
=
\begin{bmatrix}
s_{t,1}-x_1^{d_t}\\ s_{t,2}-x_2^{d_t} \\ \vdots %\\ s_{t,k-1}-x_{k-1}^d 
\\ s_{t,k}-x_k^{d_t}
\end{bmatrix}
\end{equation}
%%%%
%%%%
Since $1\leq k\leq d_{\rm min}(D)$ and
 $\{x_j\}_{j=1}^k\in C_k(\R)$,  
it follows from the properties of Vandermonde matrices that the condition (\ref{equ: matrix equation}) 
gives exactly $k$ independent conditions on the coefficients of $f_t(z)$.
Thus  the space of polynomials $f_t(z)$ in $\P^{d_t}(\C)$ which satisfies
(\ref{equ: matrix equation})
is the intersection of $k$ affine hyperplanes in general position
and has codimension $k$ in $\P^{d_t}(\C)$.
%%%%%
Hence,
the fibre $\pi_{k,D}^{-1}(c)$ is homeomorphic  to the product of an open $(k-1)$-simplex
 with the real affine space of dimension
 $2\sum_{i=1}^r(d_i-k)=2N(D)-2rk$.
It is now easy to show that  $\pi_{k,D}$ is a (locally trivial) real affine bundle over $C_{k;\Sigma}$ of rank $l_{D,k}
=2N(D)-2rk+k-1$.
\end{proof}
%%(End of proof of Lemma 5.4)%%%
%%%
%%%
%%%
%%%
%%%
%%%
%%%
%%%
%%%%%%(Lemma 5.5)%%
\begin{lmm}\label{lemma: E11}
%%%%%%
If $1\leq k\leq  d_{\rm min}$, there is a natural isomorphism
$$
E^{1;D}_{k,s}\cong
H^{2rk-s}_c(C_{k;\Sigma};\pm \Z),
$$
where 
the twisted coefficients system $\pm \Z$  comes from
the Thom isomorphism.
\end{lmm}
%%%%
\begin{proof}
%%%(Proof of Lemma 5.5)%%
Suppose that $1\leq k\leq d_{\rm min}$.
%%%
By Lemma \ref{lemma: vector bundle*}, there is a
homeomorphism
$
(\SZ_k\setminus\SZ_{k-1})_+\cong T(\xi_{D,k}),
$
where $T(\xi_{D,k})$ denotes the Thom space of
%one-point compactification of
$\xi_{D,k}$.
%%%%%%%
Since $(2N(D)+k-s-1)-l_{D,k}
=
2rk-s,$
%$$
%%%%%
by using the Thom isomorphism 
there is a natural isomorphism 
%%%
$
E^{1;D}_{k,s}
%\cong 
%\tilde{H}^{2nd+k-s-1}(T(\xi_{d,k}),\Z)
\cong
H^{2rk-s}_c(C_{k;\Sigma};\pm \Z).
$
\end{proof}
%%%(End of proof of Lemma 5.5)%%%%%%
%%
%%%
%%%
%%%
%%%
%%%(Definition 5.6)%%%
\begin{dfn}
%%%%%%%%%%%%%%%%%%%%%%
{\rm
(i)
Let $X^{\Delta}$ denote the truncated 
(after the $d_{\rm min}$-th term) simplicial resolution of $\Sigma_D$
with the natural filtration
$$
\emptyset =X^{\Delta}_0\subset
X^{\Delta}_1\subset \cdots\subset
X^{\Delta}_{d_{\rm min}}\subset X^{\Delta}_{d_{\rm min}+1}=
X^{\Delta}_{d_{\rm min}+2}=\cdots =X^{\Delta},
$$
where $X^{\Delta}_k=\SZ_k$ if $k\leq d_{\rm min}$ and 
$X^{\Delta}_k=X^{\Delta}_{d_{\rm min}+1}$ if $k\geq d_{\rm min}+1$.
%%%%%%
\par\vspace{1mm}\par
%%%%
(ii)
Let $\{\textbf{\textit{e}}_i:1\leq i\leq r\}$ denote the standard $\R$-basis
of $\R^r$ given by
%%%
%%%
%%%
%%%
%%(5.14)%%
\begin{equation}\label{eq: standard basis}
%%%%%
\textit{\textbf{e}}_1=(1,0,0,\cdots 0),
\textit{\textbf{e}}_2=(0,1,0,\cdots 0),\cdots ,
\textit{\textbf{e}}_r=(0,\cdots ,0, 0,1)\in \R^r.
\end{equation}
%%%%%%%%
For each $1\leq i\leq r$,
let us consider
%%)%%
%%%%%%%%%%%%%%
%%(5.15)%%
\begin{equation}\label{eq: Di}
%%%%%%%
D+\textit{\textbf{e}}_i=(d_1,\cdots ,d_{i-1},d_i+1,d_{i+1},\cdots ,d_r)
\end{equation}
%%%%%%%%%%%%%%
and
let $Y^{\Delta}$ denote the truncated (after the $d_{\rm min}$-th term) simplicial resolution of 
$\Sigma_{D+\textit{\textbf{e}}_i}$
with the natural filtration
$$
\emptyset =Y^{\Delta}_0\subset
Y^{\Delta}_1\subset \cdots\subset
Y^{\Delta}_{d_{\rm min}}\subset Y^{\Delta}_{d_{\rm min}+1}
=Y^{\Delta}_{d_{\rm min}+2}=\cdots =Y^{\Delta},
$$
where $Y^{\Delta}_k=\mathcal{X}_k^{D}$ if $k\leq d_{\rm min}$ and 
$Y^{\Delta}_k=Y^{\Delta}$ if $k\geq d_{\rm min}+1$.
\qed
%%%%%%%
}
%%%%%%%
\end{dfn}
%%%%%%(End of Definition 5.6)%%%%
%%%
%%%
%%%
%%%
Then by using the same argument as in 
(\ref{SS}) and Lemma \ref{Lemma: truncated}, we have the following 
two truncated spectral sequences
%%%
%%%
%%%
%%%(5.16)%%%
\begin{eqnarray}\label{TSS}
%%%%%%%%%%%%%%%%%%%
&&
\big\{\ E^{t}_{k,s}, d^{t}:E^{t}_{k,s}\to E^{t}_{k+t,s+t-1}
\big\}
\qquad \Rightarrow H_{s-k}(\Pol_D^*(S^1.\XS);\Z)
\\
\nonumber
&&\big\{\ ^{\prime}E^{t}_{k,s},\  ^{\prime}d^{t}:\ ^{\prime}E^{t}_{k,s}\to
\  ^{\prime}E^{t}_{k+t,s+t-1}
\big\}
\Rightarrow H_{s-k}(\Pol_{D+\textit{\textbf{e}}_i}^*(S^1,\XS);\Z),
%%%%
\end{eqnarray}
%%%%%%%
%if $s-k\leq 2nd-2$,
where
$$
E^{1}_{k,s}=
H^{2N(D)+k-s-1}_c(X^{\Delta}_k\setminus X^{\Delta}_{k-1};\Z),
\ 
E^{1}_{k,s}=
H^{2N(D)+k-s+1}_c(Y^{\Delta}_k\setminus Y^{\Delta}_{k-1};\Z).
$$ 
%%%%%%
\par\vspace{2mm}\par
%%%%
%%%
%%%
%%%
%%%
%%%%
%%%%(Lemma 5.7)%%
\begin{lmm}\label{lmm: E1=0}
%%%%%%%%%%%%%%%%%
%Let $r_{\rm min}=r_{\rm min}(I)$, and let $\epsilon \in\{0,1\}$.
%%%%%%%%
\begin{enumerate}
%%(i)%%%
\item[$\I$]
If $k<0$ or $k\geq \dmin +2$,
$E^1_{k,s}=\ ^{\p}E^1_{k,s}=0$ for any $s$.
%%(ii)%%
\item[$\II$]
$E^1_{0,0}=\ ^{\p}E^1_{0,0}=\Z$ and $E^1_{0,s}=\ ^{\p}E^1_{0,s}=0$ if $s\not= 0$.
%%(iii)%%%
\item[$\III$]
If $1\leq k\leq d_{\rm min}$, there are  
isomorphisms
$$
E^1_{k,s}\cong \ ^{\p}E^1_{k,s}\cong H^{2rk-s}_c(C_{k;\Sigma};\pm \Z).
$$
%%(iv)%%%
\item[$\IV$]
If $1\leq k\leq d_{min}$,
$E^1_{k,s}=\ ^{\prime}E^1_{k,s}=0$
for any $s\leq (2r_{min}(\Sigma )-1)k-1$.
%%(v)%%
\item[$\V$]
$E^1_{d_{\rm min}+1,s}=\ ^{\p}E^1_{d_{\rm min}+1,s}=0$ for any 
$s\leq (2\rmin (\Sigma)-1)d_{min}-1$.
%%%%%%%%%%
\end{enumerate}
%%%%%%
%%%%%
\end{lmm}
%%%%%%(Proof of Lemma 5.7)%%%%
%%%
%%%
%%%
%%%
\begin{proof}
%%%%%
Let us write $\rmin =\rmin (\Sigma)$.
Since the proofs of both cases are almost identical,  it suffices to prove the assertions for the case
$E^1_{k,s}$.
Since $X^{\Delta}_k=\SZ_k$ for any $k\geq d_{\rm min}+2$,
the assertions (i) and (ii) are clearly true.
Since $X^{\Delta}_k=\SZ_k$ for any $k\leq d_{\rm min}$,
(iii) easily follows from Lemma \ref{lemma: E11}.
Thus it remains to prove the assertions (iv) and (v).
\par
By using (\ref{eq: dimension C}), we see that
$2rk-s>\dim C_{k;\Sigma}$ if and only if
$s\leq (2r_{min}(\Sigma)-1)k-1$.
Thus $H^{2rk-s}_c(C_{k;\Sigma};\pm \Z)=0$ if
$s\leq (2r_{min}(\Sigma)-1)k-1$ and
the assertion (iv) follows from (iii).
 \par
 Finally consider the group
 $E^1_{\dmin +1,s}$.
%%%%
Then by Lemma \cite[Lemma 2.1]{Mo3},
%%%%%%%
\begin{align*}
%%%%%%
\dim (X^{\Delta}_{d_{\rm min}+1}\setminus 
X^{\Delta}_{d_{\rm min}})
&=
\dim (\SZ_{d_{\rm min}}\setminus \SZ_{d_{\rm min}-1})+1
=(l_{D,d_{\rm min}}+\dim C_{d_{\rm min};\Sigma})+1
\\
&=2N(D)+2d_{\rm min}-2\rmin d_{\rm min}.
\end{align*}
%%%
Since 
$E^1_{d_{\rm min}+1,s}=H_c^{2N(D)+d_{\rm min}-s}
(X^{\Delta}_{d_{\rm min}+1}\setminus X^{\Delta}_{d_{\rm min}};\Z)$
and
$2N(D)+d_{\rm min}-s>\dim (X^{\Delta}_{d+1}\setminus X^{\Delta}_d)$
$\Leftrightarrow$
$s\leq (2\rmin -1)d_{\rm min}-1$,
we see that
$E^1_{d_{\rm min}+1,s}=0$ for any $s\leq (2\rmin -1)d_{\rm min}-1$.
Hence, (v) is proved.
%%%%
\end{proof}
%%%(End of proof of Lemma 5.7)%%%%%%%
%%%%
%%%
%%%
%%%
%%%
%%%%(Definition 5.8)%%%
\begin{dfn}\label{def: connectedness}
%%%%%%%
{\rm
Define the integer $C_D(\Sigma)$ by
%%(5.17)%%
\begin{equation}
%%%
C_D(\Sigma)=
2(r_{min}(\Sigma)-1)d_{min}-1.
\end{equation}
%%%%%%
Note that $\rmin (\Sigma)\geq 2$ and $d_{min}(D)\geq 1$ for any 
$D\in \N^r$. 
Thus, we see that
$C_D(\Sigma)\geq 1$ for any $D\in \N^r$.
\qed
}
%%%%%%
\end{dfn}
%%%%%%%%(End of Definition 5.8)%%%
%%
%%
%%
%%%
%%%
%%%
%%%
%%%%(Proposition 5.9)%%
\begin{prp}\label{prp: simply connected}
%%%%%%%%%%
The space $\Pol^*_D(S^1,\XS)$ is
$C_D(\Sigma)$-connected.
Thus, the space $\Pol_D^*(S^1,\XS)$ is simply connected for any $D\in \N^r$.
%%%%%
\end{prp}
%%%%%%%%%%%
%%%(Proof of Proposition 5.9)%%%
\begin{proof}
%%%%%%%%%%%
%%
it follows from Lemma \ref{lmm: E1=0} 
 that we can easily check that
$E^1_{k,s}=0$ if $s-k\leq 
C_D(\Sigma)=2(r_{min}(\Sigma)-1)d_{min}-1.$
Hence by using the first spectral sequence given in (\ref{TSS}),
we see that $H_i(\Pol_D^*(S^1,\XS);\Z)=0$
for any $1\leq i\leq C_D(\Sigma)$.
\par
On the other hand, by using the string representation as in \cite[Appendix]{GKY1},
one can show that $\pi_1(\Pol_D^*(S^1,\XS))$ is an abelian group and
one know that
there is an isomorphism
$\pi_1(\Pol_D^*(S^1,\XS))\cong H_1(\Pol_D^*(S^1,\XS);\Z)=0.$
Thus, the space $\Pol_D^*(S^1,\XS)$ is simply connected and it follows
from the Hurewicz theorem that the space
$\Pol^*_D(S^1,\XS)$ is
$C_D(\Sigma)$-connected.
%%%
\end{proof}
%%%(End of proof of Proposition 5.9)%%%%
%%%
%%%%%%%%%%%
%%%%%%%(End of SECTION 5)%%%
%%
%%
%%
%%%
%%%
%%%
%%%
%%%
%%%%(SECTION 6)%%%%%%%%%%%%%%%%%%%%%%%%%%%%%%%%%%%%%%%%
\section{Stabilization maps}\label{section: stabilization map}
%%%%%%%%%%%%%%%%%
%%%
%%%
%%%
%%%(Definition 6.1)%%%
\begin{dfn}\label{def: stabilization map}
%%%%%%%%%%%%%%%%%%%%%%
{\rm
For  an $r$-tuple 
$D=(d_1,\cdots ,d_r)\in \N^r$
of positive integers,
let $N(D)$ denote the positive integer given by
$N(D)=\sum_{k=1}^rd_k$ as in (\ref{eq: ND}),
and
let us write
$U_D=\{w\in \C:\mbox{Re}(w)<N(D)\}$.
For each $D=(d_1,\cdots ,d_r)\in \N^r$
let us choose any any
homeomorphism
%%(6.1)%%%
\begin{equation}
%%%%%
\varphi_D: \C \stackrel{\cong}{\longrightarrow}U_D
\end{equation}
%%%
such that $\varphi_D(\R)=\{w\in \R:w<N(D)\}=(-\infty, N(D))$
and fix it.
\par
(i)
For a monic polynomial $\dis f(z)=\prod_{k=1}^d(z-\alpha_k)\in \P^d(\C)$, let
$\varphi_D(f(z))$ denote the monic polynomial of the same degree $d$ defined by
%%%(6.2)%%
\begin{equation}
\varphi_D(f(z))=\prod_{k=1}^d(z-\varphi(\alpha_k)).
\end{equation}
%%%
\par
(ii)
Now we choose a  point 
$x_D\in \C\setminus \overline{U_D} =\{w\in \C :\mbox{Re}(w)>N(D)\}$ 
freely and fix it.
Then for each $1\leq i\leq r$, 
let us recall
$$
D+\textbf{\textit{e}}_i=(d_1,\cdots ,d_{i-1},d_i+1,d_{i+1},\cdots ,d_r)
$$ 
%as in (\ref{eq: Di})
and
define {\it the stabilization map}
%%(6.3)%%
\begin{equation}
%%%
s_{D,i}:\Pol_D^*(S^1,\XS) \to \Pol_{D+\textbf{\textit{e}}_i}^*(S^1,\XS)
\quad
\mbox{ by}
\end{equation}
%%%%%%
%%(6.4)%%
\begin{equation}
%%%%
s_{D,i}(f)=
(\varphi_D(f_1(z)),\cdots ,
\varphi_D(f_i(z))(z-x_D),
\cdots ,\varphi_D(f_r(z)))
\end{equation}
%%%%
for $f=(f_1(z),\cdots ,f_r(z))\in \Pol_D^*(S^1,\XS).$
%%%
\par
(iii)
Let $D=(d_1,\cdots ,d_r)\in \N^r$ and 
$D^{\prime}=(d_1^{\p},\cdots ,d_r^{\p})\in \N^r$ be $r$-tuples of positive integers
such that
$D^{\p}-D=(d_1^{\p}-d_1,\cdots ,d_r^{\p}-d_r)\in (\N\cup
\{0\})^r\setminus\{{\bf 0}_r\}.$
%%%
Now let us choose any
point
$\textbf{\textit{x}}_D
=(x_{D,1},\cdots ,x_{D,r})\in F(\C\setminus \overline{U_D},r)$.
\par
Then
define
{\it the stabilization map}
%%(6.5)%%
\begin{equation}
s_{D,D^{\p}}:\Pol_D^*(S^1,\XS) \to \Pol_{D^{\p}}^*(S^1,\XS)
\quad
\mbox{ by}
\end{equation}
%%%
%%(6.6)%%
\begin{equation}\label{eq: adding points}
s_D(f)=(
(\varphi_D(f_1(z))(z-x_{D,1})^{d_1^{\p}-d_1},\cdots  ,\varphi_D(f_r(z))(z-x_{D,r})^{d_r^{\p}-d_r})
\end{equation}
%%%%
for $f=(f_1(z),\cdots ,f_r(z))\in \Pol_D^*(S^1,\XS).$
In particular,
when $D^{\p}=D+\e=(d_1+1,d_2+1,\cdots ,d_r+1)$,  we write
%%%
%%%()%%
%%(6.7)%%
\begin{equation}
s_D=s_{D,D^{\p}}=s_{D,D+\e}:\Pol^*_D(S^1,\XS)\to \Pol_{D+\e}^*(S^1,\XS),
\end{equation}
%%%
where we set
%%(6.8)%\
\begin{equation}
\e =(1,1,\cdots ,1)=\e_1+\e_2+\cdots +\e_r.
\end{equation} 
%%%%%%%%
}
%%%%%%
\end{dfn}
%%%(End of Definition 6.1)%%
%%%
%%%
%%%
%%%
%%%
%%%
%%%(Remark 6.2)%%%%%%%%%
\begin{rmk}
%%%%%%%
{\rm
Note that  the definition of the map $s_{D,i}$  depends on the choice of
the pair
$(\varphi_D,x_D)$, but the homotopy class of the map $s_{D,i}$ does not depend on its choice.
Similarly the definition of the map $s_{D,D^{\p}}$ also depends on the choice of the pair
$(\varphi_d,\textbf{\textit{x}}_D)$, but its homotopy class does not.
\qed
}
%%%%%%%%%%%%%%%%%%%%%%%%%
\end{rmk}
%%%(End of Remark 6.2)%%%
%%
%%
%%
%%%
%%%
%%%
%%%%%%%%%%%%%%%%%
Now
recall the definition of
the stabilization map
$$
s_{D,i}:\Pol_D^*(S^1,\XS)\to \Pol_{D+\textbf{\textit{e}}_i}^*(S^1,\XS).
$$
%%%
%%
By using the choice of the point $x_D\in \C \setminus \overline{U_D}$,
one can easily see that it extends to an open embedding
%%
%%%(6.9)%%
\begin{equation}\label{equ: sssd}
%%%%%%%%%
\overline{s}_{D,i}:\C\times 
\Pol_D^*(S^1,\XS)\to
\Pol_{D+\textbf{\textit{e}}_i}^*(S^1,\XS).
\end{equation}
%%%%%%%% 
%%%
It also naturally extends to an open embedding
$\tilde{s}_{D,i}:\P^D\to \P^{D_i}$ and  by  restriction  we obtain an open embedding
%%%%
%%%(6.10)%%
\begin{equation}\label{equ: open embedding}
%%%%%%%
\tilde{s}_{D,i}:\C\times \Sigma_D\to 
\Sigma_{D+\textbf{\textit{e}}_i}.
\end{equation}
%%%
Since one-point compactification is contravariant for open embeddings,
this map induces a map
$(\tilde{s}_{D+\textbf{\textit{e}}_i})_+:(\Sigma_{D_i})_+
\to
(\C\times \Sigma_D)_+=S^{2}\wedge \Sigma_{D+}.$
\par
Note that there is a commutative diagram
%%%
%%%%%(6.11)%%
\begin{equation}\label{eq: diagram}
%%%%%%%%%%%%%
\begin{CD}
\tilde{H}_k(\Pol_D^*(S^1,\XS);\Z) @>{s_{D,i}}_*>>
\tilde{H}_k(\Pol_{D+\textbf{\textit{e}}_i}^*(S^1,\XS);\Z)
\\
@V{Al}V{\cong}V @V{Al}V{\cong}V
\\
H^{2N(D)-k-1}_c(\Sigma_D;\Z)
@>{\tilde{s}_{{D}_i+}}^{\ *}>>
H^{2N(D)-k+1}_c(\Sigma_{D+\textbf{\textit{e}}_i};\Z)
\end{CD}
\end{equation}
%%%% 
where $Al$ is the Alexander duality isomorphism and
 ${\tilde{s}_{{D}_i+}}^{\ *}$ denotes the composite of 
%homomorphisms 
the suspension isomorphism with the homomorphism
${(\tilde{s}_{D+})^*}$,
$$
H^{2N(D)-k-1}_c(\Sigma_D;\Z)
\stackrel{\cong}{\rightarrow}
H^{2N(D)-k+1}_c
(\C\times \Sigma_D;\Z)
\stackrel{(\tilde{s}_{{D}_i+})^*}{\longrightarrow}
H^{2N(D)-k+1}_c(\Sigma_{D+\textbf{\textit{e}}_i};\Z).
$$
%%%%%
By the universality of the non-degenerate simplicial resolution
(\cite[pages 286-287]{Mo2}), 
the map $\tilde{s}_{D,i}$ also naturally extends to a filtration preserving open embedding
%%%()%%%%%
%\begin{equation}\label{equ: flitr-preserve map}
$$
\tilde{s}_{D,i}:\C \times \SZ \to \mathcal{X}^{D+\textbf{\textit{e}}_i}
$$
%\end{equation}
%%%%%%%%%%%%%%
between non-degenerate simplicial resolutions.
This  induces a filtration preserving map
$(\tilde{s}_{D,i})_+:\mathcal{X}^{D+\textbf{\textit{e}}_i}_+\to (\C \times \SZ)_+
=S^{2}\wedge \SZ_+$,
and we obtain the homomorphism of spectral sequences
%%%
%%%(6.12)%%%%
\begin{equation}\label{equ: theta1}
%%%%%%%%%%%%
\{ \tilde{\theta}_{k,s}^t:E^{t;D}_{k,s}\to 
E^{t;D+\textbf{\textit{e}}_i}_{k,s}\},
\quad
\mbox{ where}
%%%%%
\end{equation}
%%%
%%%%%%%%
\begin{eqnarray*}
%%%%%%%
\big\{E^{t;D}_{k,s}, \tilde{d}^{t}:
E^{t;D}_{k,s}\to E^{t;D}_{k+t,s+t-1}
\big\}
\quad &\Rightarrow &
 H_{s-k}(\Pol_{D}^*(S^1,\XS);\Z),
\\
\big\{E^{t;D+\textbf{\textit{e}}_i}_{k,s}, \tilde{d}^{t}:
E^{t;D+\textbf{\textit{e}}_i}_{k,s}\to 
E^{t;D+\textbf{\textit{e}}_i}_{k+t,s+t-1}
\big\}
\quad&\Rightarrow& 
 H_{s-k}(\Pol_{D+\textbf{\textit{e}}_i}^*(S^1,\XS);\Z),
% \\
% E^{1;D}_{k,s}&=&
%H_c^{2N(D)+k-1-s}(\SZ_k\setminus \SZ_{k-1}),
%\\
%E^{1;D_i}_{k,s}
%&=&
%H_c^{2N(D_i)+k-1-s}
%(\mathcal{X}_k^{D_i}\setminus 
%\mathcal{X}_{k-1}^{D_i}).
%%%
\end{eqnarray*}
%%%
and the groups
$E^{1;D}_{k,s}$ and 
$E^{1;D_i}_{k,s}$ are given by
$$
\begin{cases}
E^{1;D}_{k,s}=
H_c^{2N(D)+k-1-s}(\SZ_k\setminus \SZ_{k-1};\Z),
\\ 
E^{1;D+\textbf{\textit{e}}_i}_{k,s}
=
H_c^{2N(D)+k+1-s}
(\mathcal{X}_k^{D+\textbf{\textit{e}}_i}\setminus 
\mathcal{X}_{k-1}^{D+\textbf{\textit{e}}_i};\Z).
\end{cases}
$$
%%%
%%%
%%%
%%%
%%%
%%%
%%%(Lemma 6,3)%%%
\begin{lmm}\label{lmm: E1}
%%%%%%%%%%%%%%%%%
If $1\leq i\leq r$ and
$0\leq k\leq d_{\rm min}$, 
$\tilde{\theta}^1_{k,s}:E^{1;D}_{k,s}\to 
E^{1;D+\textbf{\textit{e}}_i}_{k,s}$ is
an isomorphism for any $s$.
\end{lmm}
%%%%%%%%%%%%%%%%
\begin{proof}
%%%%%%%%%%%%%%%%
Since the case $k=0$ is clear,
suppose that $1\leq k\leq d_{\rm min}$.
It follows from the proof of Lemma \ref{lemma: vector bundle*}
that there is a homotopy commutative diagram of affine vector bundles
$$
\begin{CD}
%\C^n\times 
\C\times (\SZ_k\setminus\SZ_{k-1}) @>>> C_{k;\Sigma}
\\
@VVV \Vert @.
\\
\mathcal{X}^{D+\textbf{\textit{e}}_i}_k\setminus 
\mathcal{X}^{D+\textbf{\textit{e}}_i}_{k-1} @>>> C_{k;\Sigma}
\end{CD}
$$
Hence, %by the naturality of Thom isomorphisms 
we have 
a commutative diagram
$$
\begin{CD}
E^{1,D}_{k,s} @>>\cong> H^{2rk-s}_c(C_{k;\Sigma};\pm \Z)
\\
@V{\tilde{\theta}_{k,s}^1}VV \Vert @.
\\
E^{1,D+\textbf{\textit{e}}_i}_{k,s} @>>\cong> H^{2rk-s}_c(C_{k;\Sigma};\pm \Z)
\end{CD}
$$
%where %$r_{\rm min}=r_{\rm min}(I)$ and 
%$T$ denotes the Thom isomorphism.
and the assertion follows.
%%%%%%%%
\end{proof}
%%(End of proof of Lemma 6.3)%%%%
%%%%%%%%%%%%%%%
%%
%%%
%%%
%%%
%%%
%%%
%%%
%%%
%%%
Now we consider the spectral sequences induced by 
truncated simplicial resolutions.
\par
%%%%%%%%%%
%%%%
By using Lemma \ref{Lemma: truncated} and the same method
as in \cite[\S 2 and \S 3]{Mo3} (cf. \cite[Lemma 2.2]{KY4}), 
we obtain the following {\it  truncated spectral sequences}
%%%%()%%%%%%%%
\begin{eqnarray*}\label{equ: spectral sequ2}
%%%%%%%%%%%%%%%%
\big\{E^{t}_{k,s}, d^{t}:E^{t}_{k,s}\to 
E^{t}_{k+t,s+t-1}
\big\}
&\Rightarrow& H_{s-k}(\Pol_{D}^*(S^1,\XS);\Z),
%%%
\\
%%%
\big\{\ 
^{\p}E^{t}_{k,s},\  d^{t}:\ ^{\p}E^{t}_{k,s}\to 
\  ^{\p}E^{t}_{k+t,s+t-1}
\big\}
&\Rightarrow& H_{s-k}(\Pol_{D+\textbf{\textit{e}}_i}^*(S^1,\XS);\Z),
%%%%
\end{eqnarray*}
%%%
where
$$
E^{1}_{k,s}=H_c^{2N(D)+k-1-s}(X^{\Delta}_k\setminus X^{\Delta}_{k-1};\Z),
\ 
^{\p}E^{1}_{k,s}=H_c^{2(N(D)+k+1-s}(Y^{\Delta}_k\setminus Y^{\Delta}_{k-1};\Z).
$$
%%%
By the naturality of truncated simplicial resolutions,
the filtration preserving map
$\tilde{s}_{D,i}:\C\times \SZ \to \mathcal{X}^{D_i}$
  gives rise to a natural filtration preserving map
%%()%%
%\begin{equation}
%%%%%%%%
$\tilde{s}_{D,i}^{\p}:\C\times X^{\Delta} \to Y^{\Delta}$
%\end{equation}
%%%%%%%% 
which, in a way analogous to  (\ref{equ: theta1}), induces
a homomorphism of spectral sequences 
%%%%%%%%%%%
%%%(6.13)%%%%
\begin{equation}\label{equ: theta2}
%%%%%%%%%%%%
\{ \theta_{k,s}^t:E^{t}_{k,s}\to \ ^{\p}E^{t}_{k,s}\}.
\end{equation}
%%%%%%%%%%
%%%
%%%
%%%
%%%
%%%(Lemma 6.4)%%
\begin{lmm}\label{lmm: E2}
%%%%%%%%%%%%%%%%%
If $0\leq k\leq d_{\rm min}$, 
$\theta^1_{k,s}:E^{1}_{k,s}\to \ ^{\p}E^{1}_{k,s}$ is
an isomorphism for any $s$.
\end{lmm}
%%%%%%%%%%%%%%%%
\begin{proof}
Since $(X^{\Delta}_k,Y^{\Delta}_k)=(\SZ_k,\mathcal{X}^{D_i}_k)$ for $k\leq \dmin $,
the assertion follows from Lemma \ref{lmm: E1}.
\end{proof}
%%%%%%%%%%%%%%
%%
%%
%%
%%
%%(Stability Theorem of stabilization map)%%
%%
%%
%%%%%(Theorem 6.5)%%
\begin{thm}\label{thm: III}
%%%%%%%%%%%%%%%%%%%
For each $1\leq i\leq r$ and $D=(d_1,\cdots ,d_r)\in \N^r$,
the stabilization map
$$
s_{D,i}:\Pol_D^*(S^1,\XS)\to \Pol_{D+\textbf{\textit{e}}_i}^*(S^1,\XS)
$$
is a homotopy equivalence through dimension $d(D,\Sigma)$, where
the number $d(D,\Sigma)$ is given by
$d(D,\Sigma)=2(\rmin (\Sigma )-1)\dmin -2$ as in (\ref{eq: rmin}).
%%%%%%%%%%%%%%%%%
\end{thm}
%%%%%%%%%%%%%%%%%%
%%(Proof of Theorem 6.5)%%%%%
\begin{proof}
%%%%%%%%%%%%%%%%%
We write $\rmin =\rmin (\Sigma)$.
Since  
$\Pol_D^*(S^1,\XS)$ and $\Pol_{D+\textbf{\textit{e}}_i}^*(S^1,\XS)$ 
are simply connected,
it suffices to show that the map $s_{D,i}$ is a homology equivalence through dimension
$d(D,\Sigma)$.

For this purpose, consider the homomorphism
$\theta_{k,s}^t:E^{t}_{k,s}\to \ ^{\p}E^{t}_{k,s}$
of truncated spectral sequences given in (\ref{equ: theta2}).
%%%%
By using the commutative diagram (\ref{eq: diagram}) 
and the comparison theorem for spectral sequences, 
it suffices to prove that the positive integer $d(D,\Sigma)$ 
has the 
following property:
%%%
\begin{enumerate}
\item[$(*)$]
$\theta^{\infty}_{k,s}$
is  an isomorphism for all $(k,s)$ such that $s-k\leq d(D,\Sigma)$.
\end{enumerate}
%%
%%
%%(The case r<0 or r \leq d+2)%%%%%
By Lemma \ref{lmm: E1=0},
$E^1_{k,s}=\ ^{\p}E^1_{k,s}=0$ if
$k<0$, or if $k\geq \dmin +2$, 
%or $1\leq k\leq \dmin$ with $s\leq (2\rmin -1)k-1$,
%%
or if $k=\dmin +1$ with 
$s\leq (2\rmin -2)\dmin -1$.
Since $(2\rmin -1)\dmin -1-(\dmin+1)=(2\rmin -2)\dmin -2
=d(D,\Sigma)$,
we  see that:
%()%%
%%%%%%%
%%%%%%%
\begin{enumerate}
%%%%%%%
\item[$(*)_1$]
if $k< 0$ or $k\geq \dmin +1$,
$\theta^{\infty}_{k,s}$ is an isomorphism for all $(k,s)$ such that
$s-k\leq d(D,\Sigma)$.
\end{enumerate}
\par
Next, we assume that $0\leq k\leq \dmin$, and investigate the condition that
$\theta^{\infty}_{k,s}$  is an isomorphism.
%%%
Now recall that $E^1_{k,s}=\ ^{\p}E^1_{k,s}=0$ if
$1\leq k\leq \dmin$ with $s\leq (2\rmin -1)k-1$ or if $k=0$ with $s\not=0$.
(by Lemma \ref{lmm: E1=0}).

%Then
Note that the groups $E^1_{k_1,s_1}$ and $^{\p}E^1_{k_1,s_1}$ are not known for
%%(S_1)%%%
$(u,v)\in\mathcal{S}_1=
\{(\dmin+1,s)\in\Z^2:s\geq (2\rmin -1)\dmin \}$.
%%%%
%%%%
By considering the differentials $d^1$'s of
$E^1_{k,s}$ and $^{\p}E^1_{k,s}$,
%$d^1:\E^1_{k,s}\to \E^{1}_{k+1,s}$
%and
%$d^1:\Ed^1_{k,s}\to \Ed^{1}_{k+1,s}$
and applying Lemma \ref{lmm: E2}, we see that
$\theta^2_{k,s}$ is an isomorphism if
$(k,s)\notin \mathcal{S}_1 \cup \mathcal{S}_2$, where
%%%
%%(S_2)%%%
$$
\mathcal{S}_2=
\{(u,v)\in\Z^2:(u+1,v)\in \mathcal{S}_1\}
=\{(\dmin ,v)\in \Z^2:v\geq (2\rmin -1)\dmin\}.
$$
%%%%
%%%
%%%%
A similar argument  shows that
$\theta^3_{k,s}$ is an isomorphism if
%%(S_3)%%
$(k,s)\notin \bigcup_{t=1}^3\mathcal{S}_t$, where
$\mathcal{S}_3=\{(u,v)\in\Z^2:(u+2,v+1)\in \mathcal{S}_1\cup
\mathcal{S}_2\}.$
%%%%%
%\par
%%%%
Continuing in the same fashion,
considering the differentials
$d^t$'s on $E^t_{k,s}$ and $^{\p}E^t_{k,s}$
and applying the inductive hypothesis,
%Lemma \ref{lmm: E2}, 
we  see that $\theta^{\infty}_{k,s}$ is an isomorphism
if $\dis (k,s)\notin \mathcal{S}:=\bigcup_{t\geq 1}\mathcal{S}_t
=\bigcup_{t\geq 1}A_t$,
where  $A_t$ denotes the set
%%%%(Def. of A_t)%%%%%%%
$$
A_t=
\left\{
\begin{array}{c|l}
 &\mbox{ There are positive integers }l_1,\cdots ,l_t
\mbox{ such that},
\\
(u,v)\in \Z^2 &\  1\leq l_1<l_2<\cdots <l_t,\ 
u+\sum_{j=1}^tl_j=\dmin +1,
\\
& \ v+\sum_{j=1}^t(l_j-1)\geq (2\rmin -1)\dmin
\end{array}
\right\}.
$$
%%%%%%%% 
Note that 
if this set was empty for every $t$, then, of course, the conclusion of 
Theorem \ref{thm: III} would hold in all dimensions (this is known to be false in general). 
%\par
If $\dis A_t\not= \emptyset$, it is easy to see that
%%%
$$
a(t)=\min \{s-k:(k,s)\in A_t\}=
(2\rmin -1)\dmin -(\dmin +1)+t
=d(D,\Sigma)+t+1.
$$
%%%
Hence, we obtain that
$\min \{a(t):t\geq 1,A_t\not=\emptyset\}=d(D,\Sigma)+2.$
%%%
 Since $\theta^{\infty}_{k,s}$ is an isomorphism
for any $(k,s)\notin \bigcup_{t\geq 1}A_t$ for each $0\leq k\leq \dmin$,
we have the following:
%%(*2)%%%
\begin{enumerate}
%%%%%%%%
\item[$(*)_2$]
If $0\leq k\leq \dmin$,
$\theta^{\infty}_{k,s}$ is  an isomorphism for any $(k,s)$ such that
$s-k\leq  d(D,\Sigma)+1.$
\end{enumerate}
%%%%%%
Then, by $(*)_1$ and $(*)_2$, we know that
$\theta^{\infty}_{k,s}:E^{\infty}_{k,s}\stackrel{\cong}{\rightarrow}
\ ^{\p}E^{\infty}_{k,s}$ is an isomorphism for any $(k,s)$
if $s-k\leq d(D,\Sigma)$, and this completes the proof.
%%%%
%%%%
\end{proof}
%%(End of proof of Theorem 6.5)%%
%%%
%%%
%%%
%%%
%%%%%(Theorem 6.6)%%
\begin{thm}\label{thm: IV}
%%%%%%%%%%%%%%%%%%%%%%
For each $D=(d_1,\cdots ,d_r)\in \N^r$ and
$D^{\p}=(d_1^{\p},\cdots ,d_r^{\p})\in \N^r$ with
$D^{\p}-D\in (\Z_{\geq 0})^r\setminus
\{{\bf 0}_r\}$,
the stabilization map
$$
s_{D,D^{\p}}:\Pol_D^*(S^1,\XS)\to \Pol_{D^{\p}}^*(S^1,\XS)
$$
is a homotopy equivalence through dimension $d(D,\Sigma)$, where
the number $d(D,\Sigma)$ is given by
$d(D,\Sigma)=2(\rmin (\Sigma )-1)\dmin -2$ as in (\ref{eq: rmin}).
%%%%%%%%%%%%%%%%%
\end{thm}
%%(Proof of Theorem 6.6)%%%
\begin{proof}
%%%%%
%Let us write
%$\Pol^*_{D_*}=\Pol^*_{D_*}(S^1,\XS)$ for $D_*\in \N^r$.
For each $1\leq k\leq r$, let us write
$m_k=d_k^{\p}-d_k$.
Let
$D=D_0=(d_1,\cdots ,d_r)$,
and for each $1\leq k\leq r$ let us write
\begin{align*}
D_k&=
D+\sum_{i=1}^km_i\e_i
=
(d_1+m_1,d_2+m_2,\cdots ,d_k+m_k,
d_{k+1},d_{k+2},\cdots ,d_r)
\\
&=
(d_1^{\p},d_2^{\p},\cdots ,d_k^{\p},d_{k+1},d_{k+2},\cdots ,d_r).
\end{align*}
Note that $D_r=(d_1^{\p},\cdots ,d_r^{\p})=D^{\p}$ for $k=r$.
Now for each $1\leq i<r$, define the map
$
f_i:\Pol^*_{D_{i-1}}(S^1,\XS)\to \Pol^*_{D_i}(S^1,\XS)
$
by the composite of stabilization maps
$$
f_i=(s_{D_{i-1}+(m_i-1)\e_i,i})\circ
(s_{D_{i-1}+(m_i-2)\e_i,i})\circ
\cdots
\circ
(s_{D_{i-1}+2\e_i,i})\circ
(s_{D_{i-1}+\e_i,i})\circ (s_{D_{i-1},i}).
$$
Since the stabilization map $s_{D_{i-1}+k\e_i,i}$ 
is a homotopy equivalence through dimension
$d(D_{i-1}+k\e_i,\Sigma)$
by
Theorem \ref{thm: III} and
$d(D_{i-1}+k\e_i,\Sigma)\leq d(D_{i-1}+(k+1)\e_i,\Sigma)$ for each
$0\leq k<r$, the map $f_i$ is a homotopy equivalence through dimension
$d(D_{i-1},\Sigma)$.
Furthermore, we can easily see that
$$
s_{D,D^{\p}}=f_{r-1}\circ f_{r-2}\circ \cdots \circ f_2\circ f_1
\quad
\mbox{(up to homotopy equivalence)}
$$
Since
$$
d(D,\Sigma)=d(D_0,\Sigma)\leq d(D_1,\Sigma)\leq d(D_2,\Sigma)\leq \cdots 
\leq d(D_r,\Sigma)=d(D^{\p},\Sigma),
$$
the map 
$s_{D,D^{\p}}$ is a homotopy equivalence through dimension
$d(D,\Sigma)$.
%%%
\end{proof}
%%(End of proof of Theorem 6.6)%%%%%

%%
%%%%(Definition 6.7)%%%
\begin{dfn}\label{dfn: the map jD}
%%%%%%%%%%%%%%%%%%%%%%%%
{\rm
Let $D=(d_1,\cdots ,d_r)\in \N^r$ be an $r$-tuple of positive integers.
If the condition $\sum_{k=1}^rd_k\textbf{\textit{n}}_k={\bf 0}_n$ is satisfied,
%the assertion follows from Theorem \ref{thm: I}.
the map
%%%%%%
%%%%%(6.14)%%%%%%
\begin{equation}
%%%%%%%
j_D:\Pol_D^*(S^1,\XS)\to \Omega \mathcal{Z}_{\KS}
\end{equation}
%%%%%%%%%%%%%%
%%
%%
was already defined in (\ref{eq: lift}), and we shall
define it when
$\sum_{k=1}^rd_k\textbf{\textit{n}}_k\not= {\bf 0}_n$.
\par
Suppose that $\sum_{k=1}^rd_k\textbf{\textit{n}}_k\not= {\bf 0}_n$.
By the assumption (\ref{equ: homogenous}.1),
there is an $r$-tuple
$D_*=(d_1^*,\cdots ,d_r^*)\in \N^r$ such that
$\sum_{k=1}^rd_k^*\textbf{\textit{n}}_k= {\bf 0}_n$.
If we choose a sufficiently large integer $m_0\in \N$, then 
the following condition holds:
%%(6.15)%%%%%%%
\begin{equation}
%%%%%%%%%%%%%%
d_k<m_0d_k^* \quad \mbox{ for each }\ 1\leq k\leq r.
\end{equation}
%%%
Now we set $D_0=m_0D_*=(m_0d^*_1,m_0d^*_2,\cdots ,m_0d^*_r)$
and
define the map
%%%(6.16)%%
\begin{equation}\label{eq:jDD}
%%%%%%%
j_D:\Pol_D^*(S^1,\XS)\to \Omega \mathcal{Z}_{\KS}
\end{equation}
%%%%%
as the composite of maps
%%
%%
%%
%%(6.17)%%
\begin{equation}\label{eq:jDDD}
%%%
j_D=j_{D_0}\circ s_{D,D_0}:
\Pol_{D}^*(S^1,\XS)
\stackrel{s_{D,D_0}}{\longrightarrow}
\Pol_{D_0}^*(S^1,\XS)
\stackrel{j_{D_0}}{\longrightarrow}
\Omega \mathcal{Z}_{\KS}.
\end{equation}
%%%%%%%
Similarly, for each $r$-tuple $D=(d_1,\cdots ,d_r)\in \N^r$ of positive integers, 
define the map
%%(6.18)%%
\begin{equation}\label{eq: iDD}
%%%%
i_D:\Pol^*_D(S^1,\XS)\to \Omega \XS
%%%
\end{equation}
%%%%%
as the composite of maps}
%%(6.19)%%
\begin{equation}
%%%%%%%%%%
i_D=q_{\Sigma}\circ j_D.
\end{equation}
%%%%
%%%%%%%%%
\end{dfn}
%%%%%(End of Definition 6.17)%%%%%
%%%%(Remark 6.18)%%%
\begin{rmk}
%%%%%%%%%%%%%%%%%%%%
{\rm
Note that the maps $j_D$ and $i_D$ depend on the choice of the pair $(D_*,m_0)$,
but one can show that their homotopy classes do not depend on this pair.
\qed
%%%
}
%%%%%%%%%
\end{rmk}
%%%(End of Remark 6.18)%%%%
%%%
%%
%%(End of SECTION 6)%%%%%
%%%%%%%%%%%%%%
%%
%%
%%
%%
%%
%%
%%%(SECTION 7)%%%%%%%%%%%%%%%%%%%%%%%%%%%%%%%%%%%%
\section{Scanning maps}\label{section: scanning map}
%%%%%%%%%%%%%%%%%%%%%%%%%%%%%%%%%%%%%%%%%%%%%%%%%%
%

In this section we consider configuration spaces and the scanning map.
In particular, we recall the definition of the horizontal scanning map and prove that it is 
a homotopy equivalence.

%%%%
%%(Definition 7.1)%%%
\begin{dfn}\label{dfn: SP}
%%%%%%%%%%%%%%%%%%%%%
{\rm
 For a positive integer $d\geq 1$ and a based space $X,$ let
 $\SP^d(X)$ denote {\it the $d$-th symmetric product} of $X$ defined as
the orbit space 
%%%%%%%%
%%(7.1)%%
\begin{equation}
%%%%%%%%%%
\SP^d(X)=X^d/S_d,
\end{equation}
where the symmetric group $S_d$ of $d$ letters acts on the $d$-fold
product $X^d$ in the natural manner.
\qed
}
\end{dfn}
%%%%%(End of Definition 7.1)%%
%%%
%%
%%

%%%%%%%%(Remark 7.2)%%%%%
\begin{rmk}
%%%%%%%%%%%%%%%%%%%%%%%%%
{\rm
%%(i)%%
(i)
%%%%%%
Note that
an element $\eta\in \SP^d(X)$ may be identified with a formal linear
combination
%%(7.2)%%%
\begin{equation}
%%%
\eta =\sum_{k=1}^sn_kx_k,
\end{equation}
%%%%%%%%%%
where
$\{x_k\}_{k=1}^s\in C_s(X)$ 
and $\{n_k\}_{k=1}^s\subset \N$ with
$\sum_{k=1}^sn_k=d$.
%%%
\par
(ii)
For example, when $X=\C$, we have the natural homeomorphism
%%%(7.3)%%
\begin{equation}\label{eq: identification}
%%%
\psi_d:
\P^d(\C)\stackrel{\cong}{\longrightarrow}\SP^d(\C)
\end{equation}
%%%
given by using the above identification
%%%(7.4)
\begin{equation}\label{eq:homeo}
%%%%
\psi_d(f(z))=
\psi_d(\prod_{k=1}^s(z-\alpha_k)^{d_k})= \sum_{k=1}^sd_k\alpha_k
\end{equation}
%%%%%%%
for $f(z)=\prod_{k=1}^s(z-\alpha_k)^{d_k}\in \P^d(\C)$.
\qed
}
%%%%%%%%%
\end{rmk}
%%%%%(End of Remark 7.2)%%%%%%%%%

%%%%%(Definition 7.3)%%%%%%%%%
\begin{dfn}
%%%%%%%%%%%%%%%%%%%%%%%%%%%%%
{\rm
(i) For a subspace $A\subset X$, let $\SP^d(X,A)$ denote the quotient space
%%(7.5)%%
\begin{equation}
%%%%%
\SP^d(X,A)=\SP^d(X)/\sim
%%%
\end{equation}
%%%%%
where the equivalence relation $\sim$ is defined by
$$
\xi\sim \eta \Leftrightarrow
\xi \cap (X\setminus A)=\eta \cap (X\setminus A)
\quad
\mbox{ for }\xi,\eta\in \SP^d(X).
$$
Thus, the points of $A$ are ignored.
When $A\not=\emptyset$, by adding a point in $A$ we have the natural
inclusion
$$
\SP^d(X,A)\subset \SP^{d+1}(X,A).
$$
Thus, when $A\not=\emptyset$, one can define the space
$\SP^{\infty}(X,A)$ by the union
%%(7.6)%%
\begin{equation}
%%%%%%
\SP^{\infty}(X,A)=\bigcup_{d\geq 0}\SP^d(X,A),
%%%
\end{equation}
%%%%%
where we set $\SP^0(X,A)=\{\emptyset\}$ and
$\emptyset$ denotes the empty configuration.
\par
(ii)
From now on, we always assume that $X\subset \C$.
For each $r$-tuple $D=(d_1,\cdots ,d_r)\in \N^r$,
let $\SP^D(X)=\prod_{i=1}^r\SP^{d_i}(X)$, and
define the spaces
$\ES_D(X)$ and $E^{\Sigma}_D(X)$  
by
%%()%%%%%%
\begin{align*}
%%%%%%%%% 
\ES_D(X)
&=
\{(\xi_1,\cdots ,\xi_r)
\in \SP^D(X)
%\prod_{i=1}^r\SP^{d_i}(X)
: (\bigcap_{k\in \sigma}\xi_k)\cap \R=\emptyset 
\ \mbox{for any }\sigma\in I(\KS)\},
%(*)_{\Sigma ,\R}\},
\\
E^{\Sigma}_D(X)
&=
\{(\xi_1,\cdots ,\xi_r)
\in \SP^D(X)
%\prod_{i=1}^r\SP^{d_i}(X)
: \ \bigcap_{k\in \sigma}\xi_k=\emptyset
\ \mbox{for any }\sigma\in I(\KS)\}.
%(*)_{\Sigma}\},
%\end{cases}
\end{align*}
%%%%%
%Note that
%$E^{\Sigma}_D(X)\subset \ES_D(X)$ and
%that
%$\ES_D(X)=\prod_{i=1}^r\SP^{d_i}(X)$ if $X\cap \R =\emptyset$.
%where the condition $(*)_{\Sigma,\R}$ and $(*)_{\Sigma}$ given by
%%%%%%%
%\begin{enumerate}
%\item[$(*)_{\Sigma,\R}$:]
%$\dis (\bigcap_{k\in \sigma}\xi_k)\cap \R=\emptyset$
%for any $\sigma \in I(\KS)$.
%%%%
%\item[$(*)_{\Sigma}$:]
%$\dis (\bigcap_{k\in \sigma}\xi_k)=\emptyset$
%for any $\sigma \in I(\KS)$.
%\end{enumerate}
%%%%%%
\par
(iii)
When $A\subset X$ is a subspace, define
an equivalence relation \lq\lq$\sim$\rq\rq \ on
the space $\ES_D(X)$ 
(resp. the space $E^{\Sigma}_D(X)$) 
by
$$
(\xi_1,\cdots ,\xi_r)\sim
(\eta_1,\cdots ,\eta_r)
\quad
\mbox{if }\quad
\xi_i \cap (X\setminus A)=\eta_i \cap (X\setminus A)
\quad
\mbox{for each }1\leq j\leq r.
$$
Let 
$\ES_D(X,A)$ 
and $E^{\Sigma}_D(X,A)$ 
be the quotient spaces
defined by
%%()%%
\begin{equation*}
\ES_D(X,A)=\ES_D(X)/\sim ,
\mbox{ and }\ 
E^{\Sigma}_D(X,A)=E^{\Sigma}_D(X)/\sim .
\end{equation*}
%%%%
When $A\not=\emptyset$,  by
adding points in $A$ we have  natural inclusions
%%()%%
\begin{equation*}\label{eq: D+k}
%%%
\ES_{D}(X,A)\subset \ \ES_{D+\e_i}(X,A),
\quad
E^{\Sigma}_D(X,A)\subset \ E^{\Sigma}_{D+\e_i}(X,A)
\end{equation*}
for each $1\leq i\leq r,$
where
$D+\e_i=(d_1,\cdots,d_{i-1},d_i+1,d_{i+1},\cdots ,d_r)$
as in (\ref{eq: Di}).
Thus,  when $A\not=\emptyset$, one can define the spaces
$\ES (X,A)$ 
and
$E^{\Sigma}_{}(X,A)$ 
by the unions
%%(7.7)%%
\begin{equation}
%%%
%\begin{cases}
\ES_{}(X,A) =\bigcup_{D\in \N^r}
\ES_{D}(X,A),
\quad
E^{\Sigma}_{}(X,A) =\bigcup_{D\in \N^r}
E^{\Sigma}_{D}(X,A),
%\end{cases}
\end{equation}
%%%%%%%%%%%
where the empty configuration
$(\emptyset ,\cdots ,\emptyset)$ is the base-point of
$\ES (X,A)$ or $E^{\Sigma}(X,A)$.
\qed
}
%%%%
\end{dfn}
%%(End of Definition 7.3)%%%
%%
%%
%%
%%
%%%
%%%
%%%
%%%
%%(Remark 7.4)%%
\begin{rmk}\label{rmk: ESigma}
%%%%
{\rm
(i)
If $X\subset \R$,  the following equalities hold:
%%()%%%%%
\begin{equation*}
%%%%%%%%%%
\ES_D (X)=E^{\Sigma}_D(X)
\quad
\mbox{and}
\quad
\ES (X,A)=E^{\Sigma}(X,A).
\end{equation*}
%%%%%
\par
(ii)
Let $D=(d_1,\cdots ,d_r)\in \N^r$.
Then 
by using the identification
(\ref{eq: identification})
we easily obtain the homeomorphism 
%%(7.8)%%
\begin{equation}\label{eq: Pol=E}
%%%
\begin{CD}
\Pol^*_D(S^1,\XS) @>\Psi_D>\cong> \ES_D(\C)
\\
(f_1(z),\cdots ,f_r(z))
@>>>
(\psi_{d_1}(f_1(z)),\cdots ,\psi_{d_r}(f_r(z)))
\end{CD}
\end{equation}
%%%%%%
\par
(iii)
Now 
let $\varphi_D:\C \stackrel{\cong}{\longrightarrow}U_D$
and
$\textbf{\textit{x}}_D=(x_{D,1},\cdots ,x_{D,r})\in F(\C\setminus \overline{U_D},r)$ 
be the homeomorphism and the point for
defining the stabilization map $s_{D}$
in Definition \ref{def: stabilization map}.
Then define the map
%%(7.9)%%
\begin{equation}
%%%%%%
s_{D}^{\Sigma}:\ES_D(\C) \to \ES_{D+\e}(\C)
\qquad
\mbox{ by}
%%%
\end{equation}
%%%
%%()%%
\begin{equation*}\label{eq: stab-sigma}
s_{D}^{\Sigma}(\xi_1,\cdots ,\xi_r)=
(\varphi_D(\xi_1)+x_{D,1},\cdots
,\varphi_D(\xi_r)+x_{D,r})
%%%%%%%%
\end{equation*}
%%%%%%%%
for
$(\xi_1,\cdots ,\xi_r)\in \ES_D$, where we write
$\varphi_D(\xi)=\sum_{k=1}^sn_k\varphi_D(x_k)$
if $\xi=\sum_{k=1}^sn_kx_k\in \SP^d(\C)$
and $(n_k,x_k)\in \N\times \C$ with $\sum_{k=1}^sn_k=d$.
%%%%
\par
%%%%
Then by using the above homeomorphism (\ref{eq: Pol=E}),
we have the following commutative diagram
%%%
%%%(7.10)%%
\begin{equation}\label{CD: stabilization}
%%%
\begin{CD}
\Pol_D^*(S^1,\XS) @>s_{D}>> \Pol_{D+\e}^*(S^1,\XS)
\\
@V{\Psi_D}V{\cong}V @V{\Psi_{D+\e}}V{\cong}V
\\
\ES_D(\C) @>s_{D}^{\Sigma}>> \ES_{D+\e}(\C)
\end{CD}
\end{equation}
%%%%%%%
%%%%%%%
\par
(iv)
Note that $\ES_D(\C)$ is path-connected.
Indeed, for any two points 
$
\xi_0,\xi_1\in E^{\Sigma}_D(\C),
$
one can construct
a path $\omega :[0,1]\to \ES_D(\C)$ such that
$\omega (i)=\xi_i$ for $i\in\{0,1\}$
by using the string representation 
used in  \cite[\S Appendix]{GKY1}. 
Thus the space $\Pol_D^*(S^1,\XS)$ is also path-connected. 
\qed
}
\end{rmk}
%%%(End of Remark 7.4)%%
%%%
%%%
%%%
%%%
%%%(Definition 7.5)%%%
\begin{dfn}
%%%%%%%%%%%%%%%%%%%%%
{\rm
%%%
Define the stabilized space $\Pol_{D+\infty}^{\Sigma}$ by
the colimit
%%(7.11)%%
\begin{equation}\label{eq: stabilized space}
%%%%%%%
\Pol_{D+\infty}^{\Sigma}=
\lim_{k\to\infty}
\Pol^*_{D+k\e}(S^1,\XS),
\end{equation}
%%%%
where  the colimt is taken from the family of stabilization maps}
$$
\{s_{D+k\e}:\Pol_{D+k\e}^*(S^1,\XS)\to \Pol_{D+(k+1)\e}^*(S^1,\XS)\}_{k\geq 0}
\quad\quad
\qed
$$
\end{dfn}
%%(End of Definition 7.5)%%%%%%
%%%
%%%
%%%
%%%
%%%%%%%%%%%%%%%%%%%%%
Now we are ready to define the scanning map.
From now on, we identify $\C=\R^2$ in a usual way.
%%
%%%
%%%
%%%
%%%
%%(Definition 7.6)%%
\begin{dfn}
%%%%%%%%%%%%%%%
%(i)
{\rm
For a rectangle $X$ in $\C =\R^2$, let $\sigma X$ denote
the union of the sides of $X$ which are parallel to the $y$-axis, and
for a subspace $Z\subset \C=\R^2$, let $\overline{Z}$ be the closure of $Z$.
From now on, let $I$ denote the interval
$I=[-1,1]$
and
let $0<\epsilon <\frac{1}{1000}$ be any fixed real number.
\par
For each $x\in\R$, let $V(x)$ be the set defined by
%%(7.12)%%
\begin{eqnarray}
%%%%
V(x)
&=&
\{w\in\C: \vert \mbox{Re}(w)-x\vert  <\epsilon , \vert\mbox{Im}(w)\vert<1 \}
\\
\nonumber
&=&
(x-\epsilon ,x+\epsilon )\times (-1,1),
\end{eqnarray}
%%%%%%%%
%%
and let's identify $I\times I=I^2$ with the closed unit rectangle
$\{t+s\sqrt{-1}\in \C: -1\leq t,s\leq 1\}$ 
in $\C$.
%%%
\par
For each $D=(d_1,\cdots ,d_r)\in \N^r$,
define {\it the horizontal scanning map}
%%%(7.13)%%
\begin{equation}
sc^{}_D:\ES_{D}(\C )\to \Omega \ES (I^2,\partial I\times I)
=\Omega \ES (I^2,\sigma I^2)
\end{equation}
%%%
as follows.
%%
%%%%%%%%%%%%%%%%%%%%
For each $r$-tuple 
$\alpha =(\xi_1,\cdots ,\xi_r)\in \ES_D(\C)$
of configurations,
let
$
sc^{}_D(\alpha):\R \to \ES (I^2,\partial I\times I)
=\ES (I^2,\sigma I^2)
$
denote the map given by
\begin{align*}
%%%%%%
\R\ni x
&\mapsto
(\xi_1\cap\overline{V}(x),\cdots ,\xi_r\cap\overline{V}(x))
\in
\ES (\overline{V}(x),\sigma \overline{V}(x))
\cong \ES_{}(I^2,\sigma I^2),
\quad
%%%%%
\end{align*}
%for $x\in \R$,
where 
we use the canonical identification
%\begin{equation}
$(\overline{V}(x),\sigma \overline{V}(x))
\cong (I^2,\sigma I^2).$
%\end{equation}
%%%
\par\vspace{2mm}\par
%%%%%
Since $\dis\lim_{x\to\pm \infty}sc^{}_D(\alpha)(x)
=(\emptyset ,\cdots ,\emptyset)$,  
by setting $sc^{}_D(\alpha)(\infty)=(\emptyset ,\cdots ,\emptyset)$
we obtain a based map
$sc^{}_D(\alpha)\in \Omega \ES(I^2,\sigma I^2),$
where we identify $S^1=\R \cup \infty$ and
we choose the empty configuration
$(\emptyset ,\cdots ,\emptyset)$ as the base-point of
$\ES (I^2,\sigma I^2)$.
%%%
%%
One can show that the following  diagram is homotopy commutative:
%(7.14)%%%%%
\begin{equation}\label{CD: scanning}
\begin{CD}
\ES_{D+k\e} (\C)@>sc_{D+k\e}>> \Omega \ES (I^2,\sigma I^2)
\\
@V{s_{D+k\e}^{\Sigma}}VV \Vert @.
\\
\ES_{D+(k+1)\e}(\C) @>sc_{D+(k+1)\e}>> \Omega \ES (I^2,\sigma I^2) 
\end{CD}
\end{equation}
%%%%%%%%%%
Thus by using the above diagram
and by identifying
$\Pol^*_{D+k\e}(S^1,\XS)$ with 
$\ES_{D+k\e}(\C)$,
%%%%
we finally obtain {\it the stable horizontal scanning map}
%%(7.15)%%%%%
\begin{equation}\label{eq: stable horizontal}
%%%%%%%%%%%%
S
=\lim_{k\to \infty}sc_{D+k\e}
:\Pol_{D+\infty}^{\Sigma}
%=\lim_{k\to\infty}\Pol_{D+k\e}^*(S^1,\XS)
\to
\Omega
\ES (I^2,\sigma I^2),
%%%%
\end{equation}
%%%%
where
$\Pol_{D+\infty}^{\Sigma}$ is defined in
(\ref{eq: stabilized space}).
%%%
\par
(ii)
If we identify
$(I,\partial I)=(I\times \{0\},\partial I\times \{0\})$, then
we notice that
$(\xi_1\cap \R,\cdots ,\xi_r\cap \R)\in \ES (I,\partial I)$
for $(\xi_1,\cdots ,\xi_r)\in \ES (I^2,\sigma I^2).$ 
Thus one can define the map
$R_{\Sigma}:\ES (I^2,\sigma I^2) \to \ES (I,\partial I)$ by
%%%
%%%(7.16)%%
\begin{equation}
R_{\Sigma}(\xi_1,\cdots ,\xi_r)=(\xi_1\cap \R,\cdots ,\xi_r\cap \R)
\ \ 
\mbox{for $(\xi_1,\cdots ,\xi_r)\in \ES (I^2,\sigma I^2)$.}
\end{equation}
%%%
%for $(\xi_1,\cdots ,\xi_r)\in \ES (I^2,\sigma I^2)$.
%%
\par
(iii)
Let $S^H:\Pol^{\Sigma}_{D+\infty}\to \Omega \ES (I,\partial I)$
denote {\it the horizontal scanning map} defined by
the composite of maps
%%(7.17)%%
\begin{equation}
S^H=(\Omega R_{\Sigma})\circ S:
\Pol^{\Sigma}_{D+\infty}
\stackrel{S}{\longrightarrow}
\Omega \ES (I^2,\sigma I^2)
\stackrel{\Omega R_{\Sigma}}{\longrightarrow}
\Omega \ES (I,\partial I),
\end{equation}
%%%
where
$S$ denotes the stable horizontal scanning map given by 
(\ref{eq: stable horizontal}).
\qed
}
%%%%%
\end{dfn}
%%(End of Definition 7.6)%%%
%%%
%%%
%%%
%%%
%%(Lemma 7.7)%%
\begin{lmm}\label{lmm: ES}
%%%
The map
$R_{\Sigma}:\ES (I^2,\sigma I^2) 
\stackrel{\simeq}{\longrightarrow} \ES (I,\partial I)=E^{\Sigma}(I,\partial I)$
is a deformation retraction.
%%%%%
\end{lmm}
%%%
\begin{proof}
%%%%%(Proof of Lemma 7.7)%%%%
We identify
$I^2=\{a+b\sqrt{-1}\in \C: -1\leq a,b\leq 1\}\subset \C$ as before.
Let $\Pi\subset I^2$ denote the subspace defined by
$\Pi =\{a+b\sqrt{-1}\in I^2:b\in \{0,\pm \frac{1}{2}\}\}.$
For $b\in \R$, let $\epsilon (b)=\frac{b}{|b|}$ if $b\not=0$ and $\epsilon (0)=0$.
Now consider the homotopy $\varphi :I^2\times [0,1]\to I^2$ given by
$\varphi (\alpha,t)=a+\{(1-t)b+\frac{\epsilon (b)t}{2}\}\sqrt{-1}$
for
$\alpha =a+b\sqrt{-1}\in I^2$
$(a,b\in\R)$.
By means of this homotopy, one can define a deformation retraction
$
R:\ES (I^2,\sigma I^2)
\stackrel{\simeq}{\longrightarrow}
\ES (\Pi ,\partial I\times \{0,\pm \frac{1}{2}\}).
$
\par
Next, by using the homotopy given by
$f_t(a+b\sqrt{-1})=ta+(1-t)+b\sqrt{-1}$ if $b=\pm \frac{1}{2}$
and $f_t(a+b\sqrt{-1})=a$ if $b=0$,
one can also define a deformation retraction
$\varphi :
\ES (\Pi, \partial I\times \{0,\pm \frac{1}{2}\})
\stackrel{\simeq}{\longrightarrow}
\ES (I,\partial I).$
Since $R_{\Sigma}=\varphi\circ R$, the map $R_{\Sigma}$ is 
a deformation retraction.
%%%%%%%%%%%%%%%%%%%%%%%%%%%%%%%%%%%%
\end{proof}
%%%%%%%%%(End of proof of Lemma 7.7)%%%%

%%%
%%%
%%%
%%%(Scanning map Theorem)%%%%%
%%(Theorem 7.8)%%%
\begin{thm}[\cite{Gu1}, \cite{Se}]\label{thm: scanning map}
%%%%%%%%%%%%%%%%%%%
The  horizontal scanning map
$$
S^H
:\Pol_{D+\infty}^{\Sigma}
\stackrel{\simeq}{\longrightarrow}
\Omega
E^{\Sigma} (I,\partial I)
$$
is a homotopy equivalence.
\end{thm}
%%%%%%%%%%%%
\begin{proof}
%%%(Proof of Theorem 7.8)%%%
Since $R_{\Sigma}$ is a homotopy equivalence and $S^H=(\Omega R_{\Sigma})\circ S$,
it suffices to show that $S$ is a homotopy equivalence.
\par
The proof is  analogous to the one given in
\cite[Prop. 3.2, Lemma 3.4]{Se} and \cite[Prop. 2]{Gu1}.
We identify $\C=\R^2$ by means of the identification
$x+\sqrt{-1}y\mapsto (x,y)$ in the usual way.
Let $B$ and $B^*$ denote the rectangles in $\R^2=\C$ given by
$
B^*= [-1,2]\times [-1,1]$ and
$B = (0,1)\times (-1,1)$. 
Let $\{V_t:0<t<1\}$ be the family of open rectangles in $B$ given by
$V_t=(t-\epsilon (t),t+\epsilon (t))\times (-1,1)$, where
$\epsilon (t)$ denotes the continuous function defined on the open interval $(0,1)$
such that
$0<\epsilon (t)<\min \{t,1-t\}$
for any $t\in (0,1)$ with 
$\dis \lim_{t\to + 0}\epsilon (t)=\lim_{t\to 1-0}\epsilon (t)=0.$
%%%%%%%%%%%%%%%
\par
%%%%
Let
$
\widetilde{sc}^{H}_D:
\ES_D(B)\times [0,1]\to \ES(\overline{B},\sigma \overline{B})
$
denote the map given by
$$
\widetilde{sc}^{H}_D((\xi_1,\cdots ,\xi_r),t)
=
(\xi_1\cap V_t,\cdots ,\xi_r\cap V_t)\in 
\ES (\overline{V_t},\sigma\overline{V_t})
\cong \ES (\overline{B},\sigma \overline{B}),
$$ 
where
we use the canonical identification
$ (\overline{V_t},\sigma\overline{V_t})
\cong  (\overline{B},\sigma \overline{B}).$
\par
Since $\dis \lim_{t\to + 0}\widetilde{sc}^{H}_D(\xi ,t)
=\lim_{t\to 1-0}\widetilde{sc}^{H}_D(\xi ,t)
=(\emptyset ,\cdots ,\emptyset)$ 
for any $\xi \in \ES_D(B)$,
the adjoint of $\widetilde{sc}^{H}_D$ defines the map
$
sc^{H}_D:\ES_D(B)\to \Omega \ES (\overline{B},\sigma \overline{B}).
$
\newline
If $s_{D}^{\Sigma}:\ES_D(B)\to \ES_{D+\e}(B)$
denotes the stabilization map 
defined by adding points from infinity as in (\ref{CD: stabilization}),
 we obtain a homotopy commutative
diagram
$$
\begin{CD}
\Pol_D^*(S^1,\XS) @>s_{D}>> \Pol_{D+\e}^*(S^1,\XS)
\\
@V{\cong}VV @V{\cong}VV
\\
\ES_D(B) @>s_{D}^{\Sigma}>> \ES_{D+\e}(B)
\end{CD}
$$
for each $D\in \N^r$.
%%%%
If $\ES_{D+\infty}(B)$ denotes the colimit
%%%%%%
$\dis \ES_{D+\infty}(B)=
\lim_{k\to \infty}\ES_{D+k\e}(B)$
%%%
%%%
taken over the stabilization maps 
$\{s_{D,+k\e}^{\Sigma}:\ES_{D+k\e}(B)\to \ES_{D+(k+1)\e}(B)\}$, 
there is a homotopy commutative diagram
$$
\begin{CD}
\Pol^*_{D+\infty}(S^1,\XS) @>S>>  \Omega \ES(I^2,\sigma I^2)
\\
@V{\cong}VV @V{\cong}VV
\\
\ES_{D+\infty}(B) @>S^{\p}>> \Omega \ES (\overline{B},\sigma \overline{B})
\end{CD}
$$
where we set $\dis S^{\p}=\lim_{k\to\infty}sc^{H}_{D+k\e}$.
%%
%\par
It suffices to prove that $S^{\p}$ is a homotopy equivalence.
%%%%%% 
%\par
%%%%%%%%%%%%
Let $\ES_D$
denote the subspace of $\ES (B^*,\sigma B^*)$ consisting of all
$r$-tuples
$(\xi_1,\cdots ,\xi_r)\in \ES (B^*,\sigma B^*)$
such that
$(\xi_1\cap B,\cdots ,\xi_r\cap B)\in \ES_D(B).$
\par
Let
$q_{D}:
\ES_D
\to
\ES (B^*,\sigma B^*\cup B)\cong 
\ES (\overline{B},\sigma \overline{B})^2$
denote the restriction of the
the quotient map
$
\ES (B^*,\sigma B^*)\to
\ES (B^*,\sigma B^*\cup B)
\cong
\ES (\overline{B},\sigma \overline{B})^2,
$
where
$\ES (\overline{B},\sigma \overline{B})^2=
\ES (\overline{B},\sigma \overline{B})\times 
\ES (\overline{B},\sigma \overline{B})$.
It is easy to see that the fiber of $q_D$ is homeomorphic to
$\ES_D(B)$.
%%%%%%%
\par
Next, define the scanning map
$\tilde{S}_D:\ES_D
\to
\Map ([0,1],\ES (\overline{B},\sigma \overline{B}))$ 
by %the adjoint of the following map
%%%(Definition of $\tilde{S}$)%%%
$$
\begin{CD}
\ES_D
\times [0,1] 
@>>> \ES (\overline{B},\sigma \overline{B})
\\
((\xi_1,\cdots ,\xi_r),t)
@>>>
(\xi_1\cap B_t,\cdots ,\xi_r\cap B_t)\in
\ES (\overline{B_t},\sigma
\overline{B_t})\cong \ES(\overline{B},\sigma \overline{B})
\end{CD}
$$
%%%%%%%%%
where $\{B_t:0<t<1\}$ denotes the family of the open rectangles in $B^*$
defined by
$B_t=(2t-1,2t)\times (-1,1)$
and
we use the canonical identification
$
 (\overline{B_t},\sigma\overline{B_t})
\cong  (\overline{B},\sigma \overline{B}).$
%%
%\par
We obtain a 
homotopy commutative diagram
%%%%%
%%%%%
%%%
%%%
%%%
%%%
%%(7.18)%%
\begin{equation}\label{CD: fibration}
%%%%%%%%
\begin{CD}
\ES_D
@>q_D>> 
\ES (\overline{B},\sigma \overline{B})^2
\\
@V{\tilde{S}_D}VV \Vert @.
\\
\Map ([0,1],\ES (\overline{B},\sigma \overline{B}))
@>res>>
\ES (\overline{B},\sigma \overline{B})^2
\end{CD}
\end{equation}
%%%
where
$$res:\Map ([0,1],\ES (\overline{B},\sigma \overline{B}))
\to
\Map (\{0,1\},\ES (\overline{B},\sigma \overline{B}))\cong
\ES (\overline{B},\sigma \overline{B})^2
$$
denotes the  restriction map.
%$res(f)=f\vert \{0,1\}$.
%%%%%%%
\par
%%%%%
Note that $res$ is a fibration with fibre
$\Omega \ES (\overline{B},\sigma \overline{B})$, but
the map
$q_D$ is not a fibration although every fiber of $q_D$ is homeomorphic to 
$\ES (B)$.
%%%%%%
However, once the map $q_D$ is stabilized, it becomes a
quasifibration. 
%homology fibration. %(for any local coefficient system).
%%
\par
%%%%%
To see this, let $J=(0,1)$ and
for each $D\in \N^r$
let us choose the $r$-tuple 
$\textit{\textbf{c}}_D=
(c_{D,1},\cdots ,c_{D,r})
\in (J\times J)^r$ 
%of the elements in $J^2$
such that
$c_{D,j}\not= c_{D^{\prime},k}$ if $(D,j)\not= (D^{\prime},k)$ with
$\lim_{D\to\infty}c_{D,k}=(\frac{1}{2},0)$
for each $1\leq k\leq r$.%%%%
\par
Let $\widehat{\ES} (B^*,\sigma B^*)$
denote the space of
$r$-tuples $(\xi_1,\cdots ,\xi_r)$
of
formal infinite divisors in $(B^*,\sigma B^*)$
%of the forms 
%in $\SP_{\K}(B^*,\sigma B^*)$
%such that
satisfying the following two conditions:
%%%%%%%%%%%%%%%%%%%
\begin{enumerate}
%%%%%%%%%%%%%%%%%%%
\item[$(\dagger)_1$]
%%%%%%%%%%%%%%%%%%
$(\bigcap_{k=1}^r\xi_k)\cap \R \cap (B^*\setminus \sigma B^*)=\emptyset .$
%%%%
\item[$(\dagger)_2$]
%%%%%%%%%%%%%%%%%%
Each divisor $\xi_k$ is represented as the formal infinite sum of the form
$\xi_k=\sum_{D}\xi_{k,D}$ 
such that $\xi_{k,D}\in \SP^1 (B^*,\sigma B^*)$ 
for each $D\in \N^r$, and it
almost coincides (except finite sums) with
$\xi_{k}^*=\sum_{D}c_{D,k}$
for each $1\leq k\leq r$.
%%%%%%%%%%%%%%%%%%%
\end{enumerate}
%%%%%
Note that $c_{D,k}\in B$ for any $(D,k)\in \N^r\times [r]$.
Thus one can
define the map
$
\hat{q}:\widehat{\ES}(B^*,\sigma B^*)
\to
\ES (\overline{B},\sigma \overline{B})^2
$
to be the natural quotient map
$$
\widehat{\ES} (B^*,\sigma B^*)
\to
\ES (B^*,\sigma B^*
\cup B)\cong 
\ES (\overline{B},\sigma \overline{B})^2.
$$
By using the Dold-Thom argument exactly as in \cite[Lemma 3.4]{Se}
together with the fact that $\ES_D(\C)$ is simply connected
(by Proposition \ref{prp: simply connected} and (\ref{eq: Pol=E})), 
one can show that $\hat{q}$ is a quasifibration. 
%(we make use of the fact that the fundamental group %of the spaces 
%$\pi_1(Q^{d_1,\cdots ,d_m;m}_{n,\K}(B))$
%is abelian for any $D=(d_1,\cdots ,d_m)\in (\Z_{\geq 1})^m$ when $d(\K)mn\geq 4.$
%of $m$-tuples of divisors of degrees precisely $d_1,\dots, d_m$ satisfying the condition $(*)_n$ are abelian, when $d(\K)mn\geq 4$. 
%This can be seen by using the string representation of elements of these groups as in  
%\cite[\S 5 Appendix]{GKY1}). 

\par
%For each integer $k\geq 1$ and $D=(d_1,\cdots ,d_r)\in \N^r$,
%let us write $D(k)=(d_1+2k,\cdots ,d_r+2k)$.
%%%%%
Now define
 stabilization maps
$f_D:\ES_D\to \ES_{D+\e}$ by
$f_D(\xi_1,\cdots ,\xi_r)=
(\xi_1+c_{D,1},
\cdots
,\xi_r+c_{D,r})$
for $(\xi_1,\cdots ,\xi_r)\in \ES_D.$
%%%
\par
Let $\ES$ denote the colimit $\dis \ES=\lim_{k\to\infty}\ES_{D+k\e}$ from the stabilization maps 
$\{f_{D+k\e}\}_{k\geq 0}.$
%%%
Since
%$Q^m_{n,\K}\subset \widehat{Q}^m_{n,\K}(B^*,\sigma B^*)$ and 
$\widehat{\ES}(B^*,\sigma B^*)$ is $\Z^r\times \ES$,
%(up to homotopy equivalence),
the restriction $q_{\infty}=\hat{q}\vert \ES:\ES\to
\ES(\overline{B},\sigma \overline{B})^2$
is also a quasifibration.
%%%
Since the fiber of $q_D$ 
is $\ES(B)$ and
$q_{\infty}\vert \ES=q_D$,
we see that
the fiber of  $q_{\infty}$ is
$\ES(B)$ and we may regard the map $q_{\infty}$ as the stabilized map
of $\{q_{D+k\e}\}_{k\geq 0}$.
\par
%Since $c_{D,j}\notin \R$ for any $(D,j)$, 
%we see that
%$\tilde{S}^{d+2}\circ f^d=\tilde{S}^d$ and
We also obtain the stabilized scanning map
$$
\tilde{S}=\lim_{k\to\infty}\tilde{S}_{D+k\e}:
\ES =\lim_{k\to\infty}\ES_{D+k\e}
\to \Map ([0,1],\ES(\overline{B},\sigma \overline{B})).
$$
Since 
%$c_{D,j}\notin \R$ for any $(D,j)\in \N^r\times [r]$
 $[0,1]$ is contractible
 and there is a homotopy equivalence
 $\ES =\lim_{k\to\infty}\ES_{D+k\e}\simeq 
 \ES(\overline{B},\sigma \overline{B})$,
%with
%$\lim_{d\to\infty}c_{d,j}=(\frac{1}{2},0)$,
we see that $\tilde{S}$ is a homotopy equivalence.
Since $B\subset B^*\setminus\sigma B^*$,
we may regard $\ES (B)$
as a subspace of $\ES (B^*,\sigma B^*)$
and we see that
$\tilde{S}_D\vert \ES (B)=sc^{H}_D$.
%(up to homotopy equivalence) 
%for each $D$.
%%
Thus, we can identify $\tilde{S}\vert \ES (B)=S^{\p}$
%(up to homotopy equivalence).
and by  the diagram (\ref{CD: fibration})
 we have the  homotopy
commutative diagram
%%%%%%%%%%%%
%%%()%%%
\begin{equation*}\label{eq: CD: fibrations}
%%%%%%%
\begin{CD}
\ES (B) 
@>>> 
\ES 
@>{q_{\infty}}>> 
\ES (\overline{B},\sigma \overline{B})^2
%%%%%
\\
@V{S^{\p}}VV @V{\tilde{S}}V{\simeq}V \Vert @.
%%%%
\\
\ES (\overline{B},\sigma \overline{B})
@>>>
\Map ([0,1],\ES (\overline{B},\sigma \overline{B}))
@>{res}>>
\ES (\overline{B},\sigma \overline{B})^2
\end{CD}
%%%
\end{equation*}
%%%%%%
Since
the upper horizontal sequence 
is a quasifibration sequence and
the lower horizontal one is a fibration sequence,
$S^{\p}$ is a homotopy equivalence.
%homology equivalence (for any local coefficient system). 
%%%
\end{proof}
%%(End of proof of Theorem 7.7)%%%%%
%%%
%%%
%%%
%%%
%%(End of SECTION 7)%%%%%%
%%
%%
%%
%%
%%
%%
%%
%%
%%
%%%%%%%%%%%%%%%%%%%
%%%%(SECTION 8)%%%%
\section{The stable result}\label{section: stability}
%%%%%%%%%%%%%%%%%%%
In this section we prove the stability result
(Theorem \ref{thm: V}).

%%%%%%%%%%%%%%%%%
%%(Definition 8.1)%%%
\begin{dfn}
%%%%
{\rm
Let $D=(d_1,\cdots ,d_r)\in \N^r$ be an $r$-tuple of positive integers.
Then it is easy to see that the following diagram is homotopy commutative:
$$
\begin{CD}
\Pol_D^*(S^1,\XS) @>j_D>> \Omega U(\KS)
%@>\Omega r_{\C}>\simeq> \Omega \mathcal{Z}_{\KS}  
\\
@V{s_D}VV   \Vert @. 
\\
\Pol_{D+\textit{\textbf{e}}}^*(S^1,\XS) 
@>j_{D+\textit{\textbf{e}}}>>
 \Omega U(\KS)
% @>\Omega r_{\C}>\simeq> \Omega \mathcal{Z}_{\KS} 
\end{CD}
$$
%%%
Hence we can stabilize the  map
%%
%%(8.1)%%%%%%%%%
\begin{equation}\label{eq: inclusion stab}
%%%%%%%%%%%%%%%%
j_{D+\infty}=\lim_{t\to\infty}j_{D+t\textit{\textbf{e}}}:
\Pol_{D+\infty}^{\Sigma}=\lim_{t\to\infty}
\Pol^*_{D+t\e}(S^1,\XS)
\to
\Omega U(\KS).
\end{equation}
%%%
}
\end{dfn}
%%(End of definition 8.1)%%%

The main purpose of this section is to prove the following result.
%%%
%%
%%
%%
%%
%%
%%
%%%(Stable Theorem)%%%%
%%%(Theorem 8.2)%%
\begin{thm}\label{thm: V}
%%%
The map
$j_{D+\infty}:
\Pol^{\Sigma}_{D+\infty}
\stackrel{\simeq}{\longrightarrow}
\Omega U(\KS)
$
is a homotopy equivalence.
\end{thm}
%%%%%%(End of Theorem 8.2)%%
%%
%%
%%
%%%%%%%%%%%%%%%%%%%
Before proving Theorem \ref{thm: V} we need the following definition and lemma.
%%
%%%(Definition 8.3)%%%
\begin{dfn}
%%%
{\rm
Now we identify $\C=\R^2$ in a usual way and let us write
$U=\{w\in \C:\vert \mbox{Re}(w) \vert <1,\vert \mbox{Im}(w)\vert <1\}
=(-1,1)\times (-1,1)$ and $I=[-1,1]$.
\par\vspace{1mm}\par
%%%%%%
(i)
For an open set $X\subset \C$, let $F(X)$ 
denote the space
of $r$-tuples $(f_1(z),\cdots ,f_r(z))\in \C [z]^r$
of (not necessarily monic) polynomials satisfying the following condition
$(*)$:
\begin{enumerate}
%%(i-1)%%
\item[($*$)]
%%%%%%
 For any $\sigma =\{i_1,\cdots ,i_s\}\in I(\KS)$,
the polynomials
$f_{i_1}(z),\cdots ,f_{i_s}(z)$ have no common 
{\it real} roots in $X$ (i.e. no common roots in $X\cap \R$).
 \end{enumerate}
%%%%
Similarly, let $F^{\C}(X)\subset F(X)$ denote the subspace
of all $(f_1(z),\cdots ,f_r(z)) \in F(X)$
satisfying the following condition $(*)_{\C}$:
\begin{enumerate}
%%(i-1)%%
\item[$(*)_{\C}$]
%%%%%%
 For any $\sigma =\{i_1,\cdots ,i_s\}\in I(\KS)$,
the polynomials
$f_{i_1}(z),\cdots ,f_{i_s}(z)$ have no common 
roots in $X$.
 \end{enumerate}
%%
%An element 
%$(f_1(z),\cdots ,f_r(z))\in F(X)$
%defines a map $X\cap \R\to U(\mathcal{K}_{\Sigma})$.
%%
\par
%%%%%%%%%%
\par
(ii) Let $ev_0:F(U)\to U(\mathcal{K}_{\Sigma})$ 
%(resp. $ev^{\C}:F^{\C}(U)\to U(\KS )) $ 
denote the map given by evaluation at $0$, i.e.
$ev_0(f_1(z),\cdots ,f_r(z))=(f_1(0),\cdots ,f_r(0))$
for $(f_1(z),\cdots,f_r(z))\in F(U)$.
%(resp. $ev_0^{\C}(f_1,\cdots ,f_r)=(f_1(0),\cdots ,f_r(0))$).
\par\vspace{1mm}\par
(iii)
Let $\tilde{F}(U)\subset F(U)$ 
(resp.  $\tilde{F}^{\C}(U)\subset F^{\C}(U)$)
denote the subspace
%consisting 
of all
$(f_1(z),\cdots ,f_r(z))\in F(U)$ 
(resp. $(f_1(z),\cdots ,f_r(z))\in F^{\C}(U)$)
such that no $f_i(z)$ is identically
zero.
Let 
$j_F:\tilde{F}^{\C}(U)\stackrel{\subset}{\longrightarrow}
\tilde{F}(U)$
denote the inclusion map.
%%
%%%
%Clearly, $\tilde{F}^{\C}(U)$ is a subspace of $\tilde{F}(U)$ 
%and there is an inclusion map
%%%(8.2)%%
%\begin{equation}
%%%%%%%%
%j_F:\tilde{F}^{\C}(U)\stackrel{\subset}{\longrightarrow} \tilde{F}(U).
%\end{equation}
\par
(iv)
Let
$ev_{\R}:\tilde{F}(U)\to U(\mathcal{K}_{\Sigma})$
%(resp. $ev^{\C}:\tilde{F}^{\C}(U)\to U(\KS))$
denote
the map given by the restriction
$ev_{\R}=ev_0\vert \tilde{F}(U).$
%(resp. $ev^{\C}=ev_0^{\C}\vert \tilde{F}(U)$).
%%
\par
(v)
Note that the group $\T^r_{\R}=(\R^*)^r$ 
(resp. $\T^r_{\C}=(\C^*)^r$)
acts freely on the space
$\tilde{F}(U)$ 
(resp. $\tilde{F}^{\C}(U)$)
in a natural way,
%, and
%by coordinate multiplication 
%(for $X=U$ or $\C$), 
and
let
%%(8.2)%%
\begin{equation}
p:\tilde{F}(U)\to \tilde{F}(U)/\T^r_{\R}
\qquad\quad
(q:\tilde{F}^{\C}(U)\to \tilde{F}^{\C}(U)/\T^r_{\C})
\end{equation}
denote the natural projection, where
 $\tilde{F}(U)/\T^r_{\R}$
 (resp. $\tilde{F}^{\C}(U)/\T^r_{\C}$)
%
%%by coordinate multiplication 
%%(for $X=U$ or $\C$), 
%
%let $\tilde{F}(X)/\T^r_{\R}$ %and $\tilde{F}^{\C}(X)/\GS$
denotes 
the  orbit space.
Note that $\overline{U}\cap \R=I\times \{0\}$, and
let
%we define the map
%%(8.3)%%
\begin{equation}
%%%%%
v:\tilde{F}(U)/\T^r_{\R} \to 
E^{\Sigma} (I, \partial I)
\end{equation}
%%%%%
denote the natural map which
assigns to an $r$-tuple  $(f_1(z),\cdots ,f_r(z))\in \tilde{F}(U)$ 
%of polynomials
the $r$-tuple of their configurations represented by
their real roots %which lie 
in 
$I=[-1,1]$.
%%%
Similarly, note that $\overline{U}=D^2$ and let
%define the map
%%(8.4)%%
\begin{equation}
u:\tilde{F}^{\C}(U)/\T^r_{\C}\to E^{\Sigma} (D^2, S^1)
\end{equation}
%%%
denote the natural map which
assigns to an $r$-tuple  $(f_1(z),\cdots ,f_r(z))\in \tilde{F}^{\C}(U)$ 
%of polynomials
the $r$-tuple of their configurations represented by
their  roots %which lie 
in $D^2$.
%%%
\par
(vi)
By using this identification
$(I,\partial I)=(I\times \{0\},\partial I\times \{0\})\subset
(D^2,S^1)$,
we obtain the inclusion map
%%(8.5)%%
\begin{equation}\label{eq: map iSigma}
%%%%%%%
%%%%
i_{\Sigma}:
\ E^{\Sigma}(I,\partial I)
\stackrel{\subset}{\longrightarrow} 
E^{\Sigma}(D^2,S^1).
%%%
\end{equation}
%%%
%%%%%%%%
}
\end{dfn}
%%
%%(End of Definition 8.3)%%%
%%(Lemma 8.4)%%
\begin{lmm}\label{lmm: ev}
%%%%%%%
%%(i)%%
$\I$
The space $\Pol_{D+\infty}^{\Sigma}$ is simply connected for any $D\in \N^r$.
\par
%%(ii)%%
$\II$
The map
$ev_{\R}:\tilde{F}(U)\stackrel{\simeq}{\longrightarrow}
U(\mathcal{K}_{\Sigma})$
is a homotopy equivalence.
\par
$\III$
The inclusion map $j_F:\tilde{F}^{\C}(U)
\stackrel{\simeq}{\longrightarrow}
 \tilde{F}(U)$ 
is a homotopy equivalence.
 \par
 $\IV$
The map
$u:\tilde{F}^{\C}(U)/\T^r_{\C}
\stackrel{\simeq}{\longrightarrow}
E^{\Sigma}(D^2,S^1)$ is a homotopy equivalence.
\par
$\V$ 
The induced homomorphisms
$$
\begin{cases}
(i_{\Sigma})_*:\pi_k(E^{\Sigma}(I,\partial I))
\stackrel{\cong}{\longrightarrow}
\pi_k(E^{\Sigma}(D^2,S^1))
\\
v_*:\pi_k(\tilde{F}(U)/\T^r_{\R})
\stackrel{\cong}{\longrightarrow}
\pi_k(E^{\Sigma}(I,\partial I))
\end{cases}
$$
are isomorphisms for any $k\geq 3$.
%\end{equation}
%%%%%
\end{lmm}
%%%
\begin{proof}
%%%%%(Proof of Lemma 8.4)%%%%
(i)
Since $\Pol_D^*(S^1,\XS)$ is simply connected for any $D\in \N^r$ 
by Proposition \ref{prp: simply connected}, 
the space $\Pol_{D+\infty}^{\Sigma}$ is also simply connected.
\par
(ii)
Let $i_0:U(\mathcal{K}_{\Sigma})\to F(U)$ be the inclusion map given by viewing constants as polynomials. 
Clearly $ev_0\circ i_0=\mbox{id}$.
Let $f: F(U)\times [0,1]\to
F(U)$ be
the homotopy given by
$f((f_1,\cdots ,f_t),t)=(f_{1,t}(z),\cdots ,f_{r,t}(z))$,
where
$f_{i,t}(z)=f_i(tz)$.
This gives a homotopy between
$i_0\circ ev_0$ and the identity map, and this proves that $ev_0$ is a deformation retraction.
Since $F(U)$ is an infinite dimensional manifold and
$\tilde{F}(U)$ is a closed submanifold of $F(U)$ of infinite codimension,
it follows from \cite[Theorem 2]{EK} that
the inclusion
$\tilde{F}(U)\to F(U)$ is a homotopy equivalence.
Hence the restriction 
$ev_{\R}$ is also  
a homotopy equivalence.
\par
%%(iii)%%
(iii)
It follows from \cite[Lemma 5.4]{KY9} that
the evaluation map 
$ev=ev_{\R}\circ j_F:\tilde{F}^{\C}(U)
\stackrel{\simeq}{\longrightarrow} U(\KS)$ is a
homotopy equivalence. 
Since $ev_{\R}$ is a homotopy equivalence by  the assertion (ii), we see that the map 
$j_F:\tilde{F}^{\C}(U)\stackrel{\simeq}{\longrightarrow}
\tilde{F}(U)$ is also a homotopy equivalence.

\par
(iv)
We know that the map $u$ is a homotopy equivalence
by \cite[page 133]{Gu2}.
However,  for the sake of completeness of this paper we will give the another proof.
%%%%%
Since all the maps and homotopies which
appear in the proof of \cite[Lemma 5.4]{KY9}
are $\T^r_{\C}$-equivariant maps, we see that the map
 $ev$ is
  indeed  a $\T^r_{\C}$-equivariant homotopy equivalence.
 Thus the map $ev$ induces a homotopy equivalence
 $\widetilde{ev}:\tilde{F}^{\C}(U)/\T^r_{\C}
 \stackrel{\simeq}{\longrightarrow}EG\times_GU(\KS)$,
 where
 $G=\T^r_{\C}$ and $EG\times_GU(\KS)$ denotes the Borel construction.
 \par
 Recall that there is a natural deformation retraction
 $rt:EG\times_GU(\KS)\stackrel{\simeq}{\longrightarrow} DJ(\KS)$
 by \cite[Theorem 6.29]{BP}.
If $r_{\Sigma}:E^{\Sigma}(D^2,S^1)
\stackrel{\simeq}{\longrightarrow}
DJ(\KS)$ denotes the deformation retraction given in \cite[Lemma 4.3]{KY9}, then one can check that the following diagram is commutative (up to homotopy equivalence):
%%(8.6)%%%
\begin{equation}
\begin{CD}
\tilde{F}^{\C}(U)/\T^r_{\C} @>\widetilde{ev}>\simeq>EG\times_GU(\KS)
\\
@V{u}VV @V{rt}V{\simeq}V
\\
E^{\Sigma}(D^2,S^1) @>r_{\Sigma}>\simeq> DJ(\KS)
\end{CD}
\end{equation}
Thus, we see that $u$ is a homotopy equivalence.
\par
(v)
Let $p^{\p}:\tilde{F}(U)/\T^r_{\R}\to \tilde{F}(U)/\T^r_{\C}$
be the natural projection.
Since $p^{\p}$ is a bundle projection with fiber $\T^r_{\C}/\T^r_{\R}\cong (S^1)^r$,
we see that
the map $p^{\p}$ induces an isomorphism on homotopy groups 
$\pi_k(\ )$ for any $k\geq 3$.
Moreover, since the inclusion $j_F$ is a $\T^r_{\C}$-equivariant map and
 the group $\T^r_{\C}$ acts freely on both spaces spaces $\tilde{F}^{\C}(U)$ and 
$\tilde{F}(U)$, there is a map
$\tilde{j}_F:
\tilde{F}^{\C}(U)/\T^r_{\C}
\to
\tilde{F}(U)/\T^r_{\C}$
such that the following diagram is commutative:
%%(8.7)%%
\begin{equation}\label{CD: fibration-diagram}
\begin{CD}
%%%%%%%%
\T^r_{\C} @>>> \tilde{F}^{\C}(U) @>q>> \tilde{F}^{\C}(U)/\T^r_{\C}
\\
\Vert @. @V{j_F}V{\simeq}V @V{\tilde{j}_F}VV
\\
\T^r_{\C} @>>> \tilde{F}(U) @>>> \tilde{F}(U)/\T^r_{\C}
%%%%%%%%
\end{CD}
\end{equation}
%%%%%%
Since two horizontal sequences of (\ref{CD: fibration-diagram}) are
fibration sequences and the map $j_F$ is a homotopy equivalence,
the map $\tilde{j}_F$ is a homotopy equivalence.
%%%
Then if we
consider the commutative diagram
%%%%%%%%
%%%(8.8)%%
\begin{equation}\label{CD: map v}
\begin{CD}
\tilde{F}(U)/\T^r_{\R} @>v>> E^{\Sigma}(I,\partial I)@>i_{\Sigma}>>
 E^{\Sigma}(D^2,S^1)
 \\
 @V{p^{\p}}VV @. \Vert  @.
 \\
 \tilde{F}(U)/\T^r_{\C}
 @<\tilde{j}_F<\simeq< \tilde{F}^{\C}(U)/\T^r_{\C} @>u>\simeq> E^{\Sigma}(D^2,S^1)
\end{CD}
\end{equation}
we see that the map $i_{\Sigma}$ induces an epimorphism on homotopy groups
$\pi_k(\ )$ for any $k\geq 3$.
Let $R:E^{\Sigma}(D^2,S^1)\to E^{\Sigma}(I,\partial I)$
denote the restriction map given by
$R(\xi_1,\cdots ,\xi_r)=(\xi_1\cap I,\cdots ,\xi_r\cap I)$.
Since $R\circ i_{\Sigma}=\mbox{id}$, the map $i_{\Sigma}$ induces a monomorphism on homotopy groups $\pi_k(\ )$ for any $k\geq 1$.
Thus the map $i_{\Sigma}$ induces an isomorphism on homotopy groups
$\pi_k(\ )$ for any $k\geq 3$.
By the diagram (\ref{CD: map v}), we can easily see that the map
$v$ also induces an isomorphism on homotopy groups
$\pi_k(\ )$ for any $k\geq 3$.
%%%%%%%%%%%%%%%%
\end{proof}
%%%%(End of proof of Lemma 8.4)%%
%%
%%
%%
%%
%%
%%%%%%%%%%%%%%%%%%%%%%%%%
%%(Proof of the stability Theorem)%%%%%%
%%(Proof of Theorem 8.2)%%%%%%%%%%%
\begin{proof}[Proof of Theorem \ref{thm: V}]
%%%%%%%%%%%%%%%%%%%%%%%%%%
%%%%
%It suffices to prove that the induced homomorphism
%$
%(j_{D+\infty})_*:
%\pi_k(\Pol^{\Sigma}_{D+\infty})
%\stackrel{\cong}{\longrightarrow}
%\pi_k(\Omega U(\KS))
%$
%is an isomorphism for any $k\geq 1$.
Since $\Pol^{\Sigma}_{D+\infty}$ and $\Omega U(\KS)$ 
are simply connected, it suffices to show that
$$
( j_{D+\infty})_*:
\pi_k( \Pol^{\Sigma}_{D+\infty})
\stackrel{\cong}{\longrightarrow}
\pi_k(\Omega U(\KS))
\leqno{(\dagger)_k}
$$
is an isomorphism for any $k\geq 2$.
Let us identify $\C =\R^2$ and
let $U=(-1,1)\times (-1,1)$ as before.
Let  $scan: \tilde{F}(\C)\to \Map (\R,\tilde{F}(U))$ denote the map
given by 
$
scan (f_1(z),\cdots ,f_r(z))(w)=(f_1(z+w),\cdots ,f_r(z+w))
$
for $w\in\R$, and 
consider the diagram
%%%
$$
\begin{CD}
\tilde{F}(U) @>ev_{\R}>\simeq>U(\mathcal{K}_{\Sigma})
\\
@V{p}VV @. 
\\
\tilde{F}(U)/\T^r_{\R} 
@>v>> 
E^{\Sigma} (I,\partial I)
%%%%%
\end{CD}
$$
%where
%$p:\tilde{F}(U)\to \tilde{F}(U)/\GS$
%denotes the natural projection map.
%%%%%%
This induces the commutative diagram below
%%%%%
$$
\begin{CD}
%%%%%
\tilde{F}(\C) @>scan>> \Map (\R, \tilde{F}(U)) 
@>(ev_{\R})_{\#}>\simeq> 
\Map (\R,U(\KS))
\\
@V{p}VV @V{p_{\#}}VV @.
\\
\tilde{F}(\C)/\T^r_{\R} @>scan>>
\Map (\R, \tilde{F}(U)/\T^r_{\R}) 
@>v_{\#}>{}>
\Map (\R, E^{\Sigma} (I,\partial I))
%%%%%
\end{CD}
$$
%%%%%
Observe that $\Map (\R,\cdot )$ can be replaced by $\Map^* (S^1,\cdot)$
by extending  from $\R$ to $S^1=\R\cup \infty$
(as base-point preserving maps).
%%%%%%
Thus 
by setting
$$
\begin{cases}
\widehat{j_D}:\Pol_D^*(S^1,\XS)
 \stackrel{\subset}{\longrightarrow} 
 \tilde{F}(\C) 
\stackrel{scan}{\longrightarrow}
\Map^*(S^1,\tilde{F}(U))
=\Omega \tilde{F}(U)
\\
%%%%%
\widehat{j_D^{\p}}:E^{\Sigma,\R}_D(\C)
\stackrel{\subset}{\longrightarrow}
\tilde{F}(\C) \stackrel{scan}{\longrightarrow}
\Map^*(S^1,\tilde{F}(U)/\T^r_{\R})
=\Omega \tilde{F}(U)/\T^r_{\R}
\end{cases}
$$
we obtain the following commutative diagram
%%%%(8.9: The main diagram)%%
\begin{equation}\label{CD: main1}
\begin{CD}
%%%%%%
\Pol_D^*(S^1,\XS) 
@>\widehat{j_D}>> \Omega \tilde{F}(U) 
@>\Omega ev_{\R}>\simeq>
\Omega U(\KS)
\\
@V{\cong}VV @V{\Omega p}V{\simeq}V @. 
%@.
%%
\\
\ES_D(\C)
@>\widehat{j_D^{\p}}>>
\Omega \tilde{F}(U)/\T^r_{\R}
@>\Omega  v>{}>
\Omega E^{\Sigma} (I,\partial I)
%%%%%%%
\end{CD}
\end{equation}
%%%%%%%%%%%%%%%%%%%%%%%%%%%%
Since $p$ is a covering projection, $\Omega p$ is a homotopy equivalence.
If we identify $\Pol^{\Sigma}_{D+\infty}$ with the colimit
$\dis \lim_{t\to\infty}\ES_{D+t\e}(\C)$, by replacing
 $D$ by $D+t\textit{\textbf{e}}$
$(t\in \N)$
and letting $t\to\infty$, 
we obtain the following homotopy commutative diagram:
%%
%%
%%%(8.10)%%%
\begin{equation}\label{CD: main diagram}
\begin{CD}
%%%%%%%%%%%%
\Pol^{\Sigma}_{D+\infty} 
@>\widehat{j_{D+\infty}}>> 
\Omega\tilde{F}(U) 
@>\Omega ev_{\R}>\simeq>
\Omega U(\KS) 
%%%
\\
%%%
 \Vert @.
@V{\Omega p}V{\simeq}V
@.
%%%%
\\
%%%%
\Pol^{\Sigma}_{D+\infty}
@>\widehat{j_{D+\infty}^{\p}}>{}>
\Omega \tilde{F}(U)/\T^r_{\R}
@>{\Omega v}>> 
\Omega E^{\Sigma} (I,\partial I)
%%%
\end{CD}
\end{equation}
%%%%%%%%%%%%%%
where we set
$\dis \widehat{j_{D+\infty}}=\lim_{t\to\infty}\widehat{j_{D+t\e}}$
and
$\dis \widehat{j_{D+\infty}^{\p}}=
\lim_{t\to\infty}\widehat{j_{D+t\e}^{\p}}.$
\newline
Since
$(\Omega ev_{\R})\circ \widehat{j_{D+t\textit{\textbf{e}}}}
=j_{D+t\textit{\textbf{e}}}$ and
$(\Omega v)\circ \widehat{j_{D+t\textit{\textbf{e}}}^{\p}}
=sc_{D+t\textit{\textbf{e}}}$
(up to homotopy equivalence),
%%%%%%%%%%%%%
we also obtain the following two equalities:
%%%%%(8.11)%%
\begin{equation}\label{eq: j=S}
%%%%%%%%%%%
j_{D+\infty}=(\Omega ev_{\R})\circ \widehat{j_{D+\infty}},
\quad 
S^H=(\Omega v)\circ \widehat{j_{D+\infty}^{\p}}. 
%%%%%%%%%%
\end{equation}
%%%%%%%%%%%%%%%%%%
Since the map $ev_{\R}$ is a homotopy equivalence, 
it suffices to prove that
%%%()%%
\begin{equation*}
(\widehat{j_{D+\infty}})_*:
\pi_k(\Pol^{\Sigma}_{D+\infty})
\stackrel{\cong}{\longrightarrow}
\pi_k(\Omega \tilde{F}(U))
\leqno{(\dagger\dagger)_k}
\end{equation*}
is an isomorphism for any $k\geq 2$.
%%%
Since
$S^H=(\Omega v)\circ \widehat{j_{D+\infty}^{\p}}$ is a homotopy equivalence
and
the map $\Omega v$ induces an isomorphism on homotopy groups 
$\pi_k(\ )$ for any $k\geq 2$ (by (v) of Lemma \ref{lmm: ev}),
by  the diagram (\ref{CD: main diagram}) we see that
the map
$\widehat{j_{D+\infty}}$ induces an isomorphism
on homotopy groups $\pi_k(\ )$ for any $k\geq 2$.
This completes the proof of Theorem \ref{thm: V}.
%%%%%%%%%%
%%%%%%%%%%%%%%%%%%%%%%%%
\end{proof}
%%(End of Proof of Theorem 8.2)%%%
%%%%%%%
%%
%%
%%(Corollary 8.5)%%
\begin{crl}
%%%%%%%%%%%%
Let $\XS$ be a simply connected non-singular toric variety 
%associated 
%to the fan $\Sigma$ 
such that 
the condition (\ref{equ: homogenous}.1) is satisfied.  
\par
%%(i)%%%
$\I$
%%%%%%
The  two-fold loop map
$
\Omega^2v:\Omega^2\tilde{F}(U)/\T^r_{\R}  \stackrel{\simeq}{\longrightarrow}
\Omega^2 E^{\Sigma} (I,\partial I)
$
is a homotopy equivalence.
%%%(ii)%%%%%
\par
$\II$
%%%%%%%%%
The loop map
$\Omega i _{\Sigma}:
\Omega E^{\Sigma} (I,\partial I) \to \Omega E^{\Sigma}(D^2,S^1)$
is a universal covering projection with fiber $\Z^r$
(up to homotopy equivalence).
%%%%%%%%%%
\end{crl}
%%%%%%%%%%%
%%(End of Corollary 8.5)%%%
%%%%
%%%%
%%%%
%%%(Proof of Corollary 8.5)%%%
\begin{proof}
%%%%%%%%
(i)
 The assertion (i) follows from (v) of Lemma \ref{lmm: ev}.
 \par
 (ii)
 Since $\Omega^2i_{\Sigma}$ is a homotopy equivalence
 (by (v) of Lemma \ref{lmm: ev}),
the assertion (ii) easily follows from the following two equalities:
%%%%%%(8.12)%%
 \begin{equation}\label{pi1}
 %%%%%%%
 \pi_1(\Omega E^{\Sigma}(I,\partial I))=0,
 \quad
 \pi_1(\Omega E^{\Sigma}(D^2,S^1))=
\Z^r.
 %%%%%%%%
 \end{equation}
 %%%%%%%%
 %It remains to prove (\ref{pi1}).
 Since  $\Pol^{\Sigma}_{D+\infty}$ is simply connected
 (by (i) of Lemma \ref{lmm: ev}),
by  Theorem \ref{thm: V}  we have  an isomorphism
$
0=\pi_1(\Pol_{D+\infty}^{\Sigma})
\cong
\pi_1(\Omega E^{\Sigma} (I,\partial I))$ and
the first equality of (\ref{pi1}) holds.
It remains the second isomorphism in (\ref{pi1}).
Since there is a homotopy equivalence
$E^{\Sigma}(D^2,S^1)\simeq DJ(\KS)$
(\cite[Lemma 4.3]{KY9}),
it suffices to show that there is an isomorphism
$\pi_2(DJ(\KS))\cong \Z^r$.
Since $\XS$ is simply connected and
$\pi_2(\XS)=\Z^{r-n}$ (by (\ref{eq: pi2XS})),
it follows from the homotopy exact sequence induced from (\ref{eq: DJ})
that
 there is a short exact sequence
$
0\to \pi_2(\XS) =\Z^{r-n} \to \pi_2(DJ(\KS)) 
\stackrel{\partial}{\longrightarrow}
 \pi_1(\T^n_{\C})=\Z^n\to 0.
$
%%%
Thus, we have an isomorphism 
$\pi_2(DJ(\KS))
\cong \Z^r$.
 %%%
\end{proof}
%%(End of Corollary 8.5)%%%%%
%%
%%
%%
%%
%%
%%
%%%%%%%%%%%%%%%%%%%%
%%%%(End of section 8)%%%%%%
%%%%%%%%%%%%%%%%%%%%
%%
%%
%%
%%
%%
%%
%%%%%%%%%%%%%%%%%%%%%%%%%%%%
%%%%(SECTION 9: Proofs of the main results)%%%
%%%%(SECTION 9)%%%%%%%%%%%%%%%%%
\section{Proofs of the main results}\label{section: proofs}
%%%%%%%%%%%%%

In this section we give the proofs of the main results
(Theorem \ref{thm: I}, Corollary \ref{crl: I},  Theorem \ref{thm: II}, Corollary \ref{crl: II},
Corollary \ref{crl: I-2}, Corollary \ref{crl: II-2}).
%%
%%
%%
%%
%%%%%%%%%%%%%%%%%%%%%%%%%%%%%%%%
%%%(Proof of Theorem 2.16 and Theorem 2.20)%%
%%%%%%%%%%%%%%%%%%%%%%%%%%%%%%%%
\begin{proof}[Proofs of Theorem \ref{thm: I} and Theorem \ref{thm: II}]
%%%%%%%%%%%%%%%%%%%%%%%%%%%%%%%%
%%(Proof of Theorem 2.16)%%%
If $\sum_{k=1}^rd_k\textbf{\textit{n}}_k={\bf 0}_n$,
 Theorem \ref{thm: I} easily follows from Theorem \ref{thm: III} 
and Theorem \ref{thm: V}.
%%%(Proof of Theorem 2.20)%%%
Next assume that
$\sum_{k=1}^rd_k\textbf{\textit{n}}_k\not= {\bf 0}_n$.
It follows from the assumption (\ref{equ: homogenous}.1) that
there is an $r$-tuple
$D_*=(d_1^*,\cdots ,d_r^*)\in \N^r$ such that
$\sum_{k=1}^rd_k^*\textbf{\textit{n}}_k= {\bf 0}_n$.
If we choose a sufficiently large integer $m_0\in \N$, then 
the condition $d_k<m_0d_k^*$ holds for each $1\leq k\leq r$.
%%%%
Then note that 
$j_D:\Pol_D^*(S^1,\XS)\to \Omega \mathcal{Z}_{\KS}$ is
given by the composite of maps
%%%
$$
j_D=j_{D_0}\circ s_{D,D_0}:
\Pol_{D}^*(S^1,\XS)
\stackrel{s_{D,D_0}}{\longrightarrow}
\Pol_{D_0}^*(S^1,\XS)
\stackrel{j_{D_0}}{\longrightarrow}
\Omega \mathcal{Z}_{\KS}
$$
%%%%
as in (\ref{eq:jDDD}), where
$D_0=m_0D_*=(m_0d^*_1,m_0d^*_2,\cdots ,m_0d^*_r).$
%%%%%
Since  the maps $s_{D,D_0}$ and
$j_{D_0}$ are  homotopy equivalences through dimension
$d(D,\Sigma)$ and
$d(D_0,\Sigma)$
(by Theorem \ref{thm: IV} and Theorem \ref{thm: I}),
by using
$d(D,\Sigma)\leq d(D_0,\Sigma)$
the map $j_D$ is a homotopy equivalence through dimension
$d(D,\Sigma)$ and
Theorem \ref{thm: II}
follows.
%%%%%%%%%%%%%%%%
\end{proof}
%%%%(End of Theorem 2.20)%%%%%%
%%%
%%%
%%%
%%%
%%%
%%%
%%(Proof of Corollary 2.17)%%%%%%%%
\begin{proof}[Proof of Corollary \ref{crl: I}.]
%%%%%%%%%%%%%%%%%%%%%%%
Since $\Omega q_{\Sigma}:\Omega \mathcal{Z}_{\KS}\to \Omega \XS$
is a universal covering, 
the assertion follows from (\ref{eq: lift}) and Theorem \ref{thm: I}.
%%%
\end{proof}
%%(End of proof of Corollary 2.17)%%%%
%%
%%
%%
%%
%%
%%
%%
%%
%%%(Proof of Corollary 2.21)%%
\begin{proof}[Proofs of Corollary \ref{crl: II}] 
Consider the map
$i_D:\Pol_D^*(S^1,\XS)\to \Omega \XS$ 
defined by
the composite of maps
%%%(9.3)%%
\begin{equation}\label{eq: iD corollary}
%%%%%%%%%%
i_D=\Omega q_{\Sigma}\circ j_D,
%%%%%%%%%%
\end{equation}
%%%%%%%%%%
where $j_D:\Pol_D^*(S^1,\XS)\to \Omega \mathcal{Z}_{\KS}$ denotes the map
given by (\ref{eq:jDDD}).
Since 
$\Omega q_{\Sigma}:\Omega \mathcal{Z}_{\KS}\to \Omega \XS$ is a universal covering,
the assertion follows from Theorem \ref{thm: II}.
\end{proof}
%%%(End of proof of Corollary 2.21)%%%
%%%%%%
%%%%%%
%%%%%%
%%%(Proofs of Corollary 2.18 and Corollary 2.22)%%%
\begin{proof}[Proofs of Corollary \ref{crl: I-2} and Corollary \ref{crl: II-2}]
If $\XS$ is compact and $\Sigma (1)\subset \Sigma_1\subsetneqq \Sigma$,
the condition (\ref{equ: homogenous}.1) holds for the fan $\Sigma_1$.
Thus, Corollary \ref{crl: I-2} follows from  
Theorem \ref{thm: I} and
Corollary \ref{crl: I}.
Similarly,
Corollary \ref{crl: II-2} also follows from
Theorem \ref{thm: II} and
Corollary \ref{crl: II}.
%%%%%%%%%
\end{proof}
%%%%%%%%%
%%%(End of Proofs of Corollary 2.18 and Corollary 2.22)
%%%%%%

%%%%%%%%%%%%%%%%%%%%%%%%%%%%%%%%%%%%%%%%%%%%%%%%%%

%%()%%%%%%
\par\vspace{2mm}\par
\noindent{\bf Acknowledgements. }
%%%%%%%%
%The authors should like to take this opportunity to thank
%Professors Martin Guest and  
% Masahiro Ohno 
%for his many valuable  insights and suggestions concerning toric varieties
%and scanning maps.
%%%
%\par
The second author was supported by 
JSPS KAKENHI Grant Number  18K03295. 
This work was also supported by the Research Institute for Mathematical Sciences, 
a Joint Usage/Research Center located in Kyoto University.



%% References





\end{document}